\documentclass[a4paper,12pt]{amsart}
\usepackage[utf8]{inputenc}

\usepackage[T1]{fontenc}
\usepackage[english]{babel}
\usepackage{color}
\usepackage{amsmath} 
\usepackage{amssymb} 
\usepackage{amsthm}
\usepackage{graphicx}
\usepackage[colorlinks=true, linkcolor=blue]{hyperref}

\usepackage{geometry}

\usepackage{lmodern}
\usepackage{microtype}

\theoremstyle{plain}

\newtheorem*{thm*}{Theorem}

\newtheorem{cor}{Corollary}[section]
\newtheorem*{cor*}{Corollary}

\newcommand {\R} {\mathbb{R}} \newcommand {\Z} {\mathbb{Z}}
\newcommand {\T} {\mathbb{T}} \newcommand {\N} {\mathbb{N}}

\newcommand {\p} {\partial}
\newcommand {\dt} {\partial_t}
\newcommand {\sgn} {\text{sgn}}

\newcommand{\bbC}{\mathbb{C}}

\newcommand{\bbR}{\mathbb{R}}

\newcommand{\bbT}{\mathbb{T}}

\newcommand{\bbZ}{\mathbb{Z}}

\DeclareMathOperator{\dive }{div}

\DeclareMathOperator{\supp }{supp}

\theoremstyle{plain}
\newtheorem{theorem}{Theorem}
\newtheorem{lemma}[theorem]{Lemma}
\newtheorem{pro}[theorem]{Proposition}

\theoremstyle{definition}
\newtheorem{definition}[theorem]{Definition}

\title[On Echoes in MHD]{On Echoes in Magnetohydrodynamics with Magnetic Dissipation}
\author{Niklas Knobel}
\address{Karlsruhe Institute of Technology\\
  Englerstraße 2\\
76131 Karlsruhe}
\author{Christian Zillinger}

\date{\today}

\begin{document}

\begin{abstract}
    We study the long time asymptotic behavior of the inviscid magnetohydrodynamic equations with magnetic dissipation near a combination of Couette flow and a constant magnetic field.
    Here we show that there exist nearby explicit global in time low frequency solutions, which we call waves. Moreover, the linearized problem around these waves exhibits resonances under high frequency perturbations, called echoes, which result in norm inflation Gevrey regularity and infinite time blow-up in Sobolev regularity.
\end{abstract}
\maketitle
\setcounter{tocdepth}{1 }
\tableofcontents

\section{Introduction and Main Results}
\label{sec:intro}
In this article we consider the two-dimensional magnetohydrodynamic (MHD) equations with
magnetic resistivity $\kappa>0$ but without viscosity
\begin{align}
  \label{eq:MHD}
  \begin{split}
    \partial_t V + V \cdot \nabla V + \nabla p &= B \cdot \nabla B, \\
    \partial_t B + V \cdot \nabla B &= \kappa \Delta B+ B \cdot \nabla V, \\
    \dive(B)= \dive(V)&=0,\\
    (t,x,y) &\in \bbR_{+}\times \bbT \times \bbR,
  \end{split}
\end{align}
near the stationary solution
\begin{align}
  \label{eq:groundstate}
  \begin{split}
    V(t,x,y)&=(y,0), \\
    B(t,x,y)&=(\alpha, 0).
  \end{split}
\end{align}
The MHD equations are a common model of the evolution of conducting fluids
interacting with (electro-)magnetic fields in regimes where the magnetization of
the fluid can be neglected.
They describe the evolution of the fluid in terms of the fluid velocity $V$,
pressure $p$ and magnetic field $B$. The constant mass and charge densities are normalized to $1$. 
Here particular examples of applications range from the modeling of solar dynamics to geomagnetism and the earths molten core to using liquid metals in industrial applications or in fusion applications \cite{davidson_2016}.

A main aim of this article is to analyze the long-time asymptotic behavior of
solutions to this coupled system and, in particular, the interaction of
instabilities, partial dissipation and the system structure of the equations.
Here we note that due to the affine structure of the stationary solution \eqref{eq:groundstate},
the corresponding linearized problem around this solution decouples in Fourier
space and can be shown to be stable in arbitrary Sobolev (or even analytic)
regularity, as we prove in Section \ref{sec:setting}.
\begin{lemma}
\label{lemma:stabilitytrivial}
Let $\alpha\in \bbR$ be given and consider the linear problem
\begin{align*}
    \partial_t V + y\p_x V + (V_2,0)&= \alpha \p_x B, \\
    \partial_t B + y\p_x B -(B_2,0)&= \kappa \Delta B+ \alpha \p_x V, \\
    \dive(B)= \dive(V)&=0,\\
    (t,x) &\in \bbR_{+}\times \bbT \times \bbR.
\end{align*}
Then these equations are stable in $H^s$ for any $s \in \R$ in the sense that there exists a constant $C>0$ such that for any choice of initial data and all times $t>0$ it holds that
\begin{align*}
    \|(\nabla^\perp \cdot V)(t, x-ty,y) \|_{H^s}^2 + \|(\nabla^\perp \cdot B)(t, x-ty, y)\|_{H^{s}}^2 
    \\
    \leq (1+ \kappa^{-2/3})^2 (\|\nabla^\perp \cdot V|_{t=0}\|_{H^s}^2 + \|\nabla^\perp \cdot B|_{t=0}\|_{H^{s}}^2). 
\end{align*}
\end{lemma}
Here $\nabla^\perp \cdot V =: W$ is the vorticity of the fluid and $\nabla^\perp
\cdot B=:J$ is the (magnetically induced) current.

In contrast to this to this very strong linear stability result, the stability results for the inviscid nonlinear equations are expected to crucially rely on very high, Gevrey regularity (see Section \ref{sec:toy} for a definition).
More precisely, similarly to the nonlinear Euler equations
\cite{dengmasmoudi2018,dengZ2019,zillinger2021echo} or Vlasov-Poisson equations
\cite{bedrossian2016nonlinear,zillinger2020landau, Villani_long} the nonlinear equations are not a priori expected to not remain close to the linear dynamics due to ``resonances'' or ``echoes''\cite{malmberg1968plasma,yu2005fluid}, which may lead to unbounded norm inflation of any Sobolev norm.
It is the main aim of this article to identify and capture this resonance mechanism for the resistive MHD equations. In particular, we ask to which extent magnetic dissipation can stabilize the dynamics.
As we discuss in Section \ref{sec:toy} the main nonlinear resonance mechanism is expected to be given by the repeated interaction of a high frequency perturbation with an underlying low frequency perturbation of \eqref{eq:groundstate}.
In this article we thus explicitly construct such low frequency nonlinear
solutions, called \emph{traveling waves} (a combination of an Alfv\'en waves and shear dynamics;
see Section \ref{sec:setting} and Lemma \ref{lemma:twave} for further discussion).
\begin{lemma}
  \label{lemma:waves}
  Let $\kappa > 0$ and $\alpha \in \R$ and let $(f_0, g_0) \in \R^2$.Then there exist smooth global in time solutions of the nonlinear, resistive MHD equations \eqref{eq:MHD}, which are of the form 
  \begin{align*}
    V(t,x,y) &= (y,0) + \frac{f(t)}{1+t^2} \nabla^\perp \sin(x-ty), \\
    B(t,x,y) & = ( \alpha,0) + \frac{g(t)}{1+t^2} \nabla^\perp \sin(x-ty),
  \end{align*}
  with $(f(0), g(0))=(f_0, g_0)$.
  Furthermore, for a suitable choice of $f_0, g_0$ it holds that
  \begin{align*}
    f(t)&\rightarrow 2c,\\
    g(t)&\rightarrow 0,
  \end{align*}
  as $t\rightarrow \infty$.
\end{lemma}
In view of the underlying shear dynamics it is natural to change to coordinates 
\begin{align*}
    (x-ty,y).
\end{align*}
In these coordinates the corresponding vorticity $W=\nabla^\perp \cdot V$ and current $J= \nabla^\perp \cdot B$ read 
\begin{align*}
    W&= -1 + f(t) \cos(x), \\
    J&= 0 -g(t) \sin(x).
\end{align*}
Unlike the stationary solution \eqref{eq:groundstate} these waves have a non-trivial $x$-dependence. As we discuss in Section \ref{sec:toy} this $x$-dependence allows resonances to propagate in frequency and underlies the nonlinear instability of the stationary solution \eqref{eq:groundstate}.
More precisely, we show that the (simplified) \emph{linearized} equations around these waves exhibit the above mentioned \emph{nonlinear} resonance mechanism (in terms of both upper and lower bounds on solutions).
In particular, we aim to obtain a precise understanding of the dependence of the resonance mechanism on the resistivity $\kappa>0$ and the frequency-localization of the initial perturbation.
The research on well-posedness and asymptotic behavior of the magnetohydrodynamic equations is a very active field of research and we in particular mention the recent work \cite{liss2020sobolev}, which considers a related, fully dissipative setting in 3D, as well as the articles \cite{jiang2022nonlinear,zhai2022stability,wu2021global,boardman2020stabilization,feireisl2019global,du2019exponential,he2018global,liu2018physical,wei2017global}.
More precisely, in \cite{liss2020sobolev} Liss studied the nonlinear, fully dissipative,
three-dimensional MHD equations around the same stationary solution
\eqref{eq:groundstate} in a doubly-periodic three-dimensional channel $\T \times
\R \times \T$ and established bounds on the Sobolev stability threshold as
$\nu=\kappa \downarrow 0$. In contrast, this article considers the 2D setting
with partial dissipation $\nu=0$, $\kappa >0$ in Gevrey regularity.
Similar questions on the stability of systems with partial dissipation in critical spaces are also a subject of active research in other (fluid) systems, such as the Boussinesq equations \cite{cao2013global,elgindi2015sharp,doering2018long}.

For simplicity of presentation and to simplify the analysis in this article we modify the linearized equations for the vorticity and current perturbations $w,j$
\begin{align}
\label{eq:linsys}
\begin{split}
   \partial_t w  &=  \alpha \partial_x j  -(2c \sin(x)\partial_y\Delta_t^{-1} w)_{\neq} \\
    \partial_t j  &   =\kappa \Delta_t j   +\alpha \partial_x w - 2 \partial_x \partial_y^t \Delta^{-1}_t j,\\
    \Delta_t &= \p_x^2 + (\p_y-t\p_x)^2, 
\end{split}
\end{align}
and fix the $x$-averages of $w$ and $j$, which also fixes the underlying shear flow. Here, for simplicity we have also replaced $f(t), g(t)$ by $2c$ and $0$, respectively.
In analogy to other fluid systems \cite{bedrossian21,bedrossian2013inviscid}, a similar structure of the equations can be achieved by considering the coordinates
\begin{align*}
    (x- \int_0^t \int V_1 dx dt, \frac{1}{t}\int_0^t \int V_1 dx dt) =: (X,Y),
\end{align*}
which however makes estimates of $\Delta_t^{-1}$ technically more involved and less transparent \cite{Zill3}. 
In the interest of a clear presentation of the resonance mechanism we hence instead fix $Y=y$ by a small forcing.

\begin{theorem}
  \label{thm:main}
Let $0<\alpha<10$ and $0<\kappa<1$ with $\beta:=\frac{\kappa}{\alpha^2}$ and $c\le \min (10^{-3 } \beta^{\frac {16} 3 },10^{-4} ) $ be
given. Consider the (simplified) linearized equations \eqref{eq:linsys} around the wave of Lemma \ref{lemma:waves}.

Then there exists a constant $C$ such that for any initial data $w_0,
j_0$ whose Fourier transform satisfies
\begin{align*}
  \sum_k \int \exp(C\sqrt{|\xi|}) (|\mathcal{F}w_0(k,\xi)|^2 + |\mathcal{F}j_0(k,\xi)|^2) d\xi < \infty
\end{align*}
the corresponding solution stays regular for all times up to a loss of constant
in the sense that for all $t>0$ it holds that
\begin{align*}
  \sum_k \int \exp(\tfrac{C}{2}\sqrt{|\xi|}) (|\mathcal{F}w(t,k,\xi)|^2 + |\mathcal{F}j(t,k,\xi)|^2) d\xi < \infty.
\end{align*}
Moreover, there exists initial data $w_0, j_0$ and $0< C^*<C$ such that
\begin{align*}
  \sum_k \int \exp(C^*\sqrt{|\xi|}) (|\mathcal{F}w_0(k,\xi)|^2 + |\mathcal{F}j_0(k,\xi)|^2) d\xi < \infty,
\end{align*}
but so that the corresponding solution $w, j$ grows unbounded in Sobolev
regularity as $t\rightarrow \infty$.
\end{theorem}
Let us comment on these results:
\begin{itemize}
\item As we discuss in Section \ref{sec:toy} the linearized equations
  around a traveling wave closely resemble the interaction of high and
  low-frequency perturbations in the nonlinear equations. These equations thus
  are intended to serve as slightly simplified model of the nonlinear resonance
  mechanism. We remark that in the full nonlinear problem the $x$-averages and
  hence the underlying shear dynamics change with time and the corresponding
  change of coordinates has to be controlled. For simplicity and clarity the
  present model instead fixes this change of coordinates.
\item The Fourier integrability with a weight $\exp(C\sqrt{|\xi|})$ corresponds
  to Gevrey $2$ regularity with respect to $y$. For simplicity of presentation the above results are stated with $L^2(\T)$ regularity in $x$. All results
  also extend to more general Fourier-weighted spaces, such as $H^N$ for any
  $N\in \N$ or suitable Gevrey or analytic spaces (see Definitions
  \ref{defi:gevrey} and \ref{defi:X}).
\item The stability and norm inflation in Gevrey $2$ regularity matches the
  regularity classes of the (nonlinear) Euler equations. In particular, the
  magnetic field and magnetic dissipation are shown to not be strong enough to
  suppress this growth. We remark that our choice of coupling between the size of the
  magnetic field and magnetic dissipation is made so that both effects are ``of
  the same magnitude'' and hence their interaction plays a more crucial role
  (see Section \ref{sec:betainf} for a discussion).
\item These results complement the work of Liss \cite{liss2020sobolev} on the
  Sobolev stability threshold in $3D$ with full dissipation. Indeed, the above derived upper and lower
  bounds establish Gevrey $2$ as the optimal regularity class of the linearized
  problem in $2D$ with partial dissipation. We expect that as for the Euler \cite{bedrossian2015inviscid} or Vlasov-Poisson equations \cite{bedrossian2016landau} nonlinear stability results match
  the regularity classes of the linearized problem around appropriate traveling waves. 
\end{itemize}
We further point out that the instability result of Theorem \ref{thm:main} also implies a norm inflation result for the nonlinear problem around each wave in slightly different spaces (see Corollary \ref{cor:nl}). 
In particular in any arbitrarily small analytic neighborhood around the stationary solution \eqref{eq:groundstate} there exist nonlinearly unstable solutions (with respect to lower than Gevrey $2$ regularity) . 

The remainder of the article is structured as follows:
\begin{itemize}
    \item In Section \ref{sec:setting} we discuss the linearized problem around the stationary state \eqref{eq:groundstate} and introduce waves as low-frequency solutions of the nonlinear problem.
    \item In Section \ref{sec:toy} we discuss the resonance mechanism for a toy model. In particular, we discuss optimal spaces for norm inflation and (in)stability results as well as the time- and frequency-dependence of resonances.
    \item The main results of this article are contained in Section \ref{Ik}, where we establish upper bounds and lower bounds on the norm inflation.
    \item The Appendix \ref{LApp} contains some auxiliary estimates of a growth factor in Section \ref{Ik}. In the second Appendix \ref{sec:norminflation} we prove a nonlinear instability result for the traveling waves. 
\end{itemize}

\section{Linear Stability, Traveling Waves and Echo Chains}
\label{sec:setting}
In this section we establish the linear stability of the resistive MHD equations \eqref{eq:MHD}
\begin{align}
\begin{split}
\partial_t V + (V\cdot \nabla )V + \nabla p &= (B \cdot \nabla) B, \\
\partial_t B + ( V\cdot \nabla ) B &= \kappa \Delta B + (B \cdot \nabla) V, \\
\nabla\cdot B = \nabla \cdot V &=0,\label{idmhd}
\end{split}
\end{align}
around the stationary solution \eqref{eq:groundstate} as stated in Lemma
\ref{lemma:stabilitytrivial}.
Furthermore, we sketch the nonlinear resonance mechanism underlying the norm
inflation result of Theorem \ref{thm:main}, which is given by the repeated
interaction of high and low frequency perturbations.
This mechanism motivates the construction of the traveling wave solutions of Lemma
\ref{lemma:twave} and the corresponding (simplified) linearized equations around these waves,
which are studied in the remainder of the article.

In order to simplify notation we may restate the MHD equations with respect to
other unknowns. That is, since we consider vector fields in two dimensions and $V$ and $B$ are divergence-free, we may introduce the magnetic potential $\Phi$, magnetic current $J$  and fluid vorticity $W$ by 
\begin{align*}
    J &= \nabla^\perp \cdot B, \\
    \Delta \Phi &= J, \\
    W &= \nabla^\perp \cdot V.
\end{align*}
Under suitable decay assumptions (or asymptotics) in infinity the equations can then equivalently be expressed as 
\begin{align}
\begin{split}
\partial_t W + (V\cdot \nabla )W  &= (B\cdot \nabla ) \Delta \Phi,  \\
\partial_t \Phi + ( V\cdot \nabla )\Phi & =\kappa \Delta \Phi, \label{wphi}
\end{split}
\end{align}
or in terms of $J$:

\begin{align}
\label{eq:wpsi}
\begin{split}
\partial_t W + (V\cdot \nabla )W  &=  (B\cdot \nabla )J , \\
\partial_t J + ( V\cdot \nabla )J & =\kappa \Delta J+( B\cdot \nabla )W- 2( \partial_i V\cdot \nabla )\partial_i\Phi  .
\end{split}
\end{align}
With these formulations we are now ready to establish the linear stability of
the stationary solution \eqref{eq:groundstate}.

\begin{proof}[Proof of Lemma \ref{lemma:stabilitytrivial}]
Consider the formulation of the MHD equations as \eqref{wphi}, then the linearization around $V=(y,0), W=-1, B=(\alpha, 0), \Phi=\alpha y$ is
given by
\begin{align*}
    \dt W + y \p_x W &= \alpha \p_x \Delta \Phi, \\
    \dt \Phi + y \p_x \Phi + V_2 \alpha &= \kappa \Delta \Phi.
\end{align*}
We note that all operators other than $y\p_x$ are constant coefficient Fourier
multipliers.
Hence we apply a change of variables
\begin{align*}
  (x,y) \mapsto (x-ty, y)
\end{align*}
to remove this transport term and obtain 
\begin{align*}
    \dt w &= \alpha \p_x \Delta_t \phi , \\
    \dt \phi &= - \alpha \p_x \Delta_t^{-1} w + \kappa \Delta_t \phi,
\end{align*}
where $w, \phi$ denote the unknowns with respect to these variables and
$\Delta_t=\p_x^2+(\p_y-t\p_x)^2$.
We note that this system decouples in Fourier space and for simplicity of
notation express it in terms of the (Fourier transform of the) current $j=\Delta_t \phi$:
\begin{align*}
    \dt w &= ik \alpha j , \\
    \dt j &= \frac{2k(kt-\xi)}{k^2+(\xi-kt)^2} j - \kappa (k^2+(\xi-kt)^2) j + ik \alpha w,
\end{align*}
where $k \in Z$ and $\xi \in \R$ denote the Fourier variables with respect to
$x\in \T$ and $y \in \R$, respectively.
Here and in the following, with slight abuse of notation, we reuse $w$ and $j$ to refer to the Fourier transforms of the vorticity and current perturbation.
For $k=0$ these equations are trivial and we hence in the following we may
assume without loss of generality that $k\neq 0$.
Furthermore, we note that the right-hand-side depends on $\xi $ only in terms of
$\frac{\xi }{k}-t$.
Hence, by shifting time we may further assume that $\xi=0$.

With this reduction we first note that by anti-symmetry for all $\alpha \in \R$
it holds that 
\begin{align*}
 \dt (|w|^2 + |j|^2)/2 = (\tfrac{2t}{1+t^2} - \kappa k^2 (1+t^2)) |j|^2.
\end{align*}
We make a few observations:
\begin{itemize}
    \item If $\kappa k^2 \geq 1$ the horizontal dissipation is sufficiently strong to absorb growth for all times.
    \item If $\kappa k^2 \leq 1$  is small, then for sufficiently large times $|t|\geq (k^2\kappa)^{-1/3}$ the right-hand-side is non-positive. 
    \item  It thus only remains to estimate the growth on the time interval $|t|\leq (k^2\kappa)^{-1/3}$, where 
    \[  \dt (|w|^2 + |j|^2)  \leq \frac{4(t)_{+}}{1+t^2} (|w|^2 + |j|^2).\]
\end{itemize}
The latter case can be bounded by an application of Gronwall's lemma and after shifting back in time it yields 
\begin{align*}
    |w(t)|^2 + |j(t)|^2 \leq (1+(k^2\kappa)^{-2/3} )^2 ( |w(0)|^2 + |j(0)|^2) 
\end{align*}
for all $t>0$.


\end{proof}
While the ground state is thus linearly stable in arbitrary Sobolev or even
analytic regularity, nonlinear stability poses to be a much more subtle question
with stronger regularity requirements.

\subsection{Wave-type Perturbations}
In order to investigate the stability of the MHD equations, it is a common approach
to consider wave-type perturbation. Here a classical result considers
perturbations around a constant magnetic field and a vanishing velocity field.
\begin{lemma}[(c.f. {\cite{alfven1942existence,davidson_2016}})]
  \label{lemma:Alfven}
  Consider the ideal MHD equations (i.e. $\kappa=0$) in three dimensions linearized around a
  constant magnetic field $B=B_0 e_z$ and vanishing velocity field $V=0$.
  Then a particular solution is of the form
  \begin{align*}
    B= (B_1(t,z),0,0), V=(V_1(t,z),0,0)
  \end{align*}
  where $B_1$ and $V_1$ are solutions of the wave equation
  \begin{align*}
    \dt^2 B_1 - B_0^2 \p_z^2 B_1=0,\\
    \dt^2 V_1 - B_0^2 \p_z^2 V_1=0,
  \end{align*}
\end{lemma}
The linearized problem thus admits wave-type solutions propagating in the
direction $e_z$ of the constant magnetic field and pointing into an orthogonal
direction.
These solutions are known as \emph{Alfv\'en waves} \cite{alfven1942existence}.

\begin{proof}[Proof of Lemma \ref{lemma:Alfven}]
  We make the ansatz that $B$ and $v$ only depend on $t$ and $z$ and express the
  linearized equations in terms of the current $J=\nabla \times B$ and vorticity
  $W=\nabla \times V$. Then the equations reduce to
  \begin{align*}
    \dt J = B_0 \p_z W, \\
    \dt W = B_0 \p_z J.
  \end{align*}
  These equations are satisfied if both $J$ and $W$ solve a
  wave equation and are chosen compatibly.
More precisely, two linearly independent solutions are given by 
\begin{align*}
    W=f(z+B_0t)=J
\end{align*}
and 
\begin{align*}
   W=g(z-B_0t)=-J,
\end{align*}
where $f$ and $g$ are arbitrary smooth function.
  
  We remark that since $B=B(t,z)$ and $V=V(t,z)$ point into a direction orthogonal to
  the $z$-axis, they are divergence-free for all times.
  Finally, since both functions are independent of $x$ it follows that all
  nonlinearities $V \cdot \nabla V, B \cdot \nabla V, V \cdot \nabla B, B \cdot
  \nabla B$
  identically vanish, so these are also nonlinear solutions. 
\end{proof}

In the following we consider the two-dimensional setting and extend this construction to also include an underlying
affine shear flow. We call the resulting solutions \emph{traveling waves} in
analogy to dispersive equations and related constructions for fluids and plasmas
\cite{dengZ2019,bedrossian2016nonlinear,zillinger2020landau,zillinger2021echo,dengmasmoudi2018}.
As we sketch in Section \ref{sec:toy} the non-trivial $x$-dependence of these
waves will allow us to capture the main nonlinear norm inflation mechanism in
the linearized equations around these waves (as opposed to linearizing around
the stationary solutions \eqref{eq:groundstate}).

\begin{lemma}
\label{lemma:twave}
Let $\alpha \in \R$ and $\kappa\geq 0$ be given. Then for any choice of parameters $(f(0),g(0))\in \R^2$ there exists a solution of \eqref{eq:wpsi} of the form 
\begin{align}
\begin{split}
  W   &= -1 + f(t)\cos(x-yt)\\
  J &=-g(t) \sin(x-yt).
\end{split}
\end{align}
We call such a solution a \emph{traveling wave}.
\end{lemma}
We remark that this construction also allows for general profiles $h(t,x-ty)$ in place of $\cos(x-ty)$.
This particular choice is made so that for $f(0)$ and $g(0)$ small, such a wave is an initially small, analytic perturbation of the stationary solution \eqref{eq:groundstate} and localized at low frequency.

\begin{proof}[Proof of Lemma \ref{lemma:twave}]
For easier reference we note that for this ansatz, we obtain
\begin{align*}
    V&=(y,0) + \tfrac{f(t)}{1+t^2} \sin(x-yt) (t,1)\\
        W   &= -1 + f(t)\cos(x-yt)\\
        B&=(\alpha ,0 ) +\tfrac {g(t)}{1+t^2}\cos(x-yt) (t,1) \\
        J&=-g(t) \sin(x-yt)\\
        \Phi &=\alpha y + \tfrac {g(t)}{1+t^2} \sin(x-yt).
\end{align*}
Inserting this into the equation \eqref{eq:wpsi} the nonlinearities vanish due to the one-dimensional structure of the waves. Therefore, this ansatz yields a solution if and only if $f$ and $g$ solve the ODE system
\begin{align}
\label{eq:odesystem}
\begin{split}
    f'(t)&= - \alpha g(t),\\
    g'(t)&= - \kappa (1+t^2) g(t)+\alpha f(t)+  \tfrac {2t}{1+t^2}g(t). 
\end{split}
\end{align}
Thus by classical ODE theory for any choice of initial data there indeed exists a unique traveling wave solution.
\end{proof}

Given such a traveling wave we are interested in its behavior, in particular for large times, and how it depends on the choices of $\kappa$ and $\alpha$.

\begin{lemma}
\label{lemma:ode}
Let $\alpha>0$ and $\kappa>0$ then for any choice of initial data the solutions $f(t), g(t)$ of the ODE system \eqref{eq:odesystem}
\begin{align}\label{fg_est}
\begin{split}
    f'(t)&= - \alpha g(t),\\
    g'(t)&= - \kappa (1+t^2) g(t)+\alpha f(t)+  \tfrac {2t}{1+t^2}g(t), 
\end{split}
\end{align}
satisfy the following estimates:
\begin{align}
    \begin{cases}
      |f(t)|^2 + |g(t)|^2 \leq (1+t^2)^2 (|f(0)|^2+ |g(0)|^2) &\text{ if } 0<t<\kappa^{-1/3} \\
     |f(t)|^2 + |g(t)|^2\leq |f(\kappa^{-1/3})|^2+ |g(\kappa^{-1/3})|^2 & \text{ if } t> \kappa^{-1/3}.
    \end{cases}
\end{align}
Furthermore, for a specific choice of initial data it holds 
\begin{align*}
  |f(t)- \epsilon| &\leq \frac{1}{2} \epsilon, 
\end{align*}
for all $ t\ge 4\beta^{-1}$ and 
\begin{align*}
    \vert g(t)\vert \to 0 
\end{align*}
as $t\to \infty$. 
\end{lemma}

\begin{proof}[Proof of Lemma \ref{lemma:ode}]
We first observe that by anti-symmetry of the coefficients it holds that 
\begin{align*}
    \dt (|f|^2+ |g|^2) = 2|g|^2 \left( - \kappa (1+t^2) +\tfrac {2t}{1+t^2}\right).
\end{align*}
In particular, for $t>\kappa^{-1/3}$ the last factor is negative and hence $|f|^2+|g^2|$ is non-increasing.
For times smaller than this, we may derive a first rough bound from the estimate
\begin{align*}
    \dt (|f|^2+ |g|^2) \le (|f|^2+|g|^2) \tfrac {4t}{1+t^2},
\end{align*}
which yields an algebraic lower and upper bound growth bound. We next turn to the case of special data, due to lower and upper norm bounds  \eqref{fg_est} is time reversible. Therefore, we can obtain 
\begin{align*}
    f(t_0)=1, \qquad g(t_0)=0
\end{align*}
for $t_0=4\beta^{-1}$. Then  we deduce 
\begin{align*}
    g(t) &=\alpha \int_{t_0}^t d\tau  \ \exp(-\kappa (t-\tau +\tfrac 1 3 (t^3-\tau^3))) \tfrac {1+t^2}  {1+\tau^2}f(\tau ) 
\end{align*}
and thus
\begin{align*}
    f(t) &= 1-\alpha \int_{t_0}^t d\tau_1  g(\tau_1)\\
    &= 1-\tfrac \kappa \beta  \int_{t_0}^t d\tau_1  (1+\tau_1^2)\int_{t_0}^{\tau_1 }  d\tau_2   \ \exp(-\kappa (\tau_1 -\tau_2 +\tfrac 1 3 (\tau_1^3-\tau^3_2))) \tfrac 1 {1+\tau^2_2}f(\tau_2 )\\
    &= 1 -\tfrac 1 \beta \int^t_{t_0} d\tau_2  \tfrac 1 {1+\tau^2_2}f(\tau_2 )(1-\exp(-\kappa(t-\tau_2 +\tfrac 1 3 (t^3-\tau_2^3)))).
\end{align*}
This gives the estimate
\begin{align}
    0\le 1-f(t) &\le \tfrac 2{\beta t_0 }\label{fest},
\end{align}
which implies that after time $t_0$  the value of $f$ satisfies the same bound. Similarly, for $g$ we recall that 
\begin{align*}
    \partial_s g &= (\tfrac {2t}{1+t^2} - \kappa (1+t^2))   g(t)+\alpha f(t)
\end{align*}
and hence for $t_1 = 2\kappa^{-\frac 13 }$ it holds that
\begin{align*}
    g(t_1) &\le \alpha \tfrac {t^2_1}{t_0^2} .
\end{align*}
Furthermore, this implies that for times $t\ge t_1$ it holds that
\begin{align*}
    g(t) &\le  \alpha \tfrac {t^2_1}{t_0^2} \exp(-\tfrac \kappa 3 t^2 (t-t_1 ) ) + \alpha \int_{t_1}^t \exp(-\tfrac \kappa 3 (1+t^2) (t-\tau ))\\
    &\le \alpha \tfrac {t^2_1}{t_0^2} \exp(-\tfrac \kappa 3t^2 (t-t_1 ) ) + \tfrac {3\alpha}  {\tfrac \kappa 3(1+t) ^2 } .
\end{align*}
Finally, for times  $t\gg \kappa^{-\frac 13 }$ we may estimate  
\begin{align}
    g&\le  \tfrac {4\alpha} {\kappa t ^2 } .  \label{gest} 
\end{align}

\end{proof}

\subsection{Paraproducts and an Echo Model}
\label{sec:toy}
As mentioned in Section \ref{sec:intro} the main mechanism for nonlinear instability is expected to be given by the repeated interaction of high- and low-frequency perturbations of the stationary solution \eqref{eq:groundstate}.
In the following we introduce a model highlighting the role of the traveling waves and discuss what stability and norm inflation estimates can be expected.

For this purpose we note that the nonlinear MHD equations \eqref{wphi} for the perturbations $w, \phi$ of the groundstate \eqref{eq:groundstate} in coordinates $(x-ty,y)$ can be expressed as 
\begin{align}
\label{eq:nlsystem}
\begin{split}
    \dt w + \nabla^{\perp} \Delta_t^{-1} w \cdot \nabla w &=\alpha \partial_x \Delta \phi + \nabla^\perp \phi \cdot \nabla \Delta_t \phi, \\
    \dt \phi + \nabla^{\perp} \Delta_t^{-1} w \cdot \nabla \phi &=\alpha \partial_x \Delta^{-1} w +  \kappa \Delta_t \phi, \\
    \Delta_t &= \p_x^2 + (\p_y-t\p_x)^2,
\end{split}
\end{align}
where we used cancellation properties of $\nabla^\perp \cdot \nabla$.
The stability result of Lemma \ref{lemma:stabilitytrivial} considered the linearized problem around the trivial solution $(0,0)$, which removes all effects of the nonlinearities.
In order to incorporate these effects into our model we thus consider the nonlinear equations as a coupled system for the low frequency part of the solution $(w_{low}, \phi_{low})$ (defined as the Littlewood-Payley projection to frequencies $<N/2$ for some dyadic scale $N$) and the high frequency part $(w_{hi}, \phi_{hi})$.
If we for the moment consider the low frequency part as given then the action of the nonlinearities on the high frequency perturbation of the vorticity can be decomposed as
\begin{align}
\label{eq:para}
    \nabla^{\perp} \Delta_t^{-1} w_{low} \cdot \nabla w_{hi} + \nabla^{\perp} \Delta_t^{-1} w_{hi} \cdot \nabla w_{low} + \nabla^{\perp} \Delta_t^{-1} w_{hi} \cdot \nabla w_{hi}.
\end{align}
Here the first term is of transport type and hence unitary in $L^2$ and we expect $\nabla^{\perp} \Delta_t^{-1} w_{low}$ to decay sufficiently quickly in time that this term should not yield a large contribution to possible norm inflation.
Similarly for the last term we note that both factors are at comparable frequencies and that we by assumption consider a small high frequency perturbation  and thus this term is also not expected to have a large impact on the evolution.
The main norm inflation mechanism thus is expected to be given by the high frequency velocity perturbation interacting with a non-trivial low frequency vorticity perturbation.

In order to build our toy model we thus focus on this part and formally replace $w_{low}, \phi_{low}$ by the traveling waves, which are solutions of the nonlinear problem.
Furthermore,  as a simplification by a similar reasoning as above we also fix the underlying shear flow for our model.
Then the equations for the (high frequency part of the) current perturbation 

$j=\Delta \phi $ and vorticity perturbation $w$ read
\begin{align}
\begin{split}\label{finEq}
    \partial_t w &=  \alpha \partial_x j  -(2c \sin(x)\partial_y\Delta_t^{-1} w)_{\neq} \\
    \partial_t j  &=\kappa \Delta_t j_{\neq}   +\alpha \partial_x w - 2 \partial_x \partial_y^t \Delta^{-1}_t j,
\end{split}
\end{align}
where we also simplified to $f(t)=2c, g(t)=0$.

We note that compared to the linearization around the stationary solution these equations break the decoupling in Fourier space. Indeed taking a Fourier transform and relabeling $j \mapsto - i j$ we arrive at
\begin{align}
\label{wpsi1}
\begin{split}
    \partial_t w(k)  &=  -\alpha k j(k) -c \tfrac \xi {(k+1)^2 }\tfrac 1 {1+( \frac \xi {k+1}-t )^2} w(k+1)+c \tfrac \xi {(k-1)^2 }\tfrac 1 {1+( \frac \xi {k-1}-t )^2} w(k-1),\\
    \partial_t j(k) &   =(2 \tfrac {t-\frac \xi k  }{1+(\frac \xi k -t )^2 } -\kappa k^2 (1+ (\tfrac \xi k -t )^2 ) ) j(k)  +\alpha k w(k).
\end{split}
\end{align}
Furthermore, if $t \approx \frac{\xi}{k}$ then the additional term is of size $c \tfrac \xi {(k)^2 }$ and hence can possibly lead to a very large change of the dynamics.
In reference to the experimental results mentioned in Section \ref{sec:intro} we can interpret this as the low frequency and high frequency perturbation resulting in an ``echo'' around the time $t \approx \frac{\xi}{k}$. For the following toy model we neglect all modes except those at frequency $k$ and $k-1$ and only include the action of the resonant mode $k$ on the non-resonant mode $k-1$.

\begin{lemma}[Toy model]
\label{lem:toy}
    Let $c,\kappa ,  \alpha$ be as in Theorem \ref{thm:main} such that $\beta =\tfrac \kappa {\alpha^2 }\ge  \pi$ and consider the Fourier variables   $k\geq 2$ and $\xi\ge 10 \max( \kappa^{-1} , \tfrac {k^2 } c ) $. Then for
    \begin{align*}
        t_{k}:=\frac{1}{2} (\frac{\xi}{k+1} + \frac{\xi}{k}) < t < \frac{1}{2} (\frac{\xi}{k} + \frac{\xi}{k-1})=: t_{k-1} 
    \end{align*}
    we consider the toy model 
    \begin{align}
    \begin{split}
        \dt w(k) &= -\alpha k j (k) ,\\
        \partial_t j(k)    &=(2 \tfrac {t-\frac \xi k  }{1+(\frac \xi k -t )^2 } -\kappa k^2 (1+ (\tfrac \xi k -t )^2 ) ) j(k)  +\alpha k w(k), \\
        \dt w(k-1) &= - \alpha (k-1) j(k-1)+ c \tfrac \xi {k^2 }\tfrac 1 {1+( \frac \xi {k}-t )^2} w(k),\\
        \dt j(k-1) &= - \kappa \frac{\xi^2}{k^2} j (k-1) + \alpha (k-1) w(k-1).           
    \end{split}
    \end{align}
    Then for initial data $w(k,t_k)=1$ and $w(k-1,t_k)=j(k,t_k)=j(k-1,t_k)=0$ we estimate
    \begin{align*}
        & \quad (|w(k)|+ |w(k-1)|+ \alpha k |j(k)|+\alpha (k-1) |j(k-1)|)|_{t=t_{k-1}} 
        \leq 2 \pi c \frac{\xi}{k^2} .
    \end{align*}
    Furthermore, this bound is attained up to a loss of constant in the sense that 
    \begin{align*}
        |w(k-1,t_{k-1})| \geq \tfrac \pi 2 c \frac{\xi}{k^2} . 
    \end{align*}
\end{lemma}

\begin{proof}[Proof of Lemma \ref{lem:toy}]
    We perform a shift in time such that $t=\tfrac \xi k +s $ and thus $s_0:=-\tfrac \xi 2 \tfrac1  {k^2+k}\le s \le \tfrac \xi 2 \tfrac1  {k^2-k} =:s_1$. Integrating the equations in time, for our choice of initial data we obtain that
\begin{align*}
        j(s,k ) &= \alpha k \int_{s_0}^t   \tfrac {1+s^2}{1+\tau_2^2} \exp(-\kappa k^2 (s-\tau_2+\tfrac 1 3 (s^3 -\tau^3_2))) w(k,\tau_2) d\tau_2
\end{align*}
and thus 
\begin{align*}
    w(k,s)&= 1-\alpha k \int  j(\tau_1,k )\ d\tau_1 \\
    &= 1-\alpha^2  k^2 \int_{s_0}^{s} \int_{s_0}^{\tau_1 }  \tfrac {1+\tau_1^2}{1+\tau_2^2} \exp(-\kappa k^2 (\tau_1-\tau_2+\tfrac 1 3 (\tau_1 ^3 -\tau^3_2))) w(k,\tau_2 )d\tau_2  d\tau_1.
\end{align*}
For the second term we insert $\alpha^2=\frac \kappa \beta $ and deduce that
\begin{align*}
\begin{split}
    \frac \kappa \beta k^2 \int_{s_0}^{s} \int_{s_0}^{\tau_1 }  &\tfrac {1+\tau_1^2}{1+\tau_2^2} \exp(-\kappa k^2 (\tau_1-\tau_2+\tfrac 1 3 (\tau_1 ^3 -\tau^3_2))) w(k,\tau_2 ) d\tau_2 d\tau_1 \\
    &= \frac 1 \beta  \int_{s_0}^{s }  \tfrac 1{1+\tau_2^2} \int_{\tau_2 }^{s} \kappa  k^2 (1+\tau_1^2) \exp(-\kappa k^2 (\tau_1-\tau_2+\tfrac 1 3 (\tau_1 ^3 -\tau^3_2))) w(k,\tau_2 ) d\tau_1d\tau_2 \\
    &= \frac 1 \beta  \int_{s_0}^{s }  \tfrac {w(k,\tau_2)} {1+\tau_2^2}\left[ \exp(-\kappa k^2 (\tau_1-\tau_2+\tfrac 1 3 (\tau_1 ^3 -\tau^3_2)))  \right]^{\tau_1 =s }_{\tau_1 =\tau_2 } d\tau_2 \\
    &=  \frac 1 \beta  \int_{s_0}^{s }  \tfrac {w(k,\tau_2)} {1+\tau_2^2} (1- \exp(-\kappa k^2 (s-\tau_2+\tfrac 1 3 (s ^3 -\tau^3_2)))) d\tau_2.
\end{split} 
\end{align*}
This further yields that
\begin{align}
    w(k,s)&= 1-\frac 1 \beta  \int_{s_0}^{s } d\tau_2 \tfrac {w(k,\tau_2)} {1+\tau_2^2} (1- \exp(-\kappa k^2 (s-\tau_2+\tfrac 1 3 (s ^3 -\tau^3_2)))).\label{wk_rel}
\end{align}
Therefore, if $1\ge w(k,s)\ge 0$ we obtain 
\begin{align*}
    \vert w(k,s )-1 \vert &\le 
     \frac 1 \beta  \int_{s_0}^{s }  \tfrac 1{1+\tau_2^2}d\tau_2 \le \tfrac 1 \beta (\arctan(s)+\tfrac \pi 2 ).
\end{align*}
and by bootstrap this assumption holds for all times if $\beta \ge \pi $. For the current $j(k)$ we similarly estimate
\begin{align*}
     \int_{s_0}^{s_1}   \tfrac {1+s^2}{1+\tau_2^2} & \exp(-\kappa k^2 (s-\tau_2+\tfrac 1 3 (s^3 -\tau^3_2))) w(k,\tau_2 )d\tau_2\\
    &\le    (\int_{s_0}^{\frac \xi{5k^2} }+\int_{\frac \xi{5k^2} }^{s_1})   \exp(-\kappa k^2 (s-\tau_2+\tfrac 1 3 (s^3 -\tau^3_2)))d\tau_2 \\
    &\le  \tfrac 1 {\kappa k^2 }  (\exp(-\kappa \xi (\tfrac 1 3 + 3^{-4}\tfrac {\xi^2}{k^4}))+ \tfrac 4 {\eta^2} ) \le \tfrac c {\kappa k^2 } , 
\end{align*}
which yields 
\begin{align*}
        \alpha k j(s,k ) &\le  (\alpha k)^2 \tfrac c {\kappa k^2 } = \tfrac c {\beta }. 
\end{align*}
Concerning the $k-1$ mode we argue similarly and write
\begin{align*}
    j(k-1) &= \alpha (k-1) \int \exp (\kappa \tfrac {\xi^2}{k^2} (s-\tau) ) w(k-1)  d\tau 
\end{align*}
and
\begin{align*}
    w(k-1)&=c\tfrac \xi {k^2} \int \tfrac 1 {1+\tau^2} w(k)d\tau -\alpha(k-1) \int j(k-1) \ d\tau_1\\
    &=c\tfrac \xi {k^2} \int \tfrac 1 {1+\tau^2} w(k)d\tau \\
    &-\alpha^2(k-1)^2 \iint   \exp (\kappa \tfrac {\xi^2}{k^2} (\tau_1-\tau_2) ) w(\tau_2, k-1) d\tau_2 d\tau_1    \\
    &= c\tfrac \xi {k^2} \int_{s_0}^{s_1} \tfrac 1 {1+\tau^2} w(k)d\tau - \tfrac {\alpha^2k^2 (k-1)^2} {\kappa \xi^2 } \int d\tau_2  w(\tau_2, k-1).
\end{align*}
Since
\begin{align*}
    \vert \tfrac {\alpha^2k^2 (k-1)^2} {\kappa \xi^2 } \int d\tau_2 \vert  &\le \tfrac 1 {\beta \frac \xi {k^2} }
\end{align*}
we deduce by bootstrap that
\begin{align*}
    \vert w(k-1)\vert &\le 2 \pi c\tfrac \xi {k^2}
\end{align*}
and thus 
\begin{align*}
    \vert w(k-1)- \pi c\tfrac \xi {k^2}\vert &\le c\tfrac \xi {k^2} \int_{\tau\notin[s_0,s_1]} \tfrac 1 {1+\tau^2} d\tau\\
    &+ c\tfrac \xi {k^2} \tfrac 1 \beta \int_{s_0}^{s_1} \tfrac 1 {1+\tau^2}(\arctan(\tau)-\tfrac \pi 2 ) d\tau+ \tfrac 1 {\beta \frac \xi {k^2} }2 \pi c\tfrac \xi {k^2}\\
    &\le \pi  c\tfrac \xi {k^2 } ( \tfrac 2\pi \tfrac {k^2} \xi +\tfrac {\pi}{2\beta }+ \tfrac 2 {\beta \frac \xi {k^2} })\\
    &\le \tfrac \pi 2 c\tfrac \xi {k^2}
\end{align*}
and 
\begin{align*}
    \alpha (k-1)j(k-1) &\le 2 \pi c\tfrac \xi {k^2} (\alpha (k-1))^2 \int \exp (\kappa \tfrac {\xi^2}{k^2} (s-\tau) ) d\tau \\
    &=  2 \pi c\tfrac \xi {k^2} (\alpha (k-1))^2\tfrac 1 {\kappa \frac {\xi^2}{k^2}}\\
    &= \pi c \tfrac 1{\beta  \xi  } \ll  \pi c \tfrac {\xi}{k^2}.
\end{align*}

\end{proof}
Based on this model we may thus expect that a repeated interaction or \emph{chain of resonances} starting at $k_0$
\begin{align*}
    k_0 \mapsto k_0-1 \mapsto \dots \mapsto 1
\end{align*}
results in a possible growth 
\begin{align*}
   |w(1,t_{1})| \geq  |w(k_0,t_{k_0})|  \prod_{k=1}^{k_0} C'(1+c \frac{\xi}{k^2}), 
\end{align*}
where $C'=C'(\beta)$.
Choosing $k_0\approx \sqrt{C'c\xi}$ to maximize this product and using Stirling's approximation formula we may estimate this growth by an exponential factor:
\begin{align*}
    \prod_{k=1}^{k_0} C'c \frac{\xi}{k^2} = \frac{(C'c\xi)^{k_0}}{(k_0!)^2} \approx \exp(\sqrt{C'c\xi})
\end{align*}
This suggests that stability can only be expected if the initial decays in Fourier space with such a rate, which is naturally expressed in terms of Gevrey spaces.

\begin{definition}
  \label{defi:gevrey}
  Let $s\geq 1$, then a function $u \in L^2(\T\times \R)$ belongs to the
  \emph{Gevrey class} $\mathcal{G}_{s}$ if its Fourier transform satisfies
  \begin{align*}
    \sum_k \int \exp(C |\xi|^{1/s})|\mathcal{F}(u)(k,\xi)|^2 d\xi< \infty
  \end{align*}
  for some constant $C>0$.
\end{definition}
In view of the more prominent role of the frequency with respect to $y$ and for
simplicity of notation this definition only includes $|\xi|^{1/s}$ as opposed to
$(|k|+|\xi|)^{1/s}$ in the exponent. All results in this article also extend to
more general Fourier weighted spaces $X$ (see Definition \ref{defi:X}) with respect to $x$ with norms
\begin{align*}
    \sum_k \int \exp(C |\xi|^{1/s}) \lambda_k|\mathcal{F}(u)(k,\xi)|^2 d\xi.
\end{align*}
We remark that any Gevrey function is also an element of $H^N$ for any $N \in
\N$ and that Gevrey classes are nested with the strongest constraint, $s=1$,
corresponding to analytic regularity with respect to $y$.

As the main result of this article and as summarized in Theorem \ref{thm:main} we show that the above heuristic model's prediction is indeed accurate and that the optimal regularity class for the (simplified) linearized MHD equations around a traveling wave are given by Gevrey $2$.

\subsection{Magnetic Dissipation, Coupling and the Influence of $\beta$} 
\label{sec:betainf}
In the preceding proof we have seen that the interaction of interaction of $w(k)$ and $j(k)$ is determined by the combination of the action of the underlying magnetic field of size $\alpha$ and magnetic resistivity $\kappa>0$ through the parameter 
\begin{align*}
    \beta =\tfrac \kappa {\alpha^2 }.
\end{align*}
More precisely, we recall that ignoring the influence of neighboring modes $w(k)$ and $j(k)$ are solutions of a coupled system:
\begin{align*}
    \begin{split}
        \partial_t w(k) &= -\alpha k j (k) ,\\
        \partial_t j(k)    &=(2 \tfrac {t-\frac \xi k  }{1+(\frac \xi k -t )^2 } -\kappa k^2 (1+ (\tfrac \xi k -t )^2 ) ) j(k)  +\alpha k w(k).         
    \end{split}
\end{align*}
Hence starting with data $w(k,s_0)=1, j(k,s_0)=0$ three different mechanisms interact to determine the size of $w(k,s)$:
\begin{itemize}
    \item The vorticity $w(k,s)$ by means of the constant magnetic field generates a current perturbation $j(k,s)$.
    \item The current perturbation $j(k,s)$ is damped by the magnetic resistivity.
    \item The current $j(k,s)$ in turn by means of the constant magnetic field acts on the vorticity and damps it.
\end{itemize}
In this system several interesting regimes may arise, which are distinguished by the parameter $\beta$.

In the limit of infinite dissipation, $\beta \rightarrow \infty$, the current is rapidly damped and the system hence formally reduces to the Euler equations
\begin{align*}
    \dt w(k)&= 0,\\
    j(k)&=0,
\end{align*}
where $w(k,s)$ remains constant in time.

As the opposite extremal case, if $\beta \downarrow 0$ we obtain the inviscid MHD equations and the system
\begin{align*}
\partial_s w(k) &= -\alpha k j (k) ,\\
\partial_s j(k)   & = 2 \tfrac {s}{1+s^2}  j(k)  +\alpha k w(k).
\end{align*}
Hence at least for $|s|$ large this suggests that 
\begin{align*}
    w(k) \approx c_1 (1+s)\cos(\alpha k s), \quad j(k) \approx c_1 (1+s) \sin(\alpha k s).
\end{align*}
In particular, in stark contrast to the Euler equations (i.e. $\alpha=0$) for the inviscid MHD equations with a magnetic field the vorticity $w(k)$ and current perturbations $j(k)$ cannot be expected to remain close to $1$  and $0$, respectively.

This article considers the regime $0<\beta < \infty$, where the interaction of both extremal phenomena results in behavior which is qualitatively different from both limiting cases.
Indeed, recall that by a repeated application of Duhamel's formula $w(k,s)$ satisfies the integral equation \eqref{wk_rel}:
\begin{align*}
     w(k,s)&= 1-\frac 1 \beta  \int_{s_0}^{s } \tfrac {w(k,\tau_2)} {1+\tau_2^2} (1- \exp(-\kappa k^2 (s-\tau_2+\tfrac 1 3 (s ^3 -\tau^3_2)))) d\tau_2
\end{align*}
Hence, as a first case which we also discussed in the toy model of Lemma \ref{lem:toy}, if we restrict to $\beta \ge \pi$ then the integral term is bounded and small 
\begin{align*}
    \frac 1 \beta  \int_{s_0}^{s } d\tau_2 \tfrac 1 {1+\tau_2^2}\le 1. 
\end{align*}
Hence, for large $\beta$ the integral term can be treated as a perturbation and $w(k,s)$ remains comparable to $1$ uniformly in $s$ and thus close to the Euler case. However, unlike for the Euler equations the evolution of the current remains non-trivial.

If instead $0<\beta < \pi$ we obtain different behaviour depending on the dissipation $\kappa k^2$, the size of the magnetic field and the frequencies considered, whose interaction determines the behavior of the solution.
More precisely, considering the integrand
\begin{align*}
   \frac{1}{1+\tau_2^2} (1- \exp(-\kappa k^2 (s-\tau_2+\tfrac 1 3 (s ^3 -\tau^3_2)))),
\end{align*}
we observe that for $\kappa k^2\gg 1$ large  the magnetic dissipation is very strong and hence the integrand is well-approximated by $\tfrac{1}{1+\tau_2^2}$.
In particular, this suggests that for these $s$ it holds that
\begin{align*}
    w(k,s)&\approx 1-\frac 1 \beta  \int_{s_0}^{s } d\tau_2 \tfrac {w(k,\tau_2)} {1+\tau_2^2},\\
    \Leftrightarrow \ \partial_s w(k,s) &\approx -\tfrac 1\beta \tfrac 1 {1+s^2}w(k,s), \\
    \Leftrightarrow   \quad  w(k,s)&\approx \exp(-\tfrac 1 \beta (\arctan(s)+\tfrac \pi 2 ))w(k,s_0),
\end{align*}
and hence $w(k,s)$ might decay by a factor comparable to $\exp(-\frac{\pi}{\beta})$.

If instead $ \kappa k^2\leq 1$ is small, different effects interact and involve the following natural time scales:
\begin{itemize}
    \item Mixing enhanced magnetic dissipation becomes relevant on time scales $(\kappa k^2)^{-1/3}\gg 1$.
    \item The resonant interval $I^k$ is of size about $\frac{\xi}{k^2}$.
    \item Within this resonant interval most of the $L^1$ norm of $\frac{1}{1+\tau_2^2}$ is achieved on a much smaller sub-interval of size about $1$.
\end{itemize}
Hence, for times  $|s|< s^\ast \ll (\kappa k^2)^{-1/3}$ which are small compared to the disspation time scale the integrand is small and we may therefore expect that 
\begin{align*}
   w(k,s)\approx 1 
\end{align*}
remains constant.
If we instead consider very large times  $|s| \gg (\kappa k^2)^{-1/3}\gg s^\ast$ in view of the exponential factor and the decay of $\tfrac{1}{1+\tau_s^2}$ the size of $w(k,s)$ should largely be determined by the action of the time interval $(-s^\ast, s^\ast)$, that is 
\begin{align*}
      w(k,s)    &\approx 1-\frac 1 \beta  \int_{-s^\ast }^{s^\ast  } \tfrac 1 {1+\tau_2^2} d\tau_2  \\
    &\approx 1-\tfrac\pi \beta,
\end{align*}
provided such such $s$ exist, that is if the size $\tfrac{\xi}{k^2}$ of $I^k$ is much bigger than the dissipative time scale.
In particular, the size of $w(k,s)$ transitions from being close to $1$ for $|s|<s^{\ast}$ to being very far from $1$ for $|s| \gg (\kappa k^2)^{-1/3}$ and further needs to be controlled on intermediate time scales.
These different regimes all have to be considered in the upper and lower bounds of Section \ref{Ik} and we in particular need to control the size of $w(k,s)$ in order to estimate the resulting norm inflation due to resonances. For this purpose we estimate $w(k,s )$ in terms of a growth factor $L$ such that 
\begin{align*}
    \vert w(k,s)\vert &\le L w(k,s_0),
\end{align*}
as we discuss in Appendix \ref{LApp}. For our upper bounds we will require that $c L \ll 1$ is sufficiently small to control back-coupling estimates.

\section{Stability for Small and Large Times}
\label{sec:small}

In this section we establish some general estimates on the (simplified) linearized MHD equations \eqref{wpsi1}.
We note that these equations decouple with respect to $\xi$. In the following we hence treat $\xi$ as an arbitrary but fixed parameter of the equations and consider \eqref{wpsi1} as an evolution equation for the sequences $w(\cdot, \xi, t)$ and $j(\cdot, \xi,t)$.
As mentioned following the statement of Theorem \ref{thm:main} in addition to $\ell^2(\Z)$ all proofs in the remainder of the article hold for a rather general family of weighted spaces:
\begin{definition}
  \label{defi:X}
Consider a weight function $\lambda_l>0$ such that
\begin{align*}
    \sup_l \tfrac {\lambda_{l\pm 1}}{\lambda_{l} }=:  \hat \lambda <10. 
\end{align*}
Then we define the Hilbert space $X$ associated to this weight function as the set of all sequences $u: \ \bbZ \to \bbC $ such that  $(u_l\lambda_l)_l \in \ell^2$. 
\end{definition}
This definition for instance includes $\ell^2$ ($\lambda_l=1$), (Fourier transforms of) Sobolev spaces $H^s$ ($\lambda_l=1+C|l|^{2s}$ with $C>0$ sufficiently small) or Gevrey regular or analytic functions with a suitable radius of convergence. 

As sketched in Section \ref{sec:toy} for a given frequency $\xi\in \R$ we expect the norm inflation for evolution by \eqref{wpsi1} to be concentrated around times $t_k \approx \frac{\xi}{k}$ for suitable $k\in \Z$.
In particular, if the time is too large, $t>2\xi$, there exists no such $k$ and we expect the evolution to be stable.
Similarly, if $t$ is small also the size of the resonance predicted by the toy model is small and we again expect the evolution to be stable.
The results of this section show that this heuristic is indeed valid and establish stability for ``small'' and ``large'' times. The essential difficulty in proving Theorem \ref{thm:main} thus lies in control the effects of resonances in the remaining time intervals, which are studied in Section \ref{Ik}. In the following we will often write $L^\infty_t$ as the supremum norm till time $t$.

\begin{lemma}[Large time]\label{lem:LT}
Consider the equation \eqref{wpsi1} on the time interval $(2\xi, \infty)$. Then the possible norm inflation is controlled uniformly in time 
\begin{align*}
    \Vert w,j\Vert_X (t) &\le \tfrac 1{1-4c}\tfrac 1{1-2c\hat \lambda }\Vert w,j\Vert_X(2\xi),
\end{align*}
where $\hat \lambda = \max_l \tfrac {\lambda_l}{\lambda_{l\pm 1 }}$ is as in Definition \ref{defi:X}.
\begin{proof}

Let $\hat w (k) = \vert (w(k), j(k))\vert$. 
Then we infer 
\begin{align*}
    \tfrac 1 2 \partial_t \hat w^2(k) &\le (a(k-1)w(k-1) -a(k+1) w(k+1))w(k)+ b(k) j(k)^2, 
\end{align*}
where we introduced the short-hand notation $a,b$ for the coefficient functions.
Since $b(t,k) \le 0 $ for $t \ge 2 \xi$, we further deduce that 
\begin{align*}
    \tfrac 1 2 \partial_t \hat w^2 (k) &\le c\tfrac \xi {1+ ( t-\xi )^2 }(\hat w(k+1)+\hat w(k-1)+2\hat w(k) )\hat w(k) \\
    \leadsto \ \hat w^2(k,t) &< \hat w^2(k,2\xi) +  2c(\vert \hat w^2(k)\vert_{L^\infty_t}+\tfrac 1 2 \vert \hat w^2(k+1)\vert_{L^\infty_t }  +\tfrac 1 2 \vert \hat w^2(k-1)\vert_{L^\infty _t} ).
\end{align*}
Hence by a bootstrap argument we control
\begin{align*}
    \hat w^2(k,t) &\le \tfrac 1{1-4c} \sum_l (2c)^{\vert k-l\vert } \hat w^2 (l,2\xi ).
\end{align*}
Summing this estimate with the weight $\lambda_k$ then concludes the proof: 
\begin{align*}
    \Vert w,j\Vert_X (t) &\le \tfrac 1{1-4c} \sqrt{ \sum_k \lambda_k \sum_l (2c)^{\vert k-l\vert } \hat w^2(l,2\xi)}\\
    &\le \tfrac 1{1-4c} \sqrt{  \sum_l \lambda_l \hat w^2(l,2 \xi) \sum_k (2c\hat \lambda )^{\vert k-l\vert }}\\
    &\le  \tfrac 1{1-4c} \tfrac 1{1-2c\hat \lambda }  \Vert w,j\Vert_X (2\xi)
\end{align*}

\end{proof}
\end{lemma}

Thus it suffices to study the evolution for times $t<2\xi$. In view of the estimates of Section \ref{sec:toy} it here is convenient to partition $(0,2\xi)$ into intervals where $t\approx \frac{\xi}{k}$ for some $k \in \Z$.
\begin{definition}
  \label{defi:Ik}
Let $\xi>0$ be given, then for any $k\in \N$ we define \
\begin{align*}
    t_{k}&=\frac{1}{2}(\tfrac{\xi}{k+1}+\tfrac{\xi}{k}) \text{ if } k>0, \\
    t_0 &= 2 \xi.
\end{align*}
We further define the time intervals $I^k=(t_k,t_{k-1})$, for $\xi<0$ we define $t_k$ analogously for $-k \in \N$.
\end{definition}
Note that 
\begin{align*}
     t_{k}< \tfrac{\xi}{k} < t_{k-1}
\end{align*}
and
\begin{align*}
    t_{k-1}-t_{k} = \frac{1}{2}\left(\tfrac{\xi}{k+1} - \tfrac{\xi}{k-1}\right) = \tfrac{\xi}{k^2-1}.
\end{align*}
Hence $I_k$ is an interval containing the time of resonance $\frac{\xi}{k}$ and is of size about $\frac{\xi}{k^2}$.

The next lemma provides a very rough energy-based estimate, which will allow us to control the evolution for small times and frequencies. That is, we show that it is easy to obtain a energy estimate with $\xi t$ in the exponent. If the time $t$ or the frequency $\xi$  are small this rough estimate is sufficient.
However, for Gevrey 2 norm estimates it will be  necessary to improve this control to a $C\sqrt \xi$ term in the exponent in subsequent estimates. Furthermore, we remark that also the magnetic part needs to be handled adequately, since it may give an additional growth by $\exp(\tfrac 4 3  \kappa^{-\frac 1 2 })$.

\begin{lemma}[Rough estimate]\label{lem:rough}
Consider a solution of \eqref{wpsi1}, then for fixed $\xi$  and for all times $t>0$ it holds that
\begin{align*}
    \Vert w,j\Vert_{X } (t)&\le \exp(\tfrac 4 3  \kappa^{-\frac 1 2 })  \exp((1+ \hat \lambda )  c\xi t )\Vert w,j\Vert_{X }(0) .
\end{align*}
\begin{proof}
We define $\hat w(k)  = \vert w,j\vert(k)$, then 
\begin{align*}
    \tfrac 1 2 \partial_t \hat w^2 (k) &= (a(k+1) w(k+1) - a(k-1) w(k-1)) w(k) + b(k ) j(k)^2 
\end{align*}
with $b(k,t ) =(2 \tfrac {t-\frac \xi k  }{1+(\frac \xi k -t)^2} -\kappa k^2(1+(\tfrac \xi k - t)^2))$.
We further define the define the growth factor 
\begin{align*}
    M(k,t) &= \begin{cases}
        1 & \text{ if }  t- \tfrac \xi k \le 0  \ \text{ or }  \ \kappa k^2 \ge 1,  \\
        \tfrac 1 {1+(\frac \xi k -t)^2}& \text{ if }  0\le t-\tfrac \xi k \le (\tfrac 2{ \kappa k^2})^{\frac 1 3 } \ \text{ and  }\ \kappa k^2 \le 1, \\
        \tfrac 1 {1+ (\tfrac  2{\kappa k^2 })^{\frac 13 }} & \text{ if }  (\tfrac 2 { \kappa k^2})^{\frac 1 3 }\le t-\tfrac \xi k  \text{ and  } \ \kappa k^2 \le 1.
    \end{cases}
\end{align*}
We note that this weight satisfies  $b(k,t)+ \tfrac {M'}{M}(k,t) \le 0 $. Hence, defining the energy 
\begin{align*}
   E= (\prod_l M(l,t))^2\sum_k \lambda_k  \hat w(k,t)^2, 
\end{align*}
 we deduce that
\begin{align*}
    \tfrac 1 2 \partial_t E &\le \left(\prod_l M(l,t)\right)^2\sum_k \lambda_k \left( a(k+1)\hat w(k+1,t)+a(k-1)\hat w(k-1,t)\right)\hat w(k)\\
    &\le    (\prod_l M(l,t))^2\sum_k (\lambda_k a(k)+\tfrac {a(k-1)\lambda_{k-1}+a(k+1)\lambda_{k+1}}2)\hat w^2(k,t) \\
    &= (1+ \hat \lambda )c\xi E.
\end{align*}
Applying Gronwall's inequality thus yields 
\begin{align*}
    E(t) &\le \exp(2(1+ \hat \lambda ) c\xi t ) E(0) .
\end{align*}
This in turn leads to the estimate 
\begin{align*}
    \Vert w,j\Vert_{X } &\le \exp((1+ \hat \lambda ) c\xi t ) \prod_l \vert M(l,t)\vert^{-1} \Vert w_0,j_0\Vert_{X }\\
    &\le \exp((1+ \hat \lambda ) c\xi t ) \prod_{l=0}^{ \kappa^{-\frac 1 2 }}(1+ (\tfrac  2{\kappa l^2 })^{\frac 13 }) \Vert w_0,j_0\Vert_{X }.
\end{align*}
Finally, we can use Stirling's approximation of the factorial, which results in the desired estimate:
\begin{align*}
    \Vert w,j\Vert_{X } (t)&\le  \exp(\tfrac 4 3  \kappa^{-\frac 1 2 })  \exp((1+ \hat \lambda )  c\xi t )\Vert w_0,j_0\Vert_{X } .
\end{align*}

\end{proof}
\end{lemma}

In the following we establish upper and lower bounds for small times. Here we use that for modes $k$ such that $\tfrac \xi{k^2}$ is small nay possible resonance will not produce large enough norm inflation and the evolution can hence be treated perturbatively. 
More precisely, we consider the evolution on the time interval
\begin{align*}
I=  [0,\tfrac \xi 2(\tfrac 1 {k_0} +\tfrac 1 {k_0-1})] 
\end{align*}
for fixed $k_0$ to be determined later. For this purpose we introduce the parameter $\eta_0:= \tfrac \xi {k_0^2}$ which later will be chosen as $\eta_0\approx \tfrac 1 {10 c } $.
\begin{lemma}\label{low_time}
Let  $w, \ j $ be a solution of \eqref{wpsi1}, define $d:=c^{-1}$ and let  $\xi, k_0$ be such that $\eta_0 \le d^2 $. Then for all times $0\leq t\le  t_{k_0}$ it holds that
\begin{align*}
    \Vert w(t),j(t) \Vert_{X }^2
    &\le \exp(2(1+\hat \lambda )  \max(c \eta_0,1) \sqrt {\xi  \eta_0}   )\Vert w_0,j_0 \Vert_{X }^2\\
    &\le \exp(C \sqrt \xi    )\Vert w_0,j_0 \Vert_{X }^2.
\end{align*}
 Furthermore, suppose that $k_0\ge \kappa^{-\frac 1 2 }$ and  $10d\le \tfrac \xi {k_0^2} \le \tfrac 1 {100 c^2 } $, then for the initial data $w(k,0)=\delta_{k_0, k}$ and $j(k,0)=0$  we obtain that 
\begin{align*}
    w(k_0,t_{k_0} )&\ge \tfrac 1 2 \max (1,w(k,t_{k_0}), j (k,t_{k_0}))\\
    j(k_0,t_{k_0}) &\le \tfrac 1{\alpha_{k_0} \xi \eta_0 } .
\end{align*}

\begin{proof}
Computing the time derivative, we obtain
\begin{align*}
   \tfrac 1 2  \partial_t \Vert w, j\Vert_{X }^2 &= \sum_l (a(l+1) w(l+1)+a(l-1) w(l-1 ))\lambda_l w(l) + b(l)\lambda_l  j(l)^2,
\end{align*}
where the coefficient functions satisfy
\begin{align*}
    a(l) &\le
    \begin{cases}
         c\eta_0 & l\ge k_0  \\
         4c\tfrac 1 {1+\eta_0} & l\le k_0  
    \end{cases}
    \\
    &\le \max(c\eta_0,4c ),\\
    b(l) &\le 1.
\end{align*}
Therefore, we conclude that 
\begin{align*}
    \partial_t \Vert w,j \Vert_{X }^2 &\le 2(1+\hat \lambda ) \max( c\eta_0,1) \Vert w,j \Vert_{X }^2,\\
    \Vert w,j \Vert_{X }^2(t)&\le \exp(2(1+\hat \lambda ) \max( c\eta_0,1) t)\Vert w_0,j_0 \Vert_{X }^2,\\
    \Vert w,j \Vert_{X }^2(t_{k_0})&\le \exp(2(1+\hat \lambda ) \max( c\eta_0,1) \sqrt {\xi  \eta_0} )\Vert w_0,j_0 \Vert_{X }^2. 
\end{align*}

To prove lower bounds on the norm inflation we further need to show that for $w(k,0)=\delta_{k_0,k }$ and $j(k,0)=0$, the mode $w(k_0,t_{k_0})$ will stay the largest mode. Therefore, we introduce the short-hand notation
\begin{align*}
    \hat w(k,t) &= \vert w,j\vert(k,t) 
\end{align*}
and have to estimate the growth of $\hat w(\cdot, t)$.
Since on the interval $[0, t_{k_0} ]$ it holds that $b(k) \le 0 $ as $k_0\ge \kappa^{-\frac 12}$, we obtain the system 
\begin{align*}
    \partial_t \hat w(k) &\le a(k+1)\hat w(k+1) +a(k-1)\hat w(k-1) \\
    \hat w(k,0) &= \delta_{k,k_0} .
\end{align*}

Let $\sqrt { \xi \pi}  = l_0\ge l \ge k_0$ to be fixed later. We want to prove by induction that
\begin{align}
\begin{split}
    \hat w(m,t_{l-1}) &\le 6\pi c \eta_0  (2c)^{\vert m-k_0\vert }\\
    \hat w(l,t_{l-1}) &\le 4(2c)^{\vert l-k_0\vert }\\
    \hat w(n,t_{l-1}) &\le 2 (2c)^{\vert m-k_0 \vert  }\label{small_boot}
\end{split}
\end{align}
for all $m>l> n$. \\  
\textbf{Induction start: }   \\
We integrate $a$ in time to estimate 
\begin{align*}
    \int_0^{t_{l_0-1}} a(k)&= c\tfrac \xi{k^2} \int_0^{t_{l_0-1}} \tfrac 1{1+(\frac \xi k-t)^2}\\
    &\le \left\{\begin{array}{cc}
         \pi c \tfrac \xi{k^2 }& k> l_0  \\
         c & k\le l_0
    \end{array} \right. \\
    &\le c.
\end{align*}
Thus we obtain that 
\begin{align*}
    \hat w(k,t_{l_0-1}) &< \delta_{k_0,k} + c (\vert \hat w(k+1)\vert_{L^\infty_t}+\vert \hat w(k-1)\vert_{L^\infty_t} ),
\end{align*}
which by a bootstrap argument yields that 
\begin{align*}
    \hat w(k, t_{l_0-1}) &\le \tfrac 1 {1-2c } (2c)^{\vert k_0 - k \vert }
\end{align*}
for all $k$ which satisfy \eqref{small_boot}. \\
\textbf{Induction step: } We fix $l$ and we assume that \eqref{small_boot} holds for all $\tilde l$ with $l_0\ge \tilde l \ge l+1\ge k_0+1 $ and then prove that it holds also for $l$. We here argue by bootstrap. That is, we show that the estimate \eqref{small_boot} at least holds up until a time $t^\ast$ with $t_{l}\leq t^\ast\leq t_{l-1}$ and that the maximal time with this property is given by $t^\ast=t_{l-1}$. For $n< l $ we estimate
\begin{align*}
     \hat w(n,t_{l-1} ) &\le\delta_{n,k_0}+ \int_0^{t_{l-1}}  a(n\pm 1 , \tau) w(n\pm 1 ,\tau  )\\
    &\le \delta_{n,k_0}+  c (4 (2c)^{\vert n+1 -k_0\vert }+2 (2c)^{\vert n-1-k_0\vert })\\
    &<  2(2c)^{\vert n-k_0\vert }.
\end{align*}

To estimate the $l$ mode we estimate the integral between $t_l$ and $t_{l-1} $ to deduce 
\begin{align*}
     \hat w(l,t_{l-1} ) &\le\hat w(l,t_{l} )+   \int_{t_l}^{t_{l-1}}   a(l\pm 1 , \tau) w(l\pm 1  ,\tau )\\
     &\le  2(2c)^{l-k_0} +6 \pi   c \eta_0 (2c)^{l+1-k_0 }  + 2c (2c)^{\vert l-1-k_0\vert }\\
     &\le 4(2c)^{l-k_0 }.
\end{align*}
 For $m> l$ we split the integrals as
\begin{align*}
     \hat w(m,t_{l-1} ) &\le \hat w(m,t_{m-1 } )+ \int_{t_{m-1}}^{t_l}  a(m+1 , t) \hat w(m+ 1 ,\tau) \\
     &+ \int_{t_{m-1}}^{t_{m-2}}  a(m-1 , t) \hat w(m-1 ,\tau  )+ \int_{t_{m-2}}^{t_l} a(m-1 , t) \hat w(m- 1 ,\tau  )\\
     &\le   4(2c)^{m-k_0 } +12\pi c^2 \eta_0  (2c)^{\vert m-k_0\vert } + 4\pi  \eta_0 (2c)^{m-k_0 }+6\pi c\eta_0   (2c)^{m-k_0 }\\
     &\le 6  \pi  \eta_0 (2c)^{m-k_0 }.
\end{align*}
So we finally deduce that
\begin{align*}
    \hat w(k,t_{k_0}) \le (2c)^{\vert k-k_0\vert } \left\{\begin{array}{cc}
         4&k \le k_0+1  \\
         6\pi \eta_0 & k>k_0+1 
    \end{array}\right. 
\end{align*}
Thus we established an upper bound for all modes, the next step is to show that for $w$ indeed the $k_0$ mode is one of the largest modes. Therefore, we estimate  $j(k_0)$ by
\begin{align*}
    j (k_0,t) &= \alpha k_0 \int^t_0 d\tau \ \tfrac {1+ (\frac \xi k_0- t)^2 }{1+ (\frac \xi k_0- \tau)^2 }\exp\left( -\kappa k_0^2 (t-\tau +\tfrac 1 3 ( (\tfrac \xi {k_0}- t)^3 - (\tfrac \xi {k_0}- \tau )^3 )\right)w(k_0,\tau ) 
\end{align*}
and hence obtain that
\begin{align*}
    j(t)&\le \alpha k_0 \int \exp (-\kappa k_0^2 \eta^2 (t-\tau)\\
    &\le \tfrac 1{\sqrt{\beta \kappa \xi \eta_0^3 }}=\tfrac 1{\alpha_{k_0} \xi \eta_0 }
\end{align*}
and 
\begin{align*}
    \alpha {k_0}\int j(k_0,\tau_2 )d\tau_1 &= \tfrac {\kappa k_0^2 }\beta\int^t_0 d\tau_1 \int^{\tau_1 }_0 d\tau_2 \ \tfrac {1+ (\frac \xi k_0- \tau_1 )^2 }{1+ (\frac \xi k_0- \tau_2)^2 }
    \\ & \quad \quad \times \exp\left( -\kappa k_0^2 (\tau_1 -\tau_2 +\tfrac 1 3 ( (\tfrac \xi k_0- \tau_1 )^3 - (\tfrac \xi k_0- \tau_2  )^3 )\right)\\
    &\le \tfrac 1 {\beta } \int d\tau_1 \tfrac 1 {1+(\frac \xi {k_0} - \tau_1)^2}\\
    &\le  \tfrac 4 {\beta\eta_0 }.
\end{align*}
With this we conclude that
\begin{align*}
    \vert w(k_0 ,t_{k_0}) -1 \vert &\le \int a(k_0+1) w(k_0+1) + a(k_0-1)w(k_0-1) + \alpha k_0 j(t)\\
    &\le 16\pi \eta_0 c^2 + 8c^2+\tfrac 4 {\beta\eta_0 } \le \tfrac 1 2,
\end{align*}
which in turn yields
\begin{align*}
    \vert w(k_0 ,t_{k_0})\vert &\ge \tfrac 1 2 \ge \max_{k\neq k_0}(  \hat w(k ,t_{k_0}) , \vert j(k_0 ,t_{k_0} )\vert).
\end{align*}

\end{proof}

\end{lemma}

\section{Resonances and Norm Inflation} \label{Ik}
Having discussed the evolution for small times and large times in Section \ref{sec:small} it remains to discuss the evolution on the interval
\begin{align*}
    (t_{k_0}, 2\xi)= \bigcup_{1\leq k \leq k_0} I^k
\end{align*}
with $I^k$ as in Definition \ref{defi:Ik}.

Based on the heuristics of the toy model of Section \ref{sec:toy} our aim here is to establish both upper lower and upper bounds on the norm inflation on each resonant interval $I^k$, where the resonant mode $w(k)$ can possibly lead to a large growth of its neighboring modes $w(k\pm 1)$.
In order to simplify notation we introduce the growth factor
\begin{align*}
L=L(\alpha ,\kappa, k),
\end{align*}
which estimates the maximal growth of $w(k)$ due to its interaction with the current $j(k)$, see Appendix \ref{LApp}.
In particular, we show that $L=1$ if $\beta\ge \pi$ and if $\beta <\pi $ we obtain an estimate $L=L(\alpha ,\kappa, k)\le \sqrt c $. We define $M$ and $M_n$ as
\begin{align*}
    M &= \sum_m 10^{-\vert m \vert }(w+\tfrac 1 {\alpha_{k_m}  }j)(k_m, s_0)  \\
    M_n &= \sum_{m} 10^{-\vert m-n\vert+\chi} (w+    \tfrac {1} {\alpha_{k_m}}   j)(k_m, \tilde s_0) 
\end{align*}
where $\chi= -\vert \sgn (m )-\sgn(n)\vert $. We note that 
\begin{align*}
    \sum_{l\neq 0}  (\tfrac 3\eta )^{\vert k_ l-k_0\vert } M_l &\le \tfrac 3\eta M.
\end{align*}
With these notations the main results of this section are summarized in the following theorem: 
\begin{theorem}\label{Thm_I}
Let $c\le \min ( 10^{-3} \beta^{\frac {16} 3 }, 10^{-4} ) $
, $\xi \ge 10   \kappa^{-1}(1+\beta^{-1} )  $ and $\eta=\tfrac \xi{k^2}\ge 10 d$ and $t_k = \tfrac \xi 2 (\tfrac 1 k+\tfrac 1{k+1})$, then it holds that
\begin{align*}
    \Vert w,j\Vert_X(t_{k-1} ) &\le 18\pi L \hat \lambda (c\eta)^\gamma  \Vert w,j\Vert_X(t_k).
\end{align*}
Furthermore, let $\kappa_k \min(\beta,1 )\ge \tfrac 1 c  $   and  $\beta \ge \tfrac 1 5 $ 
\begin{align}
\begin{split}
w(k,t_{k} )&\ge \tfrac 1 2 \max (w(l,t_{k}), j (l,t_{k})).  \label{Echo_cond}
\end{split}
\end{align}
Then $w(k-1,t_{k-1})$ satisfies \eqref{Echo_cond} with $k$ replaced by $k-1$ and 
\begin{align*}
    \vert w(k-1  ,  t_{k-1})\vert&\ge \min (\beta ,\pi )  (c\eta)^\gamma w(k,t_{k} ). 
\end{align*}
\end{theorem}
To prove the estimates of Theorem \ref{Thm_I} it is convenient to rescale $\tilde j(k) = \alpha_kj(k)$ in \eqref{wpsi1} to obtain 
\begin{align*}
    \partial_t w (k)  &=  -\tilde j(k)  \\
    & -c \tfrac \xi{(k+1)^2} \tfrac 1 {1+(\frac \xi {k+1} -t)^2 } w(k+1)\\
    & + c\tfrac \xi {(k-1)^2} \tfrac 1{1+(\frac \xi {k-1} -t)^2 } w(k-1)\\
    \partial_t \tilde j (k)&   =(2 \tfrac {t-\frac \xi k  }{1+(\frac \xi k -t)^2} -\kappa k^2(1+(t-\tfrac \xi k ^2)) \tilde j(k)  +\tfrac {\kappa k^2} \beta  w(k), 
\end{align*}
where we used that $\kappa =\beta \alpha^2 $. With respect to these unknowns the norm on our space $X$ changes slightly
\begin{align*}
    \Vert w, j \Vert_X^2&= \sum \lambda_k( w^2(k) +\tfrac \beta {\kappa k^2 }  \tilde j^2(k)) \\
    &=:\Vert w, \tilde  j \Vert_{\tilde X} ^2.
\end{align*}
In the following sections, with slight abuse of notation we omit writing the tilde symbols both for $j$ and $X$.

Given a choice of time interval $I_{k_0}$, considering $k_0$ as arbitrary but fixed (and unrelated to $k_0$ of Section \ref{sec:small}) we further introduce the relative frequencies
\begin{align*}
    k_n:=k_0+n,
\end{align*}
where $n\in \bbZ_{>-k_0}$ and also shift our time variable
\begin{align*}
    t=\tfrac \xi {k_0} +s.
\end{align*}
Introducing the coefficient functions
\begin{align}
\begin{split}
    a(k) &=c\eta  \tfrac {k_0^2}{(k)^2} \tfrac 1 {1+(\eta \tfrac  {k_0(k_0-k)}{k} -s)^2},\\
    b(k) &= 2 \tfrac {(s- \eta \tfrac  {k_0(k_0-k)}{k} ) }{1+(\eta \tfrac  {k_0(k_0-k)}{k} - s)^2} -\kappa_k(1+(\eta  \tfrac  {(k_0-k)k_0}{k} - s)^2),
    \end{split}
\end{align}
the system \eqref{wpsi1} then reads 
\begin{align}
\label{eq:ODE2} 
\begin{split}
    \partial_s w (k) &=  - j(k)  \\
    & \quad -a(k+1) w(k+1)\\
    & \quad + a(k-1) w(k-1),\\
    \partial_t j(k)  &   =\tfrac {\kappa_k} \beta  w(k) + b(k) j(k).
    \end{split}
\end{align}
For later reference, we note that the coefficient function $a$ satisfies the following estimates:
\begin{align}
\begin{split}
    a(k_0)&=c \eta \tfrac 1 {1+s^2},\\ 
    a(k_{\pm1 } ) &\le 4 \tfrac c\eta,\\
    a(k_{n } ) &\le  \tfrac c \eta, 
\end{split}
\end{align}
for all $\vert n \vert \ge 2$.

Finally, in view of cancellations of $-a(k-1)$ and $+a(k+1)$ on any given time interval $I^k$ it is convenient to work with the unknowns 
\begin{align*}
u_1=w(k_0), u_2=w(k_1)-w(k_{-1}), u_3=w(k_1)+w(k_{-1}).
\end{align*}
We then consider \eqref{eq:ODE2} as a forced system for these three modes (and a separate equation for all other modes):
\begin{align}
\begin{split}
    \partial_s \left( 
    \begin{array}{c}
         u_1  \\
         u_2 \\
         u_3\\
         j (k_0)
    \end{array}
    \right) &= \left( 
    \begin{array}{cccc}
         0&-a_1&a_2 &-1 \\
         2 c\eta\tfrac 1 {1+s^2} &0&0&0  \\
                  0&0&0&0\\
         \tfrac{\kappa_k}\beta &0&0&\tfrac {2s}{1+s^2} -\kappa_k (1+s^2)
    \end{array}
    \right)\left( 
    \begin{array}{c}
         u_1  \\
         u_2 \\
         u_3\\
         j (k_0)
    \end{array}
    \right)\\
    &+ \left( 
    \begin{array}{c}
         0  \\
         -a(k_{\pm 2})w(k_{\pm 2}) \mp j(k_{\pm 1} ) \\
         \mp a(k_{\pm2} )w(k_{\pm 2}) -j(k_{\pm 1})\\
         0
    \end{array}
    \right) \label{Echo}
\end{split}
\end{align}
where $a_1=\tfrac 1 2 ( a(k_1)+a(k_{-1})) $ and $a_2=\tfrac 1 2 ( a(k_1)-a(k_{-1} ))$.

The analysis of this system is split into multiple subsections, where we also split the time interval $I^k$ as
\begin{align*}
    I^{k_0} = [ s_0,-d] \cup  [-d,d]\cup [d,  s_1] =: I_1\cup I_2 \cup I_3,
\end{align*}
where $ s_0 = - \tfrac \eta 2 \tfrac {k_0-1}{k_0}$ and $ s_1= \tfrac \eta 2 \tfrac {k_0+1}{k_0} $. 
Similarly to the setting of the Euler equations \cite{dengZ2019} here the interaction between growth and decay of various modes interacts to determine the over all norm inflation.

\subsection{Proof of Theorem \ref{thm:main}}
Before proceeding to the proof of Theorem \ref{Thm_I}, in this subsection we discuss how it can be used to establish Theorem \ref{thm:main}.
We split the proof into two auxiliary theorems.

\begin{theorem}[technical statement]\label{TechThm}
Let $c\le \min( 10^{-3}\beta^{\frac {16} 3 }, 10^{-4} )$, $\xi \ge 10  \kappa^{-1}(1+\beta^{-1} )$ and $\tfrac \xi {k^2 } \ge 10 d $. Then there exists exists a constant $C=C(\kappa , \alpha, c) $ such that for a fixed $\xi$ we obtain 
\begin{align*}
    \Vert w,j\Vert_X(t,\xi ) &\le \exp(C\sqrt \xi )  \Vert w, j\Vert_X(0,\xi).
\end{align*}
Furthermore, let $\xi \ge  10^4\tfrac {d^2}{\beta \kappa} $, $\beta\ge \tfrac 1 5 $,  $k_0 \approx \tfrac c{10}    \sqrt \xi  $ and $k_1\approx \tfrac {4}{\sqrt{\beta \kappa}  }$, then there exists a constant $C^\ast =C^\ast (\kappa,\alpha  , c)$ such that  for initial data $w(k, 0)= \delta_{k_0,k} $ and $j(k,0)=0$ we obtain 
\begin{align*}
    w(k_1,t) &\ge \exp( \tilde C \sqrt{\xi} ).
\end{align*}
for $t\in [t_{k_1}-1,t_{k_1}+1]$.
\end{theorem}

\begin{proof}[Proof of Theorem \ref{TechThm}]
For fixed $\xi$, $t$ and $k_0 = : 10d \sqrt{\xi}$ we consider $w(\cdot, \xi, t)$ as an element in $X$. On $X$ we define the operator $S_{\tau_1,\tau_2 } : \ X\to X $ as the solution operator of \eqref{wpsi1} on $[\tau_1,\tau_2  ]$, i.e.
\begin{align*}
    S_{\tau_1,\tau_2 } [w(\cdot , \xi ,\tau_1)]&=w(\cdot , \xi ,\tau_2  )\\
    S_{\tau_1 ,\tau_2}\circ S_{\tau_2,\tau_3}&= S_{\tau_1 ,\tau_3}. 
\end{align*}
By Lemma \ref{low_time}, Theorem \ref{Thm_I} and Lemma \ref{lem:LT} this $S$ then satisfies the following norm estimates: 
\begin{align*}
    \Vert S_{0,t_{k_0+1}}\Vert_{X\to X} &= \exp(C_1 \sqrt \xi),\\
    \Vert S_{t_k ,t_{k-1}}\Vert_{X\to X}&= 3\pi c (\tfrac \xi {k^2})^\gamma, \\
    \Vert S_{t_1,t}\Vert_{X\to X}&=2 \tfrac 1 {1-c\sqrt {\hat \lambda }}.
\end{align*}
Combining these estimates with Stirling's approximation formula we thus obtain the desired upper bound:
\begin{align*}
    \Vert S_{0,t} \Vert_{X\to X}&\le 2 \exp(C_1 \sqrt \xi ) \prod_{k=1}^{k_0}  3\pi c (\tfrac \xi {k^2})^\gamma\\
    &\le \exp(C\sqrt \xi ) . 
\end{align*}

Concerning the lower bound, we use first use Lemma \ref{low_time} and then Theorem \ref{Thm_I} to deduce that 
\begin{align*}
    w(k_0,t_{k_0}) &\ge\tfrac 12\\
    w(k-1,t_{k-1} )&\ge (c\tfrac \xi{k^2} )^{\gamma } \min(\beta,  \pi)  w(k,t_k)
\end{align*}
for $\sqrt{\frac c{10}\xi} \approx k_0\ge k \ge k_1\approx  \tfrac {4}{\sqrt{\beta \kappa}}$.
Thus, by again using Stirling's approximation, we conclude that 
\begin{align*}
    w(k_1,t_{k_1})&\ge \tfrac 1 2 \prod_{k=k_1}^{k_2}   (c\tfrac \xi{k^2} )^{\gamma } \min(\beta,  \pi) \\
    &\approx \exp(\tilde C \sqrt \xi ) .
\end{align*}

\end{proof}

\begin{theorem}[Stability and blow-up]
\label{theorem:final}
Let $c\le \min( 10^{-3}\beta^{\frac {16} 3 }, 10^{-4} )$ and $w , j $ be a solution to \eqref{wpsi1} , then there exists a constant $C=C(\kappa,\alpha ,c )$ such that for all $C_1>C$ and  initial data which satisfy
\begin{align*}
    \int \exp (C_1\sqrt \xi )  \Vert w_0, j_0\Vert_X^2(\xi) \ d\xi<\infty,
\end{align*}
the solution remains Gevrey $2$ regular in the sense that 
\begin{align*}
    \sup_t \int \exp (C_2\sqrt \xi )  \Vert w, j\Vert_X^2(\xi,t) \ d\xi \leq \tilde{C} \int \exp (C_1\sqrt \xi )  \Vert w_0, j_0\Vert_X^2(\xi) \ d\xi,
\end{align*}
where $C_2 =C_1-C$ and $\tilde{C}>0$ is a universal constant.

Furthermore, additionally suppose that $\beta \ge \tfrac 1 5 $, then  there exist a constant $0<C^\ast <C$ and initial data $w_0,j_0$ which satisfy 
\begin{align*}
    \int \exp (C^\ast \sqrt \xi )  \Vert w_0, j_0\Vert_X^2(\xi) \ d\xi<\infty,
\end{align*}
such that for a subsequence $k_{n,1 } $ the solution diverges in $L^2$: 
\begin{align*}
    \Vert w(\cdot ,t_{k_{n, 1}})\Vert_{L^2 \ell^2  }\to \infty .
\end{align*}
\end{theorem}

\begin{proof}[Proof of Theorem \ref{theorem:final}]

The first part follows directly from Theorem \ref{TechThm}. For the second part we fix $\xi_1 =10^4\tfrac {d^2}{\beta \kappa} $ and define the sequence $\xi_n = n\xi_1$ with the associated $k_{0}^{\xi_n}\approx \tfrac c{10}    \sqrt {\xi_n} $ and $k_{1}\approx \tfrac {4}{\sqrt{\beta \kappa}  }$. Note that the starting mode $k_{0}^{\xi_n}$ is $\xi_n$-dependent, but the final mode $k_1$ is independent of $\xi_n$.  Furthermore, let $z_n(\xi )$ be a function in $C^\infty \cap L^2 $, such that 
\begin{align*}
    \supp z_n(\cdot  ) &\subset [\xi_n-1 ,\xi_n+1]\\
    \int  z_n(\xi )^2 \ d\xi &=1. 
\end{align*}
We then define the initial data 
\begin{align*}
    w(k,\xi ,0)&= \sum_{n=1}^\infty \tfrac 1 {n} z_n(\xi) \exp( -\tfrac 1 2  C^\ast \sqrt{\xi} )\delta_{k_{\xi_n  ,0},k }.
\end{align*}
We observe that it satisfies the estimates 
\begin{align*}
    \Vert w(\cdot,\xi ,0)\Vert_{l^2 }^2&=\sum_{n=1}^\infty \tfrac 1 {n^2} z_n(\xi )^2 \exp( - C^\ast \sqrt\xi ),\\
    \int \exp(  C^\ast \sqrt\xi ) \Vert w(\cdot,\xi ,0)\Vert_{l^2 }^2 \ d \xi&= 2 \sum \tfrac 1{n^2}=\tfrac{\pi^2}3.
\end{align*}
Furthermore, by the norm inflation results for each  mode at each time $t_{k_{n,1}}$ we obtain that 
\begin{align*}
    \Vert w(k_{n, 1},\xi ,t_{k_{n, 1}})\Vert_{l^2  } &\ge \tfrac 9 {10}   \tfrac 1 {n^2 } z(\xi,n)  \exp( (\tilde C -C^\ast)  \sqrt\xi ),
\end{align*}
and integrating in $\xi$ we conclude that
\begin{align*}
    \Vert w(\cdot ,t_{k_{n, 1}})\Vert_{L^2l^2  } &\ge \tfrac 9 {10}   \tfrac 1 {n^2 }   \exp( (\tilde C-C^\ast) \sqrt{\xi_n })\to \infty.
\end{align*}
\end{proof}

\subsection{Asymptotic Behavior on the Intervals $I_1$ and $I_3$}
\label{sect:I13}
In this section we consider the equation \eqref{Echo} on the outer intervals $I_1 =[ s_0,-d]$ and $I_3 =[d,  s_1]$. Since a lot of calculations are similar on both intervals we in general write the interval as $[\tilde s_0,\tilde s_1]$, where we only need to distinguish the two cases on a few occasions and in the statement of the conclusions. For the interval $I_1$ we prove the following proposition:
\begin{pro}[Interval $I_1$]\label{I1pro}
Let  $c < \max( 10^{-4}, 10^{-1} \beta)$ and  $\xi \ge 10 \max (   \kappa^{-1}(1+\beta^{-1} ), k_0^2  d ) $. Then for a solution of \eqref{Echo} on the interval $I_1 $ the following estimates hold at the time $d$:
\begin{align*}
|u_1(d)| & \le 2 M (c \eta)^{-\gamma_2 },\\
|u_2(d)| & \le 2 M (c \eta)^{\gamma_1},\\
    \vert u_3 \vert(-d) &\le 2 M_1, \\
    \vert w(k_n , -d)\vert &\le 2 M_n, \\
    \vert j\vert(k_0,-d) &\le  \tfrac c\beta  (c\eta) ^{-\gamma_2 } M\inf(c, \kappa_{k_0}c^{-2}), \\
    \vert j\vert(k_{\pm 1},-d)&\le \tfrac 4{\beta \eta^2 }  M,\\
    \vert j\vert(k_n,-d)&\le   \tfrac 4{\beta \eta^2 }  M_n.
\end{align*}
If we additionally assume that 
\begin{align}
\begin{split}
    w(k_0,t_{k_0} )&\ge \tfrac 1 2 \sup_l  (w(l,t_{l}), j (l,t_{l})) , \label{addI1} 
\end{split}
\end{align}
we obtain that 
\begin{align}
\begin{split}
     \vert u_1 (-d)-(c\eta)^{-\gamma_2}u_1(  s_0)\vert      &=50c   u_1( \tilde s_0)(c\eta)^{-\gamma_2}  \\
     \vert u_2 (-d)\vert &\le 50c u_1( \tilde s_0)(c\eta)^{\gamma_1},\\
    \vert u_3\vert (-d), \vert w\vert (k_m,  -d)&\le 2 \vert  u  \vert (s_0)\qquad \text{ for } \vert m\vert \ge 2, \\
     \vert j\vert (k_m,  -d )&\le \tfrac 4 \eta \vert u\vert ( s_0)\qquad \text{ for } \vert m\vert \ge 1,\\
     \vert j\vert(k_0,-d) &\le  \tfrac {2 c^2}\beta \vert u\vert ( s_0)  (c\eta) ^{-\gamma_2 }.    \label{echo_I1}
\end{split}
\end{align}
\end{pro}
The proof of this proposition is split into several lemmas and concludes at the end of this subsection.
For the interval $I_3$ in a first step we only establish asymptotic estimates. The final conclusion for interval $I_3 $ will be postponed to the proof of Theorem \ref{Thm_I}. On both intervals $I_1$ and $I_2$ the interaction of $u_1$ and $u_2$ is the main effect to be analyzed. Therefore, we consider the equations for $u_1$ and $u_2$ as an inhomogeneous linear system 
\begin{align}
\begin{split}
    \partial_s \left( 
    \begin{array}{c}
         u_1  \\
         u_2 
    \end{array}
    \right) &= \left( 
    \begin{array}{cccc}
         0&-\tfrac c\eta    \\
         2 c\eta\tfrac 1 {s^2} &0 \\
    \end{array}
    \right)\left( 
    \begin{array}{c}
         u_1  \\
         u_2
    \end{array}
    \right) + F\label{I1_hom},
\end{split}
\end{align}
where $F$ is a force term.  Equation \eqref{I1_hom} with $F=0$ has a explicit homogeneous solution and we aim to show that \eqref{Echo} can be treated as a perturbation. In the following we denote $\tilde u$ as the homogeneous solution of \eqref{I1_hom}. Furthermore, we split the forcing as
\begin{align*}
   F=: F_{all} &= F_{3mode} + F_j + F_{u_3 }+ F_{j(k_0\pm 1 )} +F_{\tilde w }
\end{align*}
where we define 
\begin{align*}
    F_{3mode} &=(\tfrac c \eta -a_1) e_1 u_2- 2c \eta \tfrac 1{s^2(s^2+1)}e_2 u_1,
\end{align*}
as the 3 mode forcing
\begin{align*}
    F_j &=-e_1 j(k_0) 
\end{align*}
as the $k_0$-th current forcing and
\begin{align*}
    F_{u_3 }+ F_{j(k_0\pm 1 )} +F_{\tilde w }
    &=e_1 a_2 u_3 \mp e_2 j(k_{\pm 1})-e_2 a(k_{\pm 2 }) w(k_{\pm 2 }) 
\end{align*}
as the forcings due to $u_3$, $j(k_0\pm 1 )$ and $\tilde w $, respectively.
The corresponding $R[F_\ast]$ are the called $r$ changes. We also define $\gamma=\sqrt{1-8c^2}$ and $\gamma_1 =\tfrac 12(1+\gamma)$ and $\gamma_2 = \tfrac 1 2 (1-\gamma) $ and note the following equalities:
\begin{align*}
    \gamma_1\gamma_2 &= 2c^2,\\
    \gamma_1+\gamma_2 &=1,\\
    \gamma&= 1+\mathcal{O}(c^2),\\
    \gamma_1&= 1+\mathcal{O}(c^2),\\
    \gamma_2&= \tfrac 1 {\gamma_1 }2c^2= 2c^2+\mathcal{O}(c^4).
\end{align*}
\begin{lemma}
Consider \eqref{I1_hom} with $F=0$, then the solution is given by 
\begin{align*}
    \tilde u(s) &= S(s) r 
\end{align*}
with 
\begin{align*}
    S(s)&= \left( 
    \begin{array}{cc}
         \vert \tfrac s \eta\vert^{\gamma_1 }& \vert\tfrac s\eta \vert^{\gamma_2 }\\
         -\tfrac {\gamma_1 }{c } \tfrac s \eta  \vert\tfrac s \eta\vert^{\gamma_1-2 }&-\tfrac {\gamma_2 }{c }\tfrac s \eta\vert\tfrac s \eta\vert ^{\gamma_2-2 }
    \end{array}
    \right), 
\end{align*}
and $r=S^{-1} (\tilde s_0) \tilde u(s_0) $. 
\end{lemma}
Furthermore, we define the operator $S^\ast$ as 
\begin{align*}
    S^\ast (s)&= \left( 
    \begin{array}{cc}
         \vert \tfrac s \eta\vert^{\gamma_1 }& \vert\tfrac s\eta \vert^{\gamma_2 }\\
         \tfrac {\gamma_1 }{c }   \vert\tfrac s \eta\vert^{\gamma_1-1 }&\tfrac {\gamma_2 }{c }\vert\tfrac s \eta\vert ^{\gamma_2-1 }
    \end{array}
    \right),
\end{align*}
which gives the estimate
\begin{align*}
    \vert S(s) r\vert &\le S^\ast r \qquad \forall r \in ( \R_+)^2
\end{align*}
The inverse of $S$ can be computed as 
\begin{align*}
    S^{-1}(s)&= \sgn(s) c \gamma^{-1} \left( 
    \begin{array}{cc}
         -\tfrac {\gamma_2 }{c }\tfrac s \eta\vert\tfrac s \eta\vert ^{\gamma_2-2 }& -\vert\tfrac s\eta \vert^{\gamma_2 }\\
         \tfrac {\gamma_1 }{c } \tfrac s \eta  \vert\tfrac s \eta\vert^{\gamma_1-2 }&\vert \tfrac s \eta\vert^{\gamma_1 }
    \end{array}
    \right)\\
    &=\left( 
    \begin{array}{cc}
         -\tfrac {\gamma_2 }{\gamma }\vert\tfrac s \eta\vert ^{\gamma_2-1}& -\tfrac c\gamma \tfrac s \eta\vert\tfrac s\eta \vert^{\gamma_2-1 }\\
         \tfrac {\gamma_1 }{\gamma }   \vert\tfrac s \eta\vert^{\gamma_1-1 }&\tfrac c\gamma  \tfrac s \eta \vert \tfrac s \eta\vert^{\gamma_1 -1}
    \end{array}
    \right).
\end{align*}

\begin{lemma}
Let $u_1, u_2$ be a solution to \eqref{I1_hom}  with given $F=(F_1,F_2 )$, then for 
\begin{align*}
    R_1[F]&=(1+10c^2)2c^2 \eta^{1-\gamma_2 } \int_{\tilde s_0}^s \tau^{\gamma_2-1 } F_1(\tau )  \  d\tau+ c \eta^{-\gamma_2 } \int_{\tilde s_0}^s \tau^{\gamma_2 }F_2(\tau)  \ d\tau,\\
    R_2[F]&=(1+10c^2)\eta^{1-\gamma_1} \int_{\tilde s_0}^s \tau^{\gamma_1 -1 }F_1(\tau) \ d\tau+ c\eta^{-\gamma_1 } \int_{\tilde s_0}^s \tau^{\gamma_1 } F_2(\tau ) \ d\tau,
\end{align*}
we estimate 
\begin{align*}
    \vert u-\tilde u\vert &\le S^\ast(s)    R[F].
\end{align*}
\begin{proof}
Since $S$ has an inverse, we write
\begin{align*}
    u = S(s) r(s)
\end{align*}
and our aim is to control the evolution of $r(s)$. Therefore, we calculate 
\begin{align*}
    \vert \partial_s r  \vert &= S^{-1} F \\
    \vert \partial_s r_1\vert &\le 2 c^2\vert\tfrac s \eta\vert^{\gamma_2-1} F_1  + c \vert\tfrac s \eta\vert^{\gamma_2}F_2 \\
    \vert \partial_s r_2\vert &\le \vert\tfrac s \eta\vert^{\gamma_1-1} F_1  +  c \vert\tfrac s \eta\vert^{\gamma_1}F_2
\end{align*}
and so 
\begin{align*}
    \vert r_1(s) -r_1(d)\vert &\le 2 c^2 (1+10c^2)\eta^{1-\gamma_2 } \int \tau^{\gamma_2-1 } F_1(\tau )+ c \eta^{-\gamma_2 } \int \tau^{\gamma_2 }F_2(\tau) \\
    \vert r_2(s) -r_2(d)\vert &\le (1+10c^2)\eta^{1-\gamma_1} \int \tau^{\gamma_1 -1 }F_1(\tau)+ c \eta^{-\gamma_1 } \int \tau^{\gamma_1 } F_2(\tau )
\end{align*}
\end{proof}
\end{lemma}

In the following we always assume that there exists $c_1, c_2 , \tilde c_1 ,\tilde c_2 \ge 0 $ such that 
\begin{align}
\begin{split}
    \vert u\vert  &\le S^\ast(s)C(s)\\
    C_1(s)&= c_1 + \tilde c_1 (\tfrac s \eta )^{-\gamma }\\
    C_2(s)&= c_2 + \tilde c_2 (\tfrac s \eta )^{\gamma }\label{Cest}
\end{split}
\end{align}
on a maximal interval $[ \tilde s_0,s^\ast]$. We will establish some estimates on the $R_i$ depending on $c_i$ and $\tilde c_i$ and then we will determine specific $c_i$ and $\tilde c_i$ such that we prove that the maximal $s^\ast$ will be greater than $\tilde s_1 $.   Later it will be sufficient to choose $\tilde c_1=0$ on $I_1$ and $\tilde c_2 =0$ on $I_3 $. We thus deduce 
\begin{align*}
    \vert u_1(s)\vert &\le (c_1+\tilde c_2)\vert \tfrac s\eta\vert^{\gamma_1}+ (\tilde c_1+c_2)\vert \tfrac s\eta\vert^{\gamma_2}\\
    \vert u_2(s)\vert&\le (\tfrac {\gamma_1 }{c} c_1 +\tfrac {\gamma_2 }{c} \tilde c_2 ) \vert \tfrac s\eta\vert^{\gamma_1-1}+(\tfrac {\gamma_1 }{c} \tilde c_1 +\tfrac {\gamma_2 }{c} c_2 ) \vert \tfrac s\eta\vert^{\gamma_2-1}\\
    &\le  c_1^\ast  \vert \tfrac s\eta\vert^{\gamma_1-1}+c_2^\ast  \vert \tfrac s\eta\vert^{\gamma_2-1}. 
\end{align*}
where $ c_1^\ast =\tfrac {\gamma_1 }{c} c_1 +\tfrac {\gamma_2 }{c} \tilde c_2 $ and $c_2^\ast=\tfrac {\gamma_1 }{c} \tilde c_1 +\tfrac {\gamma_2 }{c} c_2  $. For sake of simplicity we will often omit absolute values for the estimates.

\begin{lemma}[$3$ mode forcing estimate ]\label{R3mode}
Let $u(s)=S(s) r(s)$ be a solution of \eqref{Echo} on $[\tilde s _0 , s^\ast] $, such that $\vert u(s)\vert \le S^\ast (s) C(s) $, then we estimate 
\begin{align*}
    R_1[F_{3mode}]&\le 20 c^2c_1+(20+c^4(\tfrac {s\wedge \tilde s_0}\eta)^{-\gamma} )\tilde c_1  +(20c^2+c^4(\tfrac {s\wedge \tilde s_0}\eta)^{-\gamma} )c_2+20c^4\tilde c_2 \\
    R_2[F_{3mode}]&\le 20( \tfrac {s\vee  \tilde s_0} \eta)^\gamma ( c_1 + 2c^2 \tilde c_2 ) + 20(\tilde c_1 +  c^2 c_2) .
\end{align*}

\begin{proof}
We have the forcing term
\begin{align*}
    F_{3mode}&= (\tfrac c \eta -a_1) e_1 u_2- 2c \eta \tfrac 1{s^2(s^2+1)}e_2 u_1.
\end{align*}
Therefore, we estimate 
\begin{align*}
    R_1[e_2 2c \eta \tfrac 1 {s^2(1+s^2)}u_1 ]&\le 2c^2 \eta^{\gamma_1 } \int_{ \tilde s_0}^s \tau^{\gamma_2-4 }( (c_1+\tilde c_2) (\tfrac \tau  \eta)^{\gamma_1 } + (\tilde c_1+c_2) (\tfrac \tau \eta)^{\gamma_2 })\\
    &\le c^4  (c_1+\tilde c_2)+ c^4 (\tfrac {s\wedge \tilde s_0 } \eta)^{-\gamma} (\tilde c_1+c_2) \\
     R_2[e_2 2c \eta \tfrac 1 {s^2(1+s^2)}u_1 ]&= 2c^2 \eta^{\gamma_2 } \int_{ \tilde s_0}^s \tau^{\gamma_1-4 }( (c_1+\tilde c_2) (\tfrac \tau  \eta)^{\gamma_1 } + (\tilde c_1+c_2) (\tfrac \tau \eta)^{\gamma_2 })\\
     &\le 2c^3 \eta^{-\gamma}   (c_1+\tilde c_2) + c^4 (\tilde c_1+c_2).
\end{align*}
By Taylor formula we obtain $\vert c\tfrac 1 \eta -a_1 \vert \le 18 c \tfrac {\vert s\vert } {\eta^2} $ and so 
\begin{align*}
    R_1[(c\tfrac 1 \eta -a_1) u_2 e_1  ]&\le  (1+10c^2)c^2 \eta^{1-\gamma_2 } \int \tau^{\gamma_2-1 } 18 c \tfrac \tau  {\eta^2}((\tfrac \tau \eta)^{\gamma_1-1}c_1^\ast + (\tfrac \tau \eta)^{\gamma_2-1}c_2^\ast )\\
    &\le 20 c^3 c_1^\ast +20c  c_2^\ast  \\
    &\le 20c^2c_1+20 \tilde c_1  + 20c^2 c_2+20 c^4\tilde c_2, \\
    R_2[(c\tfrac 1 \eta -a_1) u_2 e_1 ]&=(1+10c^2)\eta^{1-\gamma_1} \int \tau^{\gamma_1 -1 }18 c \tfrac \tau {\eta^2}((\tfrac \tau \eta)^{\gamma_1-1}c_1^\ast+c (\tfrac \tau \eta)^{\gamma_2-1}c_2^\ast  ))\\
    &\le 20c  \vert \tfrac {s\vee  \tilde s_0} \eta\vert ^\gamma c_1^\ast + 20c c_2^\ast  \\
    &\le   20\vert \tfrac {s\vee  \tilde s_0} \eta\vert^\gamma (c_1+2c^2 \tilde c_2) +20 \tilde c_1 + 20 c^2 c_2.
\end{align*}
\end{proof}
\end{lemma}

\begin{lemma}[$k_0$-th current estimate ]\label{Rpsi}
Let $u(s)=S(s) r(s)$ be a solution of \eqref{Echo} on $[\tilde s _0 , s^\ast]$  such that $\vert u(s)\vert \le S^\ast(s) C(s)$, then we estimate
\begin{align*}
    R_1[F_j]&\le\tfrac {c^3}\beta (c_1+\tilde c_2 )+  \tfrac {c^3}\beta  (\tfrac {s\wedge s_0} \eta)^{-\gamma}  (\tilde c_1+c_2 )+
    \left\{\begin{array}{cc}
         \tfrac {4 c^{2+\gamma_1 }}{\kappa_{k_0} \eta^{1+\gamma_2 }}j( k_0,\tilde s_0) & on \ I_1\\ 
         \eta^{\gamma_1 } \tfrac { c^{4+\gamma_1 }}{\kappa_{k_0} }j(k_0,\tilde s_0)& on \ I_3
    \end{array}\right. \\
    R_2[F_j]&\le \min  (\tfrac 1 \beta \tfrac 1 {2c^{1+\gamma }} \eta^{-\gamma }, \tfrac c \beta (\tfrac {s\vee\tilde s_0 }  \eta )^\gamma )  (c_1+\tilde c_2 ) +\tfrac 1 \beta c (c_2+\tilde c_1 ) \\
    &+\tfrac 1 \beta c  (c_2+\tilde c_1 ) +\left\{\begin{array}{cc}4 \tfrac {c^{\gamma_2 }}{\kappa_{k_0} \eta^{1+\gamma_1}}j(k_0,\tilde s_0)& on \ I_1\\ 
          \eta^{\gamma_2}\tfrac {c^{2+ \gamma_2 }}{\kappa_{k_0} }j(k_0,\tilde s_0)& on \ I_3
    \end{array}\right.  .
\end{align*}
Furthermore,  on $I_1$ we estimate
\begin{align*}
    \vert j(k_0, \tilde s_1 )\vert &\le \tfrac {2d^2}{\eta^2 } \exp(-\tfrac \kappa{2^5}\xi \eta^2) j(k_0,\tilde s_0)\\
   &+ \tfrac { c^2} \beta  \textbf{(}(c_1+\tilde c_2 ) (\tfrac {d } \eta )^{\gamma_1 } + ( c_2 + \tilde c_1 ) ( \tfrac {d } \eta )^ {\gamma_2 }\textbf{)}
\end{align*}
and on $I_3$ 
\begin{align*}
    \vert j(k_0,\tilde s_1 )\vert &\le c^2  \eta^2 \exp(-\kappa_{k_0}\eta^3) j (k_0 ,\tilde s_0)\\
    &+ 2 \tfrac {16^2}\beta\tfrac 1 {  \eta^2 }(c_1 +c_2 + \tilde c_1+  \tilde c_2 )
\end{align*}
\end{lemma}

\begin{proof}
The equation
\begin{align*}
     \partial_s j (k_0)  &=(\tfrac{2s}{1+s^2}-\kappa_{k_0}(1+s^2))j (k_0)+ u_1 
\end{align*}
leads to 
\begin{align*}
   j(k_0) &=\tfrac {1+s^2 }{1+s_0^2}\exp(-\kappa_{k_0} (s-s_0 +\tfrac 13 (s^3 -s_0^3)))j(k_0,\tilde s_0) \\
   &+\tfrac {\kappa_k }\beta  \int_{s_0}^s \ d\tau_2  \  \tfrac {1+s^2}{1+\tau_2^2 } \exp(-\kappa_{k_0} (s-\tau_2 +\tfrac 13 (s^3 -\tau_2 ^3))) u_1(\tau_2 )\\
   &= j_1 +j_2.
\end{align*}
Therefore, we estimate 
\begin{align*}
    R_1[F_{j_2} ] &= \tfrac {\kappa_{k_0} }\beta c^2 \eta^{1-\gamma_2}  \int_{s_0}^s \ d\tau_1   \int_{s_0}^{\tau_1 }\ d\tau_2 \ \tau_1^{\gamma_2 -1 } \tfrac {1+\tau_1 ^2}{1+\tau_2^2 } \exp(-\kappa_{k_0} (\tau_1-\tau_2 +\tfrac 13 (\tau_1^3 -\tau_2 ^3))) u_1(\tau_2 )\\
    &= \tfrac {\kappa_{k_0} }\beta c^2 \eta^{1-\gamma_2} \int_{s_0}^s \ d\tau_1   \int_{s_0}^{\tau_1 }\ d\tau_2 \ \tau_1^{\gamma_2 -1 } \tfrac {1+\tau_1 ^2}{1+\tau_2^2 } \\
    &\qquad \qquad \cdot \exp(-\kappa_{k_0} (\tau_1-\tau_2 +\tfrac 13 (\tau_1^3 -\tau_2 ^3))) ((c_1+\tilde c_2) (\tfrac {\tau_2 }\eta)^{\gamma_1 } + (c_2+\tilde c_1) (\tfrac {\tau_2 }\eta)^{\gamma_2 } )\\
    &\le\tfrac {1}\beta c^{2 }\eta^{1-\gamma_2 } \int^s_{s_0}  \ d\tau_2 \ \tfrac {((c_1+\tilde c_2) (\tfrac {\tau_2 }\eta)^{\gamma_1 } + (c_2+\tilde c_1) (\tfrac {\tau_2 }\eta)^{\gamma_2 } )\tau_2^{-\gamma_1 }}{1+\tau_2^2} \ \\
    &\qquad \qquad \cdot \int^s_{\tau_2 } d \tau_1 \kappa_{k_0} (1+\tau_1^2) \exp(-\kappa_{k_0} (\tau_1-\tau_2 +\tfrac 1 3 (\tau_1^3-\tau_2^3)))\\
    &= \tfrac {1}\beta c^{2 }\eta^{1-\gamma_2 } \int^s_{s_0}  \ d\tau_2 \ ((c_1+\tilde c_2) (\tfrac {\tau_2 }\eta)^{\gamma_1 } + (c_2+\tilde c_1) (\tfrac {\tau_2 }\eta)^{\gamma_2 } ) \tfrac {\tau_2^{-\gamma_1 }}{1+\tau_2^2} \\
    &\qquad \qquad \cdot \left[-\exp(-\kappa_k (\tau_1-\tau_2 +\tfrac 1 3 (\tau_1^3-\tau_2^3))) \right]_{\tau_1=\tau_2}^{\tau_1=s}  \\
    &\le \tfrac {1}\beta c^{2 }\eta^{1-\gamma_2 } \int^s_{s_0}  \ d\tau_2 \ (c_1+\tilde c_2) \eta^{-\gamma_1 }\tau_2^{-2}  + (c_2+\tilde c_1) \eta^{-\gamma_2 } \tau_2^{-\gamma-2} \\
    &\le (c_1+\tilde c_2 )  \tfrac {1}\beta c^{2 } \eta^{1-\gamma_2-\gamma_1 } [-\tau^{-1} ]_{s_0}^s+(c_2+\tilde c_1 )  \tfrac {1}\beta c^{2 } \eta^{\gamma  } [-\tau^{-\gamma -1} ]_{s_0}^s \\
    &\le  \tfrac {c^3}\beta (c_1+\tilde c_2 )+  \tfrac {c^3}\beta  (\tfrac {s\wedge s_0} \eta)^{-\gamma}  (\tilde c_1+c_2 )
\end{align*}
and 
\begin{align*}
    R_2[F_{j_2} ] &= \tfrac {\kappa_{k_0} }\beta \eta^{1-\gamma_1} \int_{s_0}^s \ d\tau_1   \int_{s_0}^{\tau_1 }\ d\tau_2 \ \tau_1^{\gamma_1 -1 } \tfrac {1+\tau_1 ^2}{1+\tau_2^2 } \exp(-\kappa_{k_0} (\tau_1-\tau_2 +\tfrac 13 (\tau_1^3 -\tau_2 ^3))) u_1(\tau_2 )\\
    &= \tfrac {\kappa_{k_0} }\beta \eta^{1-\gamma_1} \int_{s_0}^s \ d\tau_1   \int_{s_0}^{\tau_1 }\ d\tau_2 \ \tau_1^{-\gamma_2} \tfrac {1+\tau_1 ^2}{1+\tau_2^2 } \\
    &\qquad \qquad \cdot \exp(-\kappa_{k_0} (\tau_1-\tau_2 +\tfrac 13 (\tau_1^3 -\tau_2 ^3))) ((c_1+\tilde c_2) (\tfrac {\tau_2 }\eta)^{\gamma_1 } + (c_2+\tilde c_1) (\tfrac {\tau_2 }\eta)^{\gamma_2 } )\\
    &\le   \tfrac 1 \beta   \eta^{\gamma_2  }   \int_{s_0}^{s}\ d\tau_2 \ \tfrac {((c_1+\tilde c_2) (\tfrac {\tau_2 }\eta)^{\gamma_1 } + (c_2+\tilde c_1) (\tfrac {\tau_2 }\eta)^{\gamma_2 } )\tau_2^{-\gamma_2 } }{1+\tau_2^2 } \\
    &\qquad \qquad \cdot \int_{s_0}^{\tau_2 } \ d\tau_1 \  \kappa_k  (1+\tau_1 ^2) \exp(-\kappa_k (\tau_1-\tau_2 +\tfrac 13 (\tau_1^3 -\tau_2 ^3)))  \\
    &=\tfrac {1}\beta   \int_{s_0}^{s}\ d\tau_2 \ ((c_1+\tilde c_2) \tau_2^{\gamma-2 } \eta^{-\gamma }+ (c_2+\tilde c_1) \tau_2^{-2} \eta^{-\gamma_2 } )\\
    &\qquad \qquad \cdot [\exp(-\kappa_{k_0} (\tau_1-\tau_2 +\tfrac 13 (\tau_1^3 -\tau_2 ^3)))]_{\tau_1=\tau_2 }^{\tau_1=s}    \\
    &\le \tfrac {1}\beta   \int_{s_0}^{s}\ d\tau_2 \ ((c_1+\tilde c_2) \tau_2^{\gamma-2 } \eta^{-\gamma }+ (c_2+\tilde c_1) \tau_2^{-2} \eta^{-\gamma_2 } ).
\end{align*}    
We note that for the first term we obtain
\begin{align*}
    \tfrac {1}\beta   \int_{s_0}^{s}\ d\tau_2 \ \tau_2^{\gamma-2 } \eta^{-\gamma }&\le \min ( \tfrac c \beta  (\tfrac {s\vee\tilde s_0 }  \eta )^\gamma ,  \tfrac 1{\beta c } (c\eta)^{-\gamma }),
\end{align*}
since we can either integrate it directly or first pull out $s^\gamma$ and then integrate. Finally, we obtain the following estimate 
\begin{align*}  
     R_1[F_{j_2} ]&\le \min  (\tfrac 1 \beta \tfrac 1 {2c^{1+\gamma }} \eta^{-\gamma }, \tfrac c \beta (\tfrac {s\vee\tilde s_0 }  \eta )^\gamma )  (c_1+\tilde c_2 ) +\tfrac 1 \beta c (c_2+\tilde c_1 ).
\end{align*}
On $I_1$ we estimate the $j(k_0)$ influence by 
\begin{align*}
    R_1[F_{j_1}]&=c^2 \eta^{1-\gamma_2 } j(s_0)\int \tau^{\gamma_2-1 } \tfrac {1+\tau^2}{1+s_0^2}\exp(-\kappa_{k_0} (\tau-s_0+ \tfrac 1 3 ( \tau^3 -s_0^3)))\\
    &\le 4c^{2+\gamma_1 }\eta^{-1-\gamma_2 } j(s_0)\int (1+\tau^2 ) \exp(- \kappa_{k_0} (\tau-s_0 + \tfrac 1 3 ( \tau^3 -s_0^3)))\\
    &\le \tfrac {4 c^{2+\gamma_1 }}{\kappa_{k_0} \eta^{1+\gamma_2 }}j(s_0)
\end{align*}
and 
\begin{align*}
    R_2[F_{j_2}]&=\eta^{1-\gamma_1} j(s_0)\int \tau^{\gamma_1 -1 }\tfrac {1+\tau^2}{1+s_0^2}\exp(-\kappa_{k_0} (\tau-s_0+ \tfrac 1 3 ( \tau^3 -s_0^3)))\\
    &=4\eta^{-1-\gamma_1 }c^{\gamma_2 }j(s_0) \int (1+\tau^2 ) \exp(- \kappa_{k_0} (\tau-s_0 + \tfrac 1 3 ( \tau^3 -s_0^3)))\\
    &=4 \tfrac {c^{\gamma_2 }}{\kappa_{k_0} \eta^{1+\gamma_1}}j(s_0). 
\end{align*}
We estimate $j(k_0)$ by 
\begin{align*}
    j(k_0,s)&= \tfrac {1+s^2}{1+\tilde s_0^2 } \exp(-\kappa_{k_0} (s-\tilde s_0+\tfrac 1 3 (s^3-\tilde s^3_0 )))j(k_0.\tilde s_0) \\
    &+ \tfrac {\kappa_{k_0} }\beta \int^s_{\tilde s_0} d\tau \ \tfrac {1+s^2}{1+\tau^2 } \exp(-\kappa_{k_0} (s-\tau+\tfrac 1 3 (s^3-\tau^3 )))u_1(\tau )  \\
    &\le  \tfrac {4d^2}{\eta^2 } \exp(-\tfrac \kappa{2^5}\xi \eta^2) j(k_0,\tilde s_0) +\tfrac {c^2} \beta  \textbf{(}(c_1+\tilde c_2 ) (\tfrac {d } \eta )^{\gamma_1 } + ( c_2 + \tilde c_1 ) ( \tfrac {d } \eta )^ {\gamma_2 }\textbf{)}.
\end{align*}

On $I_3$ we estimate the $j(k_0)$ influence by 
\begin{align*}
    R_1[F_{j_1}]&=c^2 \eta^{1-\gamma_2 } \int \tau^{\gamma_2-1 } \tfrac {1+\tau^2}{1+s_0^2}\exp(-\kappa_{k_0} (\tau-s_0+ \tfrac 1 3 ( \tau^3 -s_0^3)))j(s_0)\\
    &\le c^{4+\gamma_1 }\eta^{\gamma_1 } \int (1+\tau^2 ) \exp(- \kappa_{k_0} (\tau-s_0 + \tfrac 1 3 ( \tau^3 -s_0^3)))j(s_0)\\
    &\le \eta^{\gamma_1 } \tfrac { c^{4+\gamma_1 }}{\kappa_{k_0} }j(s_0)
\end{align*}
and 
\begin{align*}
    R_2[F_{j_2}]&=\eta^{1-\gamma_1} \int \tau^{\gamma_1 -1 }\tfrac {1+\tau^2}{1+s_0^2}\exp(-\kappa_{k_0} (\tau-s_0+ \tfrac 1 3 ( \tau^3 -s_0^3)))j(s_0)\\
    &=\eta^{\gamma_2 }c^{2+\gamma_2 } \int (1+\tau^2 ) \exp(- \kappa_{k_0} (\tau-s_0 + \tfrac 1 3 ( \tau^3 -s_0^3)))j(s_0)\\
    &= \eta^{\gamma_2}\tfrac {c^{2+ \gamma_2 }}{\kappa_{k_0} }j(s_0).
\end{align*}
Next we want to estimate the evolution of $j(k_0)$
\begin{align*}
j(k_0,\tilde s_1 )  &= \tfrac {1+\tilde s_1 ^2}{1+d^2} \exp( -\kappa_{k_0} (\tilde s_1-d+\tfrac 13 ( \tilde s_1^3-d^3)))j(k_0,\tilde s_0 ) \\
&+\tfrac {\kappa_{k_0} }\beta  \int_{d}^{\tilde s_1} \ d\tau_2  \  \tfrac {1+s^2}{1+\tau_2^2 } \exp(-\kappa_k (s-\tau_2 +\tfrac 13 (s^3 -\tau_2 ^3))) ((c_1+\tilde c_2 ) (\tfrac {\tau_2 } \eta )^{\gamma_1 } + ( c_2 + \tilde c_1 ) ( \tfrac {\tau_2 } \eta )^ {\gamma_2 }).
\end{align*}
Therefore, we deduce 
\begin{align*}
\tfrac {\kappa_{k_0} }\beta  \int_{d}^{\tilde s_1} \ d\tau_2  \ & \tfrac {1+s^2}{1+\tau_2^2 } \exp(-\kappa_{k_0} (s-\tau_2 +\tfrac 13 (s^3 -\tau_2 ^3)))  ( \tfrac {\tau_2 } \eta )^ {\gamma_i }\\
   &\le \tfrac {\kappa_{k_0} }\beta  \left(\int_{d}^{\frac 12 \tilde s_1}+ \int_{\frac 12 \tilde s_1}^{ \tilde s_1}\right) \tfrac {1+s^2}{1+\tau_2^2 } \exp(-\kappa_{k_0} (s-\tau_2 +\tfrac 13 (s^3 -\tau_2 ^3)))  ( \tfrac {\tau_2 } \eta )^ {\gamma_i }\\
   &\le  \tfrac {\kappa_{k_0} }\beta  {\eta^{-\gamma_i}}  \tfrac {c^{1-\gamma_i}}{\gamma_i}(1+\eta^2)\exp(- \tfrac {\kappa_{k_0}}  {25} \eta^3) +  \tfrac {\kappa_{k_0} }\beta \tfrac {2^4 } {\kappa_{k_0} \eta^2 }\\
   &\le
   \tfrac {2^5}\beta\tfrac 1 {  \eta^2 },
\end{align*}
which leads to 
\begin{align*}
    \vert j(k_0,\tilde s_1 )\vert &\le c^2 \kappa_{k_0}  \eta^2 \exp(-\kappa_{k_0}\eta^3) j (k_0 ,\tilde s_0)\\
    &+ \tfrac {2^5}\beta\tfrac 1 {  \eta^2 }(c_1 +c_2 + \tilde c_1+  \tilde c_2 ).
\end{align*}
\end{proof}

\begin{lemma}[Forcing estimate ]\label{Rrest}
Let $u(s)=S(s) r(s)$ be a solution of \eqref{Echo} on $[\tilde s _0 , s^\ast ]$  such that $\vert u(s)\vert \le S^\ast (s) C(s)$. We define for $\vert n\vert \ge 2 $
\begin{align*}
    \tilde w(n) &= 2 \sum_{\vert m\vert \ge 2 } (2c)^{\vert m-n\vert+\chi} (w+    \tfrac {4} {   \kappa \xi  }   j)(k_m, \tilde s_0)  \\
    &+ (2c)^{\vert \vert m\vert -2 \vert}c (2 c_1^\ast + \tfrac 1 {c^2}c_2^\ast ) \\
    &+( 2c)^{\vert m\vert -1 }(u_3( \tilde s_0) +\tfrac 2{  \kappa \xi\eta } (j(k_{\pm1},\tilde s_0 ))
\end{align*}
where $ \chi=\chi(m,n)  =-\vert \sgn(m)-\sgn(n)\vert $. Then we estimate
\begin{align*}
    R_1[F_{\tilde w }]&= 2c^2 (\tilde w(2)+\tilde w(-2)) \\
    R_2[F_{\tilde w }]&= c^2 (\tilde w(2)+\tilde w(-2))(\tfrac {s\vee  \tilde s_0}\eta)^{\gamma_1 }
\end{align*}
and
\begin{align*}
    R_1[F_{j(k_{\pm 1} )}]&= \tfrac {2c}{\kappa \xi \eta }  j(k_{\pm 1},  \tilde s_0 )  +  \tfrac {2c} {\beta \kappa \xi  } (\tilde w(1) + c_1^\ast + c_2^\ast ) \\
    R_2[F_{j(k_{\pm 1})}]&= \tfrac {2c} {\kappa \xi \eta }  j(k_{\pm 1},  \tilde s_0 )(\tfrac {s\vee  \tilde s_0}\eta)^{\gamma_1 } + \tfrac c {\beta \kappa \xi  } (\tilde w(1) + c_1^\ast + c_2^\ast ) (\tfrac {s\vee  \tilde s_0}\eta)^{\gamma_1 }  
\end{align*}
and
\begin{align*}
    R_1[ F_{u_3}]&=  2c\tilde w(1) \\
    R_2[ F_{u_3}]&=  2c\tilde w(1) (\tfrac {s\vee  \tilde s_0}\eta)^{\gamma_1 }.
\end{align*}
Furthermore, we estimate 
\begin{align*}
    \vert w(k_n,s) \vert &\le \tilde w(n)& \vert n \vert \ge 2  \\
    \vert u_3\vert &\le \tilde w(1) = \tilde w(-1) \\
    \vert j(k_n )\vert &\le 2 e^{-\tfrac 1 2 \kappa \xi\eta  ( s -  \tilde s_0 )}j(k_n,  \tilde s_0 )+4\tfrac {1 }  {\beta   \eta^2  } \tilde w(n) .
\end{align*}

\begin{proof}
To estimate $w(k_n,s)$ we without loss of generality assume that $n\ge 2$. We begin with the case $n\ge 3$, where we deduce that 
\begin{align*}
    \partial_s w(k_n)&= a(k_{n+1})w(k_{n+1})- a(k_{n-1}) w(k_{n-1})-  j(k_n)\\
    &\le \tfrac {c }\eta (\tilde w(n-1)+\tilde w(n+1))+2 e^{-\kappa \xi\eta  ( s -  \tilde s_0 )}j(k_0+n,  \tilde s_0 )+4\tfrac {1 }  {\beta  \eta^2 } \tilde w(n)  .
\end{align*}
We estimate  
\begin{align*}
    2  \int e^{-\frac 1 2 \kappa \xi\eta (\tau- \tilde s_0)}j(k_0+n, \tilde s_0 )  &\le 4   \tfrac {1} {\kappa \xi\eta} j(k_n, \tilde s_0 ) .
\end{align*}
Thus integrating $\partial_s w(k_0+n)$ over time yields 
\begin{align*}
    w(k_n)&\le w(k_n,\tilde s_0)+ c  (\tilde w(n-1)+\tilde w(n+1))+4   \tfrac {1} {\kappa \xi\eta} j(k_0+n, \tilde s_0 )+\tfrac {4 }  {\beta  \kappa \xi } \tilde w(n) \\
    &< \tilde w(n).
\end{align*}
For the case $n=2 $ we deduce 
\begin{align*}
    \partial_s w(k_2)&= a(k_3)w(k_3)- a(k_1) \tfrac 1 2 ( u_3 + u_2) -  j(k_2)\\
    &\le \tfrac {c }\eta (\tilde w(3)+2\tilde w(1)  +2 c_1^\ast(\tfrac s \eta)^{\gamma_1 -1 }  + 2 c_2^\ast(\tfrac s \eta)^{\gamma_2-1 }   )+ 2 e^{-\frac 1 2 \kappa \eta \xi (s- \tilde s_0)}j(k_2, \tilde s_0) + 4\tfrac {1 }  {\beta  \kappa \xi\eta } \tilde w(2)  \\
    w(k_2)&\le w(k_2, \tilde s_0)+ \tfrac 4 { \kappa \xi  \eta}  j(k_2, \tilde s_0)+ c(\tilde w(3)+2\tilde w(1) + 2c_1^\ast + \tfrac 1{c^2} c_2^\ast  )+ 4\tfrac {1 }  {\beta  \kappa \xi } \tilde w(2)\\
    &< \tilde w(2) .
\end{align*}
We estimate $u_3$ by 
\begin{align*}
    \partial_s u_3 &= a(k+2)w(k_2) - a(k-2 ) w(k_{-2} ) - j(k_1)+  j(k_{-1}  ) \\
    &\le \tfrac {2c} \eta (w(2)+ w(-2) )+2 e^{-\frac 1 2  \kappa \xi\eta  ( s -  \tilde s_0 )}j(k_{\pm1},  \tilde s_0 )+ \tfrac 4 {\beta   \kappa \xi\eta }  (\tilde w(1) + c_1^\ast +c_2^\ast ) \\
    \vert u_3\vert &\le \vert u_3( \tilde s_0) \vert +  c(\tilde w(2) + \tilde w(-2))+ \tfrac 4 {\kappa \xi\eta } (j(k_{\pm1},\tilde s_0 )+ \tfrac 1\beta \tilde w(1) + \tfrac 1\beta c_1^\ast +\tfrac 1 {\beta } c_2^\ast  )\\
    &\le \tilde w(1). 
\end{align*}
Non-resonant $j$ will often be estimated similarly. Therefore we will use the following notation frequently. We estimate  $j(k_n) $ for $n\ge 2$  by writing $\hat s = s-\tfrac {k_0(k_0-k)}{k+1} \eta $ and $\hat \tau = \tau - \tfrac {k_0(k_0-k)}{k+1} \eta$ 
\begin{align*}
    \partial_s j(k_n)&= -\kappa_{k_n} (1+\hat s^2 ) j(k_n)+2 \tfrac {\hat s }{1+\hat s^2 }j(k_n)+\tfrac 1 \beta \kappa_{k_n}w(k_n)
\end{align*}
which gives 
\begin{align*}
    j(k_n) &\le \tfrac {1+\hat s^2} {1+\hat { \tilde s }_0^2} e^{-\kappa_{k_n}  ((\hat s -   \hat {\tilde s}_0 + \frac 1 3 (\hat s^3 - \hat  {\tilde s}_0^3 ) )}j(k_n,  \tilde s_0 )\\
    &+ \tfrac 1 \beta \kappa_{k_n}   \int \ d\tau \ \tfrac{1+\hat s^2} {1+\hat \tau^2 } e^{-\kappa_{k_n} ((\hat s - \hat \tau + \frac 1 3 (\hat s^3 - \hat \tau^3 ) )}\tilde w(n)  \\
\end{align*}
For $  \tilde s_0 \le   \tau \le   s\le  \tilde s_1$ we obtain 
\begin{align*}
    \kappa_{k_n} (\hat s -\hat \tau+\tfrac 1 3 ( \hat s^3 - \hat \tau^3))&=\kappa_{k_n}\tfrac 1 3 ( s - \tau)(\hat s^2+ \hat s \hat \tau + \hat \tau^2 +1) \\
&\ge \tfrac 1 2  \kappa\max({k_n^2},k_0^2) \eta^2 (s-\tau) \\
    \tfrac {1+\hat s^2}{1+\hat \tau}&\le 2.
\end{align*}
So we infer 
\begin{align*}
    j(k_n) 
    &\le2 e^{-\frac 12  \kappa \xi\eta  (s-\tilde s_0)}j(k_n,  \tilde s_0 )\\
    &+2 \kappa_{k_n}\tfrac 1\beta  \int \ d\tau \   e^{-\frac 12  \kappa \xi\eta  (s-\tau)}\tilde w(n) \\
    &\le 2 e^{-\tfrac 1 2 \kappa \xi\eta  (s-\tilde s_0)}j(k_n,  \tilde s_0 )+\tfrac {4}  {\beta \eta^2 } \tilde w(n). 
\end{align*}

We next turn to the estimate of $j(k_{\pm 1} )$, where we without loss of generality consider $j(k_1)$. With the equation
\begin{align*}
    \partial_s j(k_1)&= (2 \tfrac {\hat s }{1+\hat s^2 }-\kappa_{k_1 } (1+\hat s^2 )) j(k_1)+\tfrac {\kappa_{k_1}}{2\beta} \kappa_{k_1} (u_3 +u_2 )
\end{align*}
we estimate 
\begin{align*}
    j(k_1) 
    &\le2 e^{-\frac 12   \kappa \xi\eta ( s -   \tilde s_0 )}j(k_1,  \tilde s_0 )\\
    &+ \tfrac {\kappa_{k_1}}{2\beta}  \int \ d\tau \  e^{-\frac 12   \kappa \xi\eta ( s -  \tau )}(\tilde w(1) + c_1^\ast (\tfrac \tau \eta)^{\gamma_1-1}+c_2^\ast  (\tfrac \tau \eta)^{\gamma_2-1})  \\
    &\le 2 e^{-\frac 1 2  \kappa \xi\eta  ( s -  \tilde s_0 )}j(k_1,  \tilde s_0 )+ \tfrac 2 {\beta  \eta^2 }  (\tilde w(1) + c_1^\ast +c_2^\ast ).
\end{align*}
Given these estimates, we next consider the effects on $R[\cdot]$ by forcing:
\begin{align*}
    F_{\tilde w}&= e_1 (a(k_2)w(k_2)+a(k_{-2})w(k_{-2})\\
    &\le e_1 \tfrac {c} \eta \tilde w(2) +e_1\tfrac {c} \eta \tilde w(-2) .
\end{align*}
For constant $e_2$ functions we estimate
\begin{align*}
    R_1[e_2]&\le \tfrac {c }2\eta  \\
    R_2[e_2]&\le \tfrac c 3  \eta (\tfrac {s\vee  \tilde s_0}\eta)^{\gamma_1 }.
\end{align*}
Therefore, we can control $F_{\tilde w}$ by 
\begin{align*}
    R_1[F_{\tilde w}]&= c^2 (\tilde w(2)+\tilde w(-2)) \\
    R_2[F_{\tilde w}]&= c^2 (\tilde w(2)+\tilde w(-2))(\tfrac {s\vee  \tilde s_0}\eta)^{\gamma_1 }.
\end{align*}
For $F_{j(k_{\pm 1})}$ we use 
\begin{align*}
    F_{j(k_{\pm 1})}&= -e_2 j(k_{\pm 1}) \\
    &\le e_2 ( 2 e^{-\frac 1 2  \kappa \xi\eta  ( s -  \tilde s_0 )}j(k_{\pm1},  \tilde s_0 )+ \tfrac 2 {\beta   \kappa \xi\eta }  (\tilde w(1) + c_1^\ast +c_2^\ast )) ,
\end{align*}
to estimate
\begin{align*}
    R_1[F_{j(k_{\pm 1} )}]&= \tfrac {2c}{\kappa \xi \eta }  j(k_{\pm 1},  \tilde s_0 )  +  \tfrac {2c} {\beta \kappa \xi  } (\tilde w(1) + c_1^\ast + c_2^\ast ) \\
    R_2[F_{j(k_{\pm 1})}]&= \tfrac {2c} {\kappa \xi \eta }  j(k_{\pm 1},  \tilde s_0 )(\tfrac {s\vee \tilde s_0}\eta)^{\gamma_1 } + \tfrac c {\beta \kappa \xi  } (\tilde w(1) + c_1^\ast + c_2^\ast ) (\tfrac {s\vee  \tilde s_0}\eta)^{\gamma_1 } .
\end{align*}
Furthermore, for $F_{u_2}$ we estimate
\begin{align*}
    F_{u_3}&= e_1 a_2 u_3 \\
    &\le e_1 \tfrac c \eta  \tilde w(1)
\end{align*}
and 
\begin{align*}
    R_1[e_1]&\le \eta \\
    R_2[e_1]&\le   \eta (\tfrac {s \vee \tilde s_0}\eta)^{\gamma_1 }.
\end{align*}
to deduce 
\begin{align*}
    R_1[ F_{u_3}]&\le  c\tilde w(1) \\
    R_2[ F_{u_3}]&\le  c\tilde w(1) (\tfrac {s\vee \tilde s_0}\eta)^{\gamma_1 }.
\end{align*}
\end{proof}
\end{lemma}

\begin{proof}[Proof of Proposition \ref{I1pro}]
For the interval $I_1$ we have $\tilde s_0=s_0$, $\tilde s_1 =-d$. The initial data of $r$ can be calculated by  $r( \tilde s_0)= S^{-1}(s_0)u(s_0 ) $ and so 
\begin{align*}
    r_1( \tilde s_0)&= -\tfrac {\gamma_2 }\gamma (\tfrac {k_0}{2(k_0+1)})^{\gamma_2-1}u_1( \tilde s_0) +  \tfrac c \gamma (\tfrac {k_0}{2(k_0+1)})^{\gamma_2 }   u_2( \tilde s_0) \\
    &\approx -4c^2u_1( \tilde s_0)+c u_2( \tilde s_0),\\
    r_2( \tilde s_0)&= \tfrac {\gamma_1 }\gamma (\tfrac {k_0}{2(k_0+1)})^{\gamma_1-1}u_1( \tilde s_0) - \tfrac c\gamma  (\tfrac {k_0}{2(k_0+1)})^{\gamma_1 }   u_2( \tilde s_0) \\
    &\approx u_1( \tilde s_0) -\tfrac c 2  u_2( \tilde s_0).
\end{align*}
For other initial data we define 
\begin{align*}
   N  &= \sum_{\vert m \vert \ge 2 } (2c)^{\vert m \vert   }(w+\tfrac 8 {\kappa \eta \xi  }   j)(k_m, \tilde s_0) \\
   &+ 2  c  (u_3( \tilde s_0)+\tfrac 8 {\kappa \eta \xi  }   j(k_{\pm 1 }, \tilde s_0))\\
   &+ \tfrac {2c} {\xi \kappa} j(k_0,s_0),
\end{align*}
to bound the impact of the less important terms in the following bootstrap.  Let $C(s)$ be defined by the terms 
\begin{align*}
    c_1&= 45c^2 u_1(s_0) +2c u_2(s_0 )+ 2N,\\
    \tilde c_1 &= 2\tfrac {c^3}{\beta \vee 1 } c_2, \\
    c_2 &=  2u_1(s_0)+ 45c u_2(s_0)+2N,\\
    \tilde c_2&=0.
\end{align*}
AS $c_1>r_1(\tilde s_0)$ and $c_2>r_2(\tilde s_0)$ and we have a smooth solution, the estimate $\vert u \vert \le  S^\ast(s) C(s)$  holds at least for a small time. Let $s^\ast$ be the maximal time such that  $\vert u \vert \le  S^\ast (s) C(s)$. We then aim to show that necessarily $s^\ast \geq -d$, since otherwise the estimate improves, which contradicts the maximality. By Lemma \ref{R3mode}, Lemma \ref{Rpsi} and Lemma \ref{Rrest} we estimate 
\begin{align*}
     R_1[F_{all}  ]&= R_1[F_{3mode}]+R_1[F_{j}]+ R_1 [F_{\tilde w }]+R_1[F_{j(k_0\pm 1)}]+R_1[F_{u_3}]\\
     &=20 c^2c_1 +(20+c^4(\tfrac s \eta)^{-\gamma} )\tilde c_1  +(20c^2+c^4(\tfrac s\eta)^{-\gamma} )c_2\\
     &+\tfrac {c^3}\beta c_1+  \tfrac {c^3}\beta  (\tfrac s  \eta)^{-\gamma}  (\tilde c_1+c_2 )+\tfrac {4 c^{2+\gamma_1 }}{\kappa_{k_0} \eta^{1+\gamma_2 }}j(k_0, s_0)\\
     &+ 2c^2 (\tilde w(2)+\tilde w(-2))\\
     &+\tfrac {2c}{\kappa \xi \eta }  j(k_{\pm 1},  \tilde s_0 )  +  \tfrac {2c} {\beta \kappa \xi  } (\tilde w(1) + c_1^\ast + c_2^\ast ) \\
     &+2c\tilde w(1) \\
     &< 21c^2 c_1 + \tilde c_1 (21 +\tfrac {c^3}\beta (\tfrac s\eta )^{-\gamma})+ c_2 ( 21c^2 + \tfrac {c^3}\beta  (\tfrac s \eta)^{-\gamma} )+ N
\end{align*}
and 
\begin{align*}
    R_2[F_{all}  ]&= R_2[F_{3mode}]+R_2[F_{j}]+ R_2[F_{\tilde w }]+ R_2[F_{j(k_0\pm 1)}]+R_2[F_{u_3}]\\
     &=20c_1  + 20\tilde c_1 + 20 c^2 c_2 \\
     &+\tfrac c \beta  c_1+\tfrac c \beta (c_2+\tilde c_1 ) +4 \tfrac {c^{\gamma_2 }}{\kappa_{k_0} \eta^{1+\gamma_1}}j(k_0, s_0) \\
     &+ c^2 (\tilde w(2)+\tilde w(-2))\\
     &+ \tfrac {2c} {\kappa \xi \eta }  j(k_{\pm 1},  \tilde s_0 ) + \tfrac c {\beta \kappa \xi  } (\tilde w(1) + c_1^\ast + c_2^\ast ) \\
     &+2c\tilde w(1) \\
     &<  21 c_1+ 21 \tilde c_1 + \tfrac c {\beta\wedge  1}  c_2 + N.
\end{align*}
We split $R_1$ as
\begin{align*}
    R_1[all ]&=R_1[all ][1]+R_1[all ][(\tfrac s \eta)^\gamma ],
\end{align*}
into the part with and without a $(\tfrac s \eta)^\gamma $ term, respectively. We then estimate
\begin{align*}
    r_1( \tilde s_0) + R_1[all ][1]&<r_1( \tilde s_0)+  21c^2  c_1 +21 \tilde c_1 + 21 c^2 c_2 +N, \\
    R_1[all ][(\tfrac s \eta)^\gamma ]&< \tfrac {c^3}{1\wedge \beta} \tilde c_1 + \tfrac {c^3}{1\wedge \beta} c_2, \\
    r_2( \tilde s_0) + R_2[all ][1]&< r_2( \tilde s_0)+ 21 c_1+ 21 \tilde c_1 + 2\tfrac c {1\wedge \beta} c_2 + N,
\end{align*}
and thus we conclude the bootstrap that
\begin{align*}
    r_1( \tilde s_0) + R_1[all ][1]&< c_1  \\
    R_1[all ][(\tfrac s \eta)^\gamma ]&< \tilde c_1 \\
    r_2( \tilde s_0) + R_2[all ][1] &<  c_2 .
\end{align*}
We can therefore extend the estimates past the time $s^\ast$, which contradicts the maximally. Therefore, we obtain that for all times $s\leq -d$ it holds that
\begin{align*}
    \vert u(s) \vert \le  S^\ast(s) C(s),
\end{align*}
which yields the upper bound 
\begin{align*}
    \vert u(-d) \vert &\le \left(
    \begin{array}{c}
         (c\eta)^{-\gamma_1}c_1+ (\tilde c_1 +c_2 )(c\eta)^{-\gamma_2 } \\
         \tfrac 1 c (c\eta)^{1-\gamma_1}c_1+ (\tfrac 1 c \tilde c_1 +c c_2 )(c\eta)^{1-\gamma_2 }
    \end{array}\right)\\
    &\le 2 M \left( \begin{array}{c}
         (c\eta)^{-\gamma_2 }\\
         (c\eta)^{1-\gamma_2 }
    \end{array}\right).
\end{align*}
We next aim to establish an estimate on $\tilde w(n)$. For this purpose we note that 
\begin{align*}
    c (2 c_1^\ast + \tfrac 1 {c^2}c_2^\ast )&\approx 2c_1 +\tfrac 1 {c^2} \tilde c_1 + 2c_2 \\
    &\approx 2c_1 +2c_2 \\
    &\le 4 u_1(s_0) +  100 c u_2(s_0) + 3 N
\end{align*}
and 
\begin{align*}
    \tilde w(n) &\le 2 \sum_{\vert m\vert \ge 2 } (2c)^{\vert m-n\vert+\chi} (w+    \tfrac {4} {\kappa \xi \eta }   j)(k_m, \tilde s_0)  \\
    &+ (2c)^{\vert \vert n\vert -2 \vert}c (2 c_1^\ast + \tfrac 1 {c^2}c_2^\ast )  \\
    &+( 2c)^{\vert n\vert -1 }(u_3( \tilde s_0)+\tfrac 2{\kappa \xi \eta}j(k_{\pm 1} ,\tilde s_0 ) ) .
\end{align*}
We hence deduce that 
\begin{align*}
    \tilde w(n) &\le 2 \sum_{\vert m\vert \ge 2 } (2c)^{\vert m-n\vert+\chi} (w+    \tfrac {4} {\kappa \xi \eta }   j)(k_m, \tilde s_0)  \\
    &+ (2c)^{\vert \vert n\vert -2 \vert}( 4 u_1(s_0) +  100 c u_2(s_0) + 3N ) \\
    &+( 2c)^{\vert n\vert -1 }(u_3( \tilde s_0)+\tfrac 2{\kappa \xi \eta}j(\tilde s_0, k_{\pm 1} )   ) \\
    &\le 2  M_n, 
\end{align*} 
when $\chi = -\vert \sgn(m)-\sgn(n)\vert $. To prove \eqref{echo_I1} under the condition \eqref{addI1} we estimate 
\begin{align*}
    \vert u(-d)-\tilde u(-d) \vert &\le S(-d)  R[all]\\
    &\le   u_1( \tilde s_0)\left( \begin{array}{c}
         (c\eta)^{-\gamma_2 }\\
         5c (c\eta)^{\gamma_1 }
    \end{array}
    \right).
\end{align*}

Furthermore, we use 
\begin{align*}
    \tilde u (-d) &= \left( 
    \begin{array}{cc}
        (c\eta)^{-\gamma_1 }&(c\eta)^{-\gamma_2 }\\
        -\tfrac {\gamma_1 }{2c}(c\eta)^{1-\gamma_1 }&-\tfrac {\gamma_2 }{2c}(c\eta)^{1-\gamma_2 }
    \end{array}
    \right)
    \left( 
    \begin{array}{cc}
    4c^2u_1( \tilde s_0)-2c u_2( \tilde s_0)\\
    u_1( \tilde s_0) +c u_2( \tilde s_0)
    \end{array}
    \right)\\
    &\approx   u_1( \tilde s_0)\left( 
    \begin{array}{cc}
    (c\eta)^{-\gamma_2} \\
    O(c) (c\eta)^{\gamma_1}
    \end{array}\right)
\end{align*}
and thus
\begin{align*}
     \vert u_1 (-d)-(c\eta)^{-\gamma_2}u_1( \tilde s_0)\vert      &=10c   u_1( \tilde s_0)(c\eta)^{-\gamma_2}  \\
     \vert u_2 (-d)\vert &\le 10c u_1( \tilde s_0)(c\eta)^{\gamma_1}.
\end{align*}
The remaining terms can be estimated by 
\begin{align*}
    M&\le \tfrac 1 {1-10^{-1} } u_1( \tilde s_0),\\
    M_n&\le \tfrac 4 {1-10^{-1} } u_1( \tilde s_0).
\end{align*}

\end{proof}

\subsection{The Resonance and Upper Bounds in $I_2$}  \label{SecI2}

The bounds on the evolution of \eqref{Echo} on the interval $I_2=[-d,d]$ are summarized in the following proposition:
\begin{pro}\label{I2_pro} 
Let $c\le \min ( (8\pi)^{-\frac 4 3 } \beta^{\frac {16} 3 }, 10^{-4} )$. Consider a solution of \eqref{Echo} on the interval $I=[s_0,d]$, then it holds that
\begin{align*}
    \vert u_1(d)\vert 
    &\le 3 (c\eta)^{-\gamma_2}LM,\\
    \vert u_2(d) \vert 
    &\le 7\pi   ( c \eta )^{\gamma_1 }LM,\\
    \vert u_3 (d)\vert
    &\le  7\pi  (\tfrac 5 \eta)^2  ( c \eta )^{\gamma_1 }LM +  2 M_1,\\
    \vert w(k_n,d)\vert  
    &\le  7\pi  (\tfrac 5 \eta)^{\vert n\vert-1 }  ( c \eta )^{\gamma_1 }LM+2 M_n,\\
        \vert j (k_n,d)\vert  
    &\le \tfrac 4 {\beta \eta^2 }  (7\pi   (\tfrac 5 \eta)^{\vert n\vert-1 }  ( c \eta )^{\gamma_1 }LM+2 M_n ),\\
       \vert  j(k_0,d) \vert 
    &\le \tfrac {4}\beta \min ( \kappa_{k_0} \pi d^2 , 1)  (c\eta)^{-\gamma_2}L M .
\end{align*}
\end{pro}
For interval $I_2$ we are mostly concerned with the interaction between $j (k_0)$ and $u_1$ and in particular the growth this induces for $u_2 $. Therefore, consider the ODE system 
\begin{align}
\begin{split}
    \partial_s u_1 &= -j(k_0)+F \\
    \partial_s j(k_0) &= \tfrac {\kappa_{k_0}} \beta u_1 +(\tfrac {2s} {1+s^2} - \kappa_{k_0} (1+s^2) )j(k_0),\label{PDEI2}
\end{split}
\end{align}
our aim is to bound the growth of $j  (k_0)$ and $u_1$ by a factor. Let $U(\tau,s)$ be the solution of \eqref{PDEI2} with initial data $u_1(\tau)=1$ and $j(\tau)=0 $ and $L$  as the constant which satisfies 
\begin{align}
    \vert U(\tau,s )\vert \le L=L (\beta, \kappa,k ). \label{UL}
\end{align}
With the restriction
\begin{align*}
    c&\le (8\pi)^{-\frac 4 3 } \beta^{\frac {16} 3 },
\end{align*}
$L$ is estimated by the following two cases, if $\beta\ge \pi$ we obtain $L=1$ and if $\beta <\pi $ we obtain a $L=L(\alpha ,\kappa, k)\le \sqrt c $. A proof and more specific bounds can be found in Appendix \ref{LApp} and for simplicity of presentation we here only consider two cases.

\begin{lemma}\label{L_lem}
Let $u_1$ be a solution of \eqref{PDEI2}  on $[-d,d]$such that \eqref{UL} holds, then we estimate  
\begin{align*}
    \tfrac 1 L \vert u_1 \vert &\le u_1(-d )+ \int_{-d}^s  \vert F(\tau)\vert  \ d\tau + \vert j(-d)\vert \int \tfrac {1+\tau^2}{1+d^2} \exp(-\kappa_k (\tau+d+\tfrac 1 3(s^3+d^3 )))  .
\end{align*}
\begin{proof}
We may without loss of generality restrict to the case $j(-d)=0$, since we can choose $\tilde F =F +  \tfrac {1+s^2}{1+d^2} \exp(-\kappa_k (\tau+d+\tfrac 1 3(s^3+d^3
)))j(-d) $. By Duhamel' principle the equation \eqref{PDEI2} is solved by 
\begin{align*}
    u_1(s) &= U(-d, s) u_1(-d)  + \int U(\tau,s ) F(\tau )
\end{align*}
which yields the desired bound. 
\end{proof}
\end{lemma}

\begin{proof}[Proof of Proposition \ref{I2_pro}]
With Proposition \ref{I1pro} we estimate until time $-d$
\begin{align*}
    \vert u_1  \vert (-d) &\le  2 M (c \eta)^{-\gamma_2 }\\
    \vert u_2  \vert (-d) &\le  2 M (c \eta)^{\gamma_1 },\\
    \vert u_3 \vert(-d) &\le 2  M_1, \\
    \vert w(k_n , -d)\vert &\le 2  M_n, \\
    \vert j\vert(k_0,-d) &\le \tfrac {c}\beta M(c\eta)^{-\gamma_2 }\min ( \kappa_{k_0}c^{-2}, 1 ) ,\\
    \vert j\vert(k_{\pm 1},-d)&\le \tfrac 4{\beta \eta^2 }  M,\\
    \vert j\vert(k_n,-d)&\le \tfrac 4{\beta \eta^2 }  M_n .
\end{align*} 
We next aim to prove by a bootstrap  that
\begin{align*}
    \vert u_1\vert 
    &\le 3L (c\eta)^{-\gamma_2} M, \\
    \vert u_2 \vert 
    &\le 7\pi  L ( c \eta )^{\gamma_1 }M,\\
    \vert u_3 \vert
    &\le 15\pi  L (\tfrac 5 \eta)^2  ( c \eta )^{\gamma_1 }M +  2 M_1,\\
    \vert w(k_n)\vert  
    &\le  7\pi  L (\tfrac 5 \eta)^{\vert n\vert-1 }  ( c \eta )^{\gamma_1 }M+2 M_n ,\\
    \int j(k_{\pm1}) 
    &\le  7\pi  L \tfrac {3d}{\beta \eta^2 } ( c \eta )^{\gamma_1 }M ,\\
    \int j(k_n) 
    &<  \tfrac 2 \beta \tfrac d{\eta^2 } ( 7\pi  L \tfrac 5 \eta ( c \eta )^{\gamma_1 }M +  2 M_n ).
\end{align*}
To estimate $u_1$ we use Lemma \ref{L_lem} to deduce 
\begin{align*}
    \vert u_1\vert &\le L\left( u_1(-d) +   \int a_1 u_2 +a_2 u_3 + \tfrac {1+s^2}{1+d^2} \exp(-\kappa_{k_0} (s-\tau+\tfrac 1 3 (s^3-\tau^3)))j(-d) \right) \\
    &\le 2L (c\eta)^{-\gamma_2} M + L \tfrac 1\eta (1+\eta^{-2} ) (7\pi  L ( c \eta )^{\gamma_1 }M +  2 M_1 )  \\
    &+ \tfrac L{1+d^2} \min ( \tfrac 1 {\kappa_{k_0}}, d^3 ) j(k_0,-d) \\
    &\le 2L (c\eta)^{-\gamma_2} M + L \tfrac 1\eta (1+\eta^{-2} )(7\pi  L ( c \eta )^{\gamma_1 }M   +  2 M_1 ) + L \tfrac c \beta (c\eta)^{-\gamma_2 }   M\\
    &< 3L (c\eta)^{-\gamma_2}M,
\end{align*}
where we used that $\tfrac {40} \eta M_1 \le \tfrac {1}{10}  M (c\eta)^{-\gamma_2 } $ since $\eta \ge \tfrac 1{10c} $ and that $ 7\pi  L c<\tfrac 1 2  $. We estimate $u_2$ by 
\begin{align*}
    \vert u_2 \vert &\le 2(c \eta)^{\gamma_1 }M +\int 2c\eta \tfrac 1 {1+s^2} u_1 + a(k_{\pm 2 }) w(k_{\pm 2 }) +  j(k_{\pm 1})  \\
    &\le 2 (c \eta)^{\gamma_1 }M +2\pi  c \eta \vert u_1\vert_{L^\infty_s} + \tfrac 4\eta \vert w(k_{\pm 2 }) \vert_{L^\infty_s} + \int j(k_{\pm 2 } ) \\
    &< 7\pi  L ( c \eta )^{\gamma_1 }M. 
\end{align*}
In order to control $u_3$, we integrate $\partial_s u_3 $ in time, which yields  
\begin{align*}
    \vert u_3 \vert &\le \vert u_3\vert (-d)  + \int  a(k_{\pm 2 }) w(k_{\pm 2 }) +  j(k_{\pm 1}) \\
    &\le 2 M_1 + \tfrac 5 \eta ( 7\pi  L \tfrac 4 \eta ( c \eta )^{\gamma_1 }M +  2 M_2 )+ 7\pi  L \tfrac {3d}{\beta \eta^2 } ( c \eta )^{\gamma_1 }M \\
    &< 8\pi  L (\tfrac 5 \eta)^2  ( c \eta )^{\gamma_1 }M +  2 M_1.
\end{align*}
For $w(k_n)$ we first consider $\vert n \vert \ge 3 $. By integrating $\partial_s w(k_n)$ we deduce 
\begin{align*}
    \vert w(k_n)\vert  &\le \vert w(k_n,-d)\vert +  \tfrac 2 \eta ( \vert w(k_{n+1})\vert_{L^\infty_s}  + \vert w(k_{n-1})\vert_{L^\infty_s}  )+\int j(k_n) \\
    &<  7\pi  L (\tfrac 5 \eta)^{\vert n\vert-1 }  ( c \eta )^{\gamma_1 }M+2 M_n.
\end{align*}
For the cases $n=\pm 2 $ we similarly conclude that 
\begin{align*}
    \vert w(k_{\pm2 })\vert  &\le\vert w(k_{\pm2 },-d)\vert+ \tfrac 2 \eta ( \vert u_2 \vert_{L^\infty_s} + \vert u_3 \vert_{L^\infty_s}   + \vert w(k_{\pm3 })\vert_{L^\infty_s}  )\\
    &<   7\pi  L \tfrac 5 \eta ( c \eta )^{\gamma_1 }M +  2 M_{\pm 2 } .
\end{align*}
For the estimates on the current $j$ we argue similarly as on the interval $I_1$ and introduce $\hat s$ as the shifted time coordinates. To estimate $j(k_{\pm1})$ we integrate $\partial_s j(k_{\pm 1})$ in time:
\begin{align*}
    j(k_{\pm1})&=2\exp( -\tfrac 12 \kappa \xi \eta (s+d))j(k_{\pm 1 } ,-d)  \\
    &+\tfrac {\kappa_{k_{\pm1}}} {\beta } \int \tfrac {1+\hat s^2} {1+\hat \tau^2} \exp(- \kappa_{k_{\pm1}}(\hat s-\hat \tau +\tfrac 1 3 (\hat s^3 -\hat \tau^3 )))(u_2\pm u_3 ).
\end{align*}
The impact of $j(k_{\pm 1})$ is bounded by  
\begin{align*}
    \int j(k_{\pm1}) &\le \tfrac {4} {\kappa \xi \eta } j(k_{\pm 1 }, -d)+ \tfrac {2 d}  {\beta \eta^2 } (\vert u_2 \vert_{L^\infty_s}  + \vert u_3 \vert_{L^\infty_s} )  \\
    &\le 7\pi  L \tfrac {3d}{\beta \eta^2 } ( c \eta )^{\gamma_1 }M
\end{align*}
and hence yields the estimate 
\begin{align*}
    j (k_{\pm1 }) &\le 2\exp(-\tfrac 1 2 d \kappa \xi \eta ) j(k_{\pm1 },-d )+ \tfrac 4 {\beta \eta^2  } (\vert u_3 \vert_{L^\infty_s} +\vert u_2 \vert_{L^\infty_s} ) \\
    &<\tfrac 4 { \beta \eta^2 } (7\pi  L ( c \eta )^{\gamma_1 }M +  2 M_1  ). 
\end{align*}
 By integrating we thus obtain the following estimate for $j(k_{n })$:
\begin{align*}
    j(k_{n})&=2\exp( -\kappa \xi \eta (s+d) )j(k_n,-d )  \\
    &+\tfrac {\kappa_{k_n}} {\beta }  \int \tfrac {1+\hat s^2} {1+\hat \tau^2} \exp(- \kappa_{k_n}(\hat s-\hat \tau +\tfrac 1 3 (\hat s^3 -\hat \tau^3 )))w(k_n) ,
\end{align*}
which leads to 
\begin{align*}
    \int j(k_n) &= \tfrac 1 {\kappa \xi \eta } j(k_n,-d)+\tfrac 1 \beta \tfrac d{\eta^2   } \vert w (k_n) \vert_{L^\infty_s} \\
    &\le  \tfrac 2 \beta \tfrac d{\eta^2 } ( 7\pi  L \tfrac 4 \eta ( c \eta )^{\gamma_1 }M +  2 M_n )
\end{align*}
and 
\begin{align*}
    j (k_n) &\le 2\exp(-\kappa \xi d \eta ) j(k_n,-d )+ \tfrac 4 {\beta \eta^2 } \vert w(k_n)\vert_{L^\infty_s} \\
    &<\tfrac 4 {\beta \eta^2 }  (7\pi  L (\tfrac 5 \eta)^{\vert n\vert-1 }  ( c \eta )^{\gamma_1 }M+2 M_n ).
\end{align*}
We estimate  $j(k_0)$ by integrating
\begin{align*}
    j(k_0) &=\tfrac {1+s^2}{1+d^2 } \exp( -\kappa_{k_0} (s+d+\tfrac 1 3 (s^3 +d^3 ))j(k_0,-d)\\
    &+\tfrac {\kappa_{k_0}}\beta  \int \tfrac {1+s^2}{1+\tau^2 } \exp( -\kappa_{k_0} (s-\tau+\tfrac 1 3 (s^3 -\tau ^3 ))) u_1(\tau ).
\end{align*}
The second term can be estimated by
\begin{align*}
    \tfrac {\kappa_{k_0}}\beta\int &\tfrac {1+s^2}{1+\tau^2 } \exp( -\kappa_{k_0} (s-\tau+\tfrac 1 3 (s^3 -\tau ^3 ))) u_1(\tau )\\
    &\le \tfrac {1}\beta \min ( \kappa_{k_0} \pi d^2 , 1) \vert u_1\vert_{L^\infty_s} 
\end{align*}
and thus 
\begin{align*}
    j(k_0,d) &=\exp( -\kappa_{k_0} (2d+\tfrac 2 3 d^3)j(k_0,-d)\\
    &+\tfrac {1}\beta \min ( \kappa_{k_0} \pi d^2 , 1) \vert u_1\vert_{L^\infty_s} \\
    &\le \exp( -\kappa_{k_0} (2d+\tfrac 2 3 d^3)\tfrac {c^2}\beta M(c\eta)^{-\gamma_2 }\min ( \kappa_{k_0}, 1 )\\
    &+\tfrac {3L M}\beta \min ( \kappa_{k_0} \pi d^2 , 1)  (c\eta)^{-\gamma_2}\\
    &< \tfrac {4L M}\beta \min ( \kappa_{k_0} \pi d^2 , 1)  (c\eta)^{-\gamma_2}.
\end{align*}
\end{proof}

\subsection{The Echo and Lower Bounds in the Interval $I_2 $} 
\label{EchoI2}
In this section we establish the echo mechanism on the interval $I_2$, i.e. our aim is to show that the mode $u_1 $ induces growth of the $u_2$ mode. For this echo mechanism we need the additional assumption
\begin{align}
    \kappa k_0^2\min(\beta,1) >\tfrac 1 { c }. \label{kapk_est}
\end{align}
As shown in Subsection \ref{sec:betainf}, this  is not only a technical assumption. When $k_0$ is too small, the $u_1$ term can become negative due to the action of $j$ and hence negate the growth of $u_2$ and we could even obtain $u_2(d)\approx 0$. We will use initial data of the form
\begin{align}
\begin{split}
    u_1(-d)&=1, \\
    u_2(-d) &\le 50  c^2\eta , \\
    \vert j\vert(k_0,-d) &\le  \tfrac {2 c^2}\beta  , \\
    \vert w\vert (k,-d),\vert u_3\vert(-d)   &\le 5 (c\eta)^{\gamma_2 }, \\
    \vert j \vert ( k,-d) &\le \tfrac {20}\eta  (c\eta)^{\gamma_2 } . \label{init_Echo_I2}
\end{split}
\end{align}
Which corresponds to the echoes on $I_1$ normalized in terms of $u(-d)$. We will prove that $u$ closely matches the following asymptotics:
\begin{align*}
    \tilde u_1 &= \exp( -\tfrac 1 \beta (\tan^{-1}(s)+\tan^{-1}(d))), \\
    \tilde u_2 &= u_2(-d)+  2c \eta \beta ( 1-\exp(-\tfrac 1 \beta (\tan^{-1}(s)+\tan^{-1}(d)))).
\end{align*}
\begin{pro} \label{echoI2}
Consider a solution of \eqref{Echo} with initial data \eqref{init_Echo_I2}, then the following estimates hold:
\begin{align}
\begin{split}
    \vert u_1(d)-\tilde u_1 (d)\vert &= 12\pi c,\\
    \vert u_2(d) - \tilde u_2 (d)\vert &\le 24\pi c^2 \eta ,\\
    w(k_n,d),u_3(d)&\le 6(c\eta)^{\gamma_2} ,\\
    j(k_n,d)&\le \tfrac {25} {\beta \eta^2 } (c\eta)^{\gamma_2 },\\
     j(k_0,d)&\le \tfrac 2 \beta .\label{I2Echo} 
\end{split}
\end{align}
\end{pro}
In the following it is convenient to introduce the \emph{good unknown}: 
\begin{align*}
    g(s) &= (1+s^2)j-\tfrac {u_1}\beta, 
\end{align*}
In terms of $g$ our equations then read
\begin{align*}
    \partial_s u_1 &=-\tfrac 1 \beta \tfrac 1 {1+s^2} u_1  -a_1u_2 + a_2 u_3-\tfrac 1 {1+s^2} g
\end{align*}
and
\begin{align*}
    \partial_s g &= 2sj(k_0)+(1+s^2) \partial_s j(k_0) -\tfrac 1\beta \partial_s u_1 \\
    &= \tfrac {2s} {1+s^2} (g+\tfrac 1 \beta u_1 )\\
    &-\kappa_{k_0}  (1+s^2) g +\tfrac {2s}{1+s^2} (g+\tfrac 1 \beta u_1 )\\
    &+ \tfrac 1 {\beta^2} \tfrac 1 {1+s^2} u_1  +\tfrac 1 \beta a_1u_2 - \tfrac 1 \beta  a_2 u_3+\tfrac 1 \beta \tfrac 1 {1+s^2} g  \\
    &= (\tfrac {4s+\frac 1\beta  }{1+s^2}- \kappa_{k_0} (1+s^2) ) g  \\
    &+ \tfrac 1 \beta \tfrac {4s+\frac 1\beta }{1+s^2}u_1 +\tfrac 1 \beta a_1u_2 - \tfrac 1 \beta  a_2 u_3.
\end{align*}
Therefore, \eqref{Echo} can be equivalently expressed as
\begin{align}
\begin{split}
    \partial_s \left( 
    \begin{array}{c}
         u_1  \\
         u_2 \\
         u_3\\
         g 
    \end{array}
    \right) &= \left( 
    \begin{array}{cccc}
         -\tfrac 1 \beta \tfrac 1 {1+s^2} &-a_1  & a_2 &-\tfrac 1 {1+s^2}  \\
         2 c\eta\tfrac 1 {1+s^2} &0&0&0  \\
                  0&0&0&0\\
         \tfrac{1}\beta\tfrac {4s+\frac 1\beta }{1+s^2}  &\tfrac 1 \beta a_1&\tfrac 1 \beta a_2 &\tfrac {4s+\frac 1\beta}{1+s^2} -\kappa_k (1+s^2)
    \end{array}
    \right)\left( 
    \begin{array}{c}
         u_1  \\
         u_2 \\
         u_3\\
         g 
    \end{array}
    \right)\\
    &+ \left( 
    \begin{array}{c}
         0  \\
         a(k\pm2)w(k\pm 2) -j(k\pm 1) \\
         \pm a(k\pm2)w(k\pm 2) \mp j(k\pm 1)\\
         0
    \end{array}
    \right) .\label{Echo_g}
\end{split}
\end{align}
The homogeneous system with respect to \eqref{Echo_g}  is given by 
\begin{align}
    \begin{split}
        \partial_s \left( \begin{array}{c}
             \tilde u_1\\\tilde u_2 
        \end{array}\right)&= \left(\begin{array}{cc}
             -\tfrac1\beta \tfrac 1 {1+s^2}& 0 \\
             2c\eta \tfrac 1 {1+s^2} &0 
        \end{array}\right)\left( \begin{array}{c}
             \tilde u_1\\\tilde u_2 
        \end{array}\right)
    \end{split}\label{eq:I2_hom_ech}
\end{align}
with the explicit solution
\begin{align}
\begin{split}
    \tilde u_1 &= \exp(-\tfrac 1 \beta (\tan^{-1}(s)+\tan^{-1}(d)) ) u_1(-d)\\
    \tilde u_2 &= u_2(-d) + 2 c \eta \beta ( 1-\exp(-\tfrac 1 \beta (\tan^{-1}(s)+\tan^{-1}(d))))u_1(-d).\label{I2homlsg}
\end{split}
\end{align}
In the following, we prove that the solution of \eqref{Echo_g} can be treated as a perturbation of \eqref{I2homlsg}. Note that we can approximate
\begin{align*}
    \beta ( 1-\exp(-\tfrac 1 \beta (\tan^{-1}(s)+\tan^{-1}(d))))\approx \min(\beta , (\tan^{-1}(s)+\tan^{-1}(d))), 
\end{align*}
where ``$\approx$'' in this case corresponds to the explicit bounds  
\begin{align*}
    \tfrac 1 2 \min(\beta , \cdot)&\le \beta ( 1-\exp(-\tfrac 1 \beta \cdot )) \le \min(\beta , \cdot). 
\end{align*}

\begin{proof}[Proof of Proposition \ref{echoI2}]

We want to show by a bootstrap that
\begin{align}
\begin{split}
    \vert u_1-\tilde u_1 \vert &\le c_1=12\pi c \\
    \vert u_2-\tilde u_2 \vert &\le c_2 = (2\pi+1)  c \eta c_1  \\
    \vert u_3\vert, \vert w(k_n)\vert &\le 6(c\eta)^{\gamma_2 } \\
    \int j(k_n)
    &\le \tfrac {13d}{\beta \eta^2} (c\eta)^{\gamma_2 }\\
    \int j(s, k_{\pm1 } )&\le \tfrac {10\pi} {\beta \eta } .\label{I2boot} 
\end{split}
\end{align}
Let $s^\ast $ be the maximal time such that \eqref{I2boot} holds. We assume that $s^\ast \le d$ and show that this leads to a contradiction by improving \eqref{I2boot}. The estimates of $j(k_n)$ for $n\neq 0$ are done similarly as in Proposition \ref{I2_pro} and we hence omit them here. First, we estimate $g$:
\begin{align*}
    g_0(s)&=\tfrac {(1+s^2)^2}{(1+d^2)^2}\exp(\tfrac 1 \beta (\tan^{-1}(s)+\tan^{-1}(d))-\kappa_{k_0}(s+d+\tfrac 1 3 (s^3 +d^3 )))g (-d),\\
    g(s)-g_0(s) &=\tfrac 1 \beta  \int \tfrac {(1+s^2)^2}{(1+\tau^2)^2}\exp \left(\tfrac 1 \beta (\tan^{-1}(s)-\tan^{-1} (\tau))-\kappa_{k_0}(s-\tau+\tfrac 1 3 (s^3-\tau^3))\right) \\
    &\qquad \left( \tfrac {4\tau+\frac 1\beta  }{1+\tau ^2} u_1(\tau)+\tfrac 1 \beta  a_2 u_2(\tau) -\tfrac 1 \beta a_3 u_3(\tau) \right)\ d\tau.
\end{align*}
Next we estimate the size of the perturbations by
\begin{align*}
    \int d\tau_2 \  &\tfrac{\exp(-\tfrac 1 \beta (\tan^{-1}(s)-\tan^{-1}(\tau_2)))} {1+\tau_2^2}(g-g_0) (\tau_2)\\
    &=\tfrac 1 \beta \int d\tau_2  \int d\tau_1\  \exp\left(-\tfrac 1 \beta (\tan^{-1}(s)-\tan^{-1} (\tau_1))-\kappa_{k_0}(\tau_2-\tau_1+\tfrac 1 3 (\tau_2^3-\tau_1^3))\right)\\
    &\qquad \cdot \tfrac {1+\tau_2^2}{(1+\tau_1^2)^2} \left(\tfrac {4\tau_1+\frac 1\beta  }{1+\tau_1^2} u_1(\tau_1)+\tfrac 1 \beta  a_2 u_2(\tau_1) -\tfrac 1 \beta a_3 u_3(\tau_1)\right)\\
&\le \tfrac 2 {\beta\kappa_{k_0} }  \int d\tau_1 \    \exp(-\tfrac 1 \beta (\tan^{-1}(s)-\tan^{-1}(\tau_1))) \\
& \quad \cdot \tfrac 1 {(1+\tau_1^2)^2} \left(\tfrac {4\tau_1+\frac 1\beta  }{1+\tau^2_1 } u_1(\tau_1)+\tfrac 1 \beta  a_2 u_2(\tau_1) -\tfrac 1 \beta a_3 u_3(\tau_1) \right)\\
&\le\tfrac 2 {\kappa_{k_0} \beta  } (2\vert u_1\vert_{L^\infty_s} + \tfrac {4c}{\eta  }(\vert u_2\vert_{L^\infty_s} +\vert u_3\vert_{L^\infty_s} ) )
\end{align*}
and 
\begin{align*}
    \int &\tfrac{\exp(-\tfrac 1 \beta (\tan^{-1}(s)-\tan^{-1}(\tau_2)))} {1+\tau_2^2}g_0 (\tau_2) \ d\tau_2 \\
    &= c^4 \exp(-\tfrac 1 \beta (\tan^{-1}(s)+\tan^{-1}(d))) g (-d)\\
    &\qquad \cdot\int (1+\tau_2^2)\exp(-\kappa_{k_0}(\tau_2+d+\tfrac 1 3 (\tau_2 ^3 +d^3 )))g_0 (-d) \\
    &=\tfrac {c^4}{\kappa_{k_0}} \exp(-\tfrac 1 \beta (\tan^{-1}(s)+\tan^{-1}(d))) \tfrac 3 \beta g(-d) \le \tfrac 3\beta\tfrac {c^4}{\kappa_{k_0} },  
\end{align*}
where we used \eqref{init_Echo_I2} to obtain $g(-d) \le \tfrac 3 \beta $. To estimate $u_1$ we look at the difference to the homogeneous system,
\begin{align*}
    \partial_s (u_1-\tilde u_1 ) &= -\tfrac 1 \beta \tfrac 1 {1+s^2} (u_1-\tilde u_1 ) -a_1 u_2 +a_3 u_3 -\tfrac 1 {1+s^2} g 
\end{align*}
which leads after integrating to 
\begin{align*}
    \vert u_1 -\tilde u_1 \vert &\le \tfrac 4 \eta (\vert u_2 \vert_{L^\infty_s} +\vert u_3 \vert_{L^\infty_s}) + \tfrac 1 {\kappa_{k_0} } (\tfrac 2 \beta \vert u_1\vert_{L^\infty_s} + \tfrac {4c}{\eta \beta }(\vert u_2\vert_{L^\infty_s} +\vert u_3\vert_{L^\infty_s} )) + \tfrac 3 \beta \tfrac {c^4}{ \kappa_{k_0}} \\
    &\le 8 c  \min(\beta , \pi )+ 4(2\pi+1)  c  c_1  +\tfrac 4 {\kappa_{k_0}\beta }  (1+c_1)+\tfrac 5 \eta 6(c\eta)^{\gamma_2 }  <c_1
\end{align*}
since $\kappa k_0^2 \ge \tfrac 1{ \beta c } $. We estimate  $u_2-\tilde u_2 $ by 
\begin{align*}
    \partial_s (u_2-\tilde u_2 ) &= 2c\eta \tfrac 1 {1+s^2} (u_1-\tilde u_1) +a(k_{\pm2 }) w(k_{\pm 2}) -j(k_{\pm 1})
\end{align*}
which implies by integrating in $s$, that
\begin{align*}
    \vert u_2 -\tilde u_2 \vert &\le 2\pi c \eta c_1 +\tfrac 2 \eta \vert w(k_{\pm 2})\vert_{L^\infty_s} +\int j (k_{\pm 1})\\
    &\le 2\pi c \eta c_1 + (12+10\pi\tfrac 1 \beta ) c^{\gamma_2 }\eta^{-\gamma_1 }\\
    &< (2\pi+1)  c \eta c_1.
\end{align*}
Next we estimate $w(k_n)$ for $\vert n \vert \ge 3$. We remark that the estimates for  $u_3 $ and $w_{k_{\pm 2}}$ are similar and hence we omit them. By integrating over the derivative we deduce 
\begin{align*}
   w(k_n,-d)&\le  w(k_n,d) +\tfrac 2 \eta (\vert w(k_{n+1})\vert_{L^\infty_s} +\vert w(k_{n-1} )\vert_{L^\infty_s} ) +\int j(k_n) \\
   &< 6(c\eta)^{\gamma_2 }.
\end{align*}
So the bootstrap is concluded. It is left to estimate $j(k_0)$.  We write
\begin{align*}
    \partial_s j(k_0) &= \tfrac {\kappa_{k_0}}\beta u_1 + (\tfrac {2s}{1+s^2}-\kappa_{k_0} (1+s^2))j(k_0)\\
    &\le \tfrac {\kappa_{k_0}}\beta u_1 -\tfrac 8{9} \kappa_{k_0} j(k_0)
\end{align*}
where in the second line we used \eqref{kapk_est}. By integrating, we obtain 
\begin{align*}
    j(k_0,s) &\le \exp(-\tfrac 89 \kappa_{k_0}(s+d) )j(k_0,-d)\\
    &+ \tfrac {\kappa_{k_0}}\beta \int_{-d}^s \ d \tau  \exp(-\tfrac 89 \kappa_{k_0}(s-\tau ) )    u_1(\tau)
\end{align*}
which leads to 
\begin{align*}
     j(k_0,d) &\le \exp(-2 d \tfrac 89 \kappa_{k_0} )j(k_0,-d)\\
    &+ \tfrac 98  \tfrac {1}\beta \vert u_1\vert_{L^\infty_s} \\
    &\le \tfrac 2 \beta .
\end{align*}

\end{proof}

\subsection {Proof of Theorem \ref{Thm_I} }\label{I3}
 In Subsections \ref{sect:I13}, \ref{SecI2} and  \ref{EchoI2} we proved lower and upper bounds until the time $s=d$. Furthermore, in Subsection \ref{sect:I13} we already showed the asymptotic behavior on the interval $I_3$. In this subsection we need to combine the results of these subsections to obtain the final lower and upper bounds for the complete interval $I^k$. This will be achieved in two steps: first we conclude the bootstrap on $I_3$, afterwards we show that all terms result in the desired estimates.

\begin{proof}[Proof of Theorem \ref{Thm_I}]
Following we proceed similarly as in the proof of Proposition \ref{I1pro}, just for $I_3$. In particular we use the tools from Subsection \ref{sect:I13}. We thus need to prove the missing estimate on $[d,s_1] $. Let $r_i(d)$ be the initial data of $r(s)$. We define the $c_i$ terms by
\begin{align}
\begin{split}
    c_1 &= 2 (r_1(d)+(21c^2 + 2 \tfrac{c^3}\beta (c \eta )^{\gamma } ) r_2(s_0)+N+N_j) \\
    \tilde c_1&=0 \\
    c_2 &=  \tfrac 1 {1-2\frac c\beta } ( r_2(d) + N + N_j )  \\
    \tilde c_2 &=22 c_1  + \tfrac c\beta  \tilde c_2.\label{I3_c} 
\end{split}
\end{align}
and
\begin{align*}
   N&= 2c\tfrac 1 {\kappa \xi \eta^{\gamma_2 } }j(k_0\pm 1,d )+2c u_3 (d) \\
    &+2 \sum_{\vert m\vert \ge 2 } (2c)^{\vert m \vert } (w+\tfrac 4{\kappa \eta \xi }j)(k_m,d)\\
    N_j &= 4 (c\eta)^{\gamma_2}\tfrac {c^{2}}{\kappa_{k_0} }j(k_0,d).
\end{align*}

We prove by bootstrap that 
\begin{align}
    \vert u\vert (s) &\le S^\ast(s) C(s). \label{SCest}
\end{align}
Since $c_i\ge r_i(d)$ this estimate holds locally, and we again let $s^\ast$ be the maximal time such that  \eqref{SCest} holds. We assume that $s^\ast \le s_1 $ and improve the estimate, which gives a contradiction and thus proves that \eqref{SCest} holds on $[d,s_1]$. For the $R_i$ we obtain with the Lemmas \ref{R3mode}, \ref{Rpsi}, \ref{Rrest} that
\begin{align*}
     R_1[F_{all}  ]&= R_1[F_{3mode}]+R_1[F_{j}]+ R_1 [F_{\tilde w }]+R_1[F_{j(k_0\pm 1)}]+R_1[F_{u_3}]\\
     &\le 20 c^2c_1+20 c^4\tilde c_2  +(20c^2+c^4(c \eta)^{\gamma} )c_2\\
    &+\tfrac {c^3}\beta (c_1+\tilde c_2 )+  \tfrac {c^3}\beta  (c \eta)^{\gamma}  c_2 + (c\eta)^{\gamma_1 } \tfrac {4 c^{4 }}{\kappa_{k_0} }j(k_0,d) \\
     &+ 2c^2 (\tilde w(2)+\tilde w(-2))\\
     &+\tfrac {2c}{\kappa \xi \eta }  j(k_{\pm 1},  d )  +  \tfrac {2c} {\beta \kappa \xi  } (\tilde w(1) + c_1^\ast + c_2^\ast )\\
     &+2c\tilde w(1) \\
     &\le 21 c^2  c_1+ 2 \tfrac {c^3 }\beta \tilde c_2 + (21c^2+ \tfrac {c^3}\beta  (c \eta)^{\gamma})   c_2  +N+ c^2 (c\eta)^{\gamma }N_j,
\end{align*}
and 
\begin{align*}
    R_2[F_{all} ]&= R_2[F_{3mode}]+R_2[F_{j}]+ R_2[F_{\tilde w }]+ R_2[F_{j(k_0\pm 1)}]+R_2[F_{u_3}]\\
     &\le 20(\tfrac s \eta )^\gamma (c_1+ 2  c^2  \tilde c_2) + 20c^2c_2   \\
    &+  \tfrac c \beta (\tfrac {s}  \eta )^\gamma  (c_1+\tilde c_2 )  +\tfrac c \beta   c_2 + (c\eta)^{\gamma_2}\tfrac {c^{2}}{\kappa_{k_0} }j(k_0,d))\\
     &+ c^2 (\tilde w(2)+\tilde w(-2))(\tfrac {s}\eta)^{\gamma_1 } \\
     &+\tfrac {2c} {\kappa \xi \eta }  j(k_{\pm 1}, d )(\tfrac {s}\eta)^{\gamma_1 } + \tfrac c {\beta \kappa \xi  } (\tilde w(1) + c_1^\ast + c_2^\ast ) (\tfrac {s}\eta)^{\gamma_1 }   \\
     &+2c\tilde w(1)(\tfrac {s}\eta)^{\gamma_1 } \\
     &\le 21 (\tfrac s \eta)^\gamma c_1 + 2 \tfrac c \beta  c_2+ \tfrac c \beta (\tfrac s \eta )^\gamma \tilde c_2 +N (\tfrac s \eta)^\gamma + N_j.
\end{align*}

Therefore, we deduce that
\begin{align*}
    r_1(s_0) + R_1[all][1] &\le r_1(d)+21  c^2 ( c_1+\tilde c_2) + (21 c^2+ \tfrac {c^3}\beta  (c \eta)^{\gamma})   c_2  \\
    &\quad +N+ c^2 (c\eta)^\gamma N_j <c_1,\\
    r_2(s_0) + R_2[all][1] &\le r_2(d)   + 2\tfrac c \beta  c_2+ N+N_j <c_2,\\
    R_2[all][(\tfrac s \eta)^\gamma ] &<20 c_1  + 2c^2 \tilde c_2  <\tilde c_2.
\end{align*}

This concludes the bootstrap and we estimated $\vert u\vert(s)\le S^\ast(s)C(s)$ for $s\le s_1 $. To finish the proof of the theorem we need to establish the norm estimate at the final time. With Proposition \ref{I2_pro} we obtain the folloiwng bounds:
\begin{align*}
    \vert u_1\vert (d)
    &\le 3 (c\eta)^{-\gamma_2}LM,\\
    \vert u_2 \vert (d)
    &\le 7\pi  ( c \eta )^{\gamma_1 } LM,\\
    \vert u_3 \vert(d)
    &\le   7\pi   (\tfrac 5 \eta)^2  ( c \eta )^{\gamma_1 }LM +  2M_1,\\
    \vert w(k_n,d)\vert  
    &\le  7\pi  (\tfrac 5 \eta)^{\vert n\vert-1 }  ( c \eta )^{\gamma_1 }L M+2 M_n, \\
        j (k_n,d) 
    &\le \tfrac 4 { \eta^2 \beta } (7\pi (\tfrac 4 \eta)^{\vert n\vert-1 }  ( c \eta )^{\gamma_1 } L M+2 M_n ),\\
        j(k_0,d)
    &\le  \tfrac {4L M}\beta \min ( \kappa_{k_0} \pi d^2 , 1)  (c\eta)^{-\gamma_2}.
\end{align*}
This in turn yields
 \begin{align*}
     N&= 2c\tfrac 1 {\kappa \xi \eta^{\gamma_2 } }j(k_0\pm 1,d )+2c u_3 (d) \\
     &+2 \sum_{\vert m\vert \ge 2 } (2c)^{\vert m-k_0\vert } (w+\tfrac8{\kappa \eta \xi n^2 }j)(k_m,d) \\
     &\le c    (c\eta)^\gamma LM,   \\
     N_j &\le 4 (c\eta)^{\gamma_2}\tfrac {c^{2}}{\kappa_{k_0} }  \tfrac {4L M}\beta \min ( \kappa_{k_0} \pi d^2 , 1)  (c\eta)^{-\gamma_2}  \\
     &\le \tfrac {16}\beta \min(\pi, \tfrac {c^2} {\kappa_{k_0}} )LM .
 \end{align*}
 Using these bounds, we consider
Afterwards, we estimate 
\begin{align*}
    r(d) &= S^{-1}(d) u(d)\\\
    &= - 2c \gamma^{-1} \left( 
    \begin{array}{cc}
         -\tfrac {\gamma_2 }{2c }\vert c \eta\vert ^{-\gamma_2+1 }& -\vert c\eta \vert^{-\gamma_2 }\\
         \tfrac {\gamma_1 }{2c }  \vert c \eta\vert^{-\gamma_1+1 }&\vert c \eta\vert^{-\gamma_1 }
    \end{array}
    \right)u( d)\\
   \left( 
   \begin{array}{cc}
        \vert r_1 \vert \\
        \vert r_2 \vert 
   \end{array}
   \right)  (d)  &\le LM \left( 
   \begin{array}{cc}
        15\pi c (c\eta )^\gamma   \\
        4
   \end{array}
   \right) 
\end{align*}
and hence deduce that
\begin{align*}
    c_1 &=  (c\eta)^\gamma (30 \pi c  +30\tfrac {c^2} \beta +c ) LM\\
    &\le 31 \pi  c   (c\eta)^\gamma LM\\
    c_2 &= (5+ \tfrac {16 \pi}\beta )LM .
\end{align*}
This implies the estimate 
\begin{align*}
    u(s_1) &\le S^\ast(s_1) C(s_1)\\
    &\le LM\left( \begin{array}{cc}
         \tfrac 1 2 &1  \\
         \tfrac 1 {2c} & 2c  
    \end{array}\right) \left(\begin{array}{cc}
         31 \pi  c  (c\eta)^\gamma     \\
         (5+ \tfrac {16 \pi}\beta )
    \end{array}\right)\\
    &\le  L  M ( c \eta)^{\gamma}  \left( \begin{array}{c}
        16c + (5+ \tfrac {16 \pi}\beta ) (c\eta)^{-\gamma}   \\16\pi  
    \end{array}\right),
\end{align*}
where we used that $(c\eta)^{-\gamma} \tfrac 1 {\kappa_{k_0}}= (\tfrac {k_0^2}{c\xi})^\gamma \tfrac 1 {\kappa k_0^2} =\tfrac 1 {(c\xi )^\gamma \kappa k_0^{2\gamma_2 }}\ll \beta c$. For $\tilde w(n)$ we obtain  
\begin{align*}
    \tilde w(n)&= 2 \sum_{\vert m\vert \ge 2 } (2c)^{\vert m-n\vert+\chi} (w+    \tfrac4{\kappa \eta \xi }  j)(k_m, d)  \\
    &+ (2c)^{\vert \vert n\vert -2 \vert}c (c_1^\ast + \tfrac 1 {c^2}c_2^\ast ) \\
    &+( 2c)^{\vert c\vert -1 }u_3( d) \\
    &\le L (2c)^{\vert n\vert } M + M_n \\
    u_3 (n) &\le L  (2c)^{\vert n\vert+2  } M + M_1.
\end{align*}
Furthermore, by integrating over $\partial_s j(k_n)$ we obtain 
\begin{align*}
    j(k_n,s_1)&\le L \tfrac 5 {\kappa \xi \eta } ((2c)^{\vert n\vert } M + M_n),\\
    j(k_{\pm 1 } ,s_1)&\le L \tfrac 5 {\kappa \xi \eta } ((2c)^{\vert n\vert +2} M + M_1).
\end{align*}
In order to estimate $j(k_0) $ we use Lemma \ref{Rpsi}:
\begin{align*}
    \vert j(k_0, s_1 )\vert &\le L  c^2  \eta^2 \exp(-\kappa_{k_0}\eta^3) j (k_0 ,d)\\
    &+ 2 \tfrac {16^2}\beta\tfrac 1 {  \eta^2 }(c_1 +c_2 + \tilde c_1+  \tilde c_2 )\\
    &\le   3 \pi   \tfrac {\kappa_{k_0} } \beta\eta^2 \exp(-\kappa_{k_0}\eta^3)    (\tfrac d \eta)^{\gamma_2} M\\
    &+  L 4 \tfrac {16^2}\beta\tfrac 1 {  \eta^2 } 2M (c \eta)^{-\gamma_2} \\
    &\le \tfrac {2^{11}}\beta\tfrac 1 {  \eta^2 } M (c \eta)^{-\gamma_2}.
\end{align*}
We further estimate
\begin{align*}
    M^2&= \sum_{m,n\ge 1} 10^{-m-n}(w+\tfrac 1 {\alpha_{k_n}}j) (k_n)(w+\tfrac 1 {\alpha_{k_m }}j) (k_m)\\
    &\le \tfrac 2 {1-10^{-1}} \sum_{n\ge 1 } 10^{-n} (w^2+\tfrac 1 {\alpha_{k_n}^2}j^2) (k_n)\\
    &\le\tfrac 2 {1-10^{-1}} \tfrac 1 {\lambda_{k_0}} \sum_{n\ge 1 } (10^{-n}\tfrac {\lambda_{k_0}}{\lambda_{k_n}})\lambda_{k_n} (w^2+\tfrac 1 {\alpha_{k_n}^2}j^2) (k_n)\\ 
    &\le \tfrac 2 {1-10^{-1}}\tfrac 1 {\lambda_{k_0}} \Vert w, j\Vert_X (s_0)^2\\
     \sum_{\vert n\vert \ge 1 }\lambda_{ {k_n}} M_n^2 &=\sum_n \lambda_{k_n} \sum_{m,l}10^{-\vert m-l\vert -\vert l-n \vert - \chi_l -\chi_m } (w+\tfrac 1 {\alpha_{k_m }}j)^2(k_m)\\
     &\le \tfrac 2 {1-10^{-1}} \sum_n \lambda_{k_n}\sum_m 10^{-\vert m-n \vert - \chi_m }(w^2+\tfrac 1 {\alpha_{k_m }^2}j^2)(k_m) \\
     &\le  \tfrac 2 {1-10^{-1}} \sum_m \lambda_{k_m}(w+\tfrac 1 {\alpha_{k_m }}j)^2(k_m) \sum_n 10^{-\vert m-n\vert-\chi_m } \tfrac {\lambda_{k_n}}{\lambda_{k_m}}\\
     &\le \tfrac {2\hat \lambda^2} {1-10^{-1}}\Vert w,j \Vert_X^2.
\end{align*}
Combining these bounds we infer the norm estimate
\begin{align*}
    \Vert w,j\Vert_X^2 (s_1)&\le 16 \pi L^2 M^2 (c \eta )^{2\gamma} (\lambda_{k_{\pm 1}   }   +\lambda_{k_0}(16\pi +5(c\eta)^{-2\gamma})^2  )   \\
    &+\sum_{\vert n\vert \ge 1 }L^2\lambda_{ {k_n}}(10^{-\vert n\vert } M  + M_n )^2\\
    &= M^2(c\eta )^{-2\gamma} (\lambda_{k_{\pm 1}^2   } (2c)^2  +\lambda_{k_0}^2 +2\sum_{\vert n\vert \ge 1 }\lambda_{ {k_n}}^2 10^{- 2 \vert n\vert }+L^2\sum_{\vert n\vert \ge 1 }\lambda_{ {k_n}}  M_n^2\\
    &\le L^2( \hat \lambda (16\pi)^2+ 3\hat \lambda^2 )  (c \eta )^{2\gamma}   \Vert w,j\Vert_X^2 (s_0).   
\end{align*}
This finally allows us to complete the proof of the upper bound and obtain that
\begin{align*}
    \Vert w,j\Vert_X (s_1)&\le 18\pi L \hat \lambda (c\eta)^\gamma  \Vert w,j\Vert_X (s_0).
\end{align*}
To prove the lower bound we use Proposition \ref{I1pro} and Proposition \ref{echoI2} and obtain that at time $s=d$ it holds that 
\begin{align*}
    \vert u_1(d)-\exp(-\tfrac \pi \beta ) (c\eta)^{\gamma_2} \vert &=O(c) \\
    \vert u_2(d) -2\beta(1-\exp(-\tfrac \pi \beta ) )(c\eta )^{\gamma_1 } \vert &\le O(c) \\
    w(k_n,d),u_3(d)&\le 6\\
    j(k_n,d)&\le \tfrac {10\pi}{\beta \eta } \tfrac 1 \eta \\
     j(k_0,d)&\le \tfrac 2 \beta.
\end{align*}
We calculate $\tilde u_2 $ by 
\begin{align*}
    \tilde u_2(s_1) &= (0 \ 1 ) S(s_1) S^{-1}(d) u(d)\\
    &\approx  (\tfrac{1}{2c}  \ 2c)
 2c  \left( 
    \begin{array}{cc}
         -c(c\eta) ^{\gamma_1}& -(c\eta )^{-\gamma_2 }\\
         \tfrac {1}{2c } (c\eta)^{\gamma_2}&(c\eta)^{- \gamma_1 }
    \end{array}
    \right) u (d) \\
    &\approx (-c(c\eta)^{\gamma_1 }+2c(c\eta)^{\gamma_2 }) u_1(d) + (-(c\eta )^{-\gamma_2 }+4c^2(c\eta) ^{-\gamma_1 } )u_2(d) \\
    &\approx c(c\eta)^{\gamma_1}u_1(d) + (c\eta)^{\gamma_2} u_2(d)\\
    &\approx2(c\eta)^{\gamma_1 } \beta (1-\exp(-\tfrac \pi \beta )) u_1(-d) \\
    &\approx2(c\eta)^{\gamma } \beta (1-\exp(-\tfrac \pi \beta )) u_1(s_0).
\end{align*}
The difference $u_2-\tilde u_2$ is estimated by 
\begin{align*}
    \vert u_2-\tilde u_2 \vert &\le (0 \ 1 ) S^\ast(s_1) R[F]\\
    &\le \tfrac 1 {2c} R_1[F] +2c R_2[F]\\
    &\le(c\eta)^{\gamma_2 } u_2(d) +  O(c)\\
    &=   2(c\eta)^{\gamma } \beta (1-\exp(-\tfrac \pi \beta )) u_1(s_0) +    O(c).
\end{align*}
Furthermore, we obtain 
\begin{align*}
    M&\le \tfrac 1 {1-10^{-1} } u_1( \tilde s_0)\\
    M_n&\le \tfrac 4 {1-10^{-1} } u_1( \tilde s_0).
\end{align*}
So we finally obtain since $\beta\ge \tfrac 1 5 $
\begin{align*}
    w(k_{-1}, t_{k_{-1} } ) &\approx2(c\eta)^{\gamma } \beta (1-\exp(-\tfrac \pi \beta )) u_1(s_0) \\
    &\ge \tfrac 1 2 \max_l (w(k_l,t_{k_l}),
\end{align*}
which gives 
\begin{align*}
    w(k_{-1}, t_{k_{-1} } ) & \ge
     (c\eta)^{\gamma } \min(\beta , \pi) w(k_{-1}, t_{k_{-1} } ).
\end{align*}

\end{proof}

In this article we have studied the asymptotic (in)stability of the magnetohydrodynamic equations with a shear, a constant magnetic field and magnetic dissipation. Here multiple effects compete to determine the long time behavior of solutions:
\begin{itemize}
    \item Echoes in the inviscid fluid equations may lead to large norm inflation.
    \item The underlying magnetic field leads to an exchange between kinetic and magnetic energy. In particular, for large magnetic fields oscillation my diminish norm inflation.
    \item Magnetic dissipation may stabilize the flow. Hence, a priori, it is not clear whether stability requires Gevrey regularity (as for the Euler equations) or Sobolev regularity (as for the fully dissipative problem) and how the evolution depends on the size of the magnetic field $\alpha$ and on the resistivity $\kappa$.
\end{itemize}
As the main result of this article we show that the balance between these effects is parametrized by the parameter $\beta=\tfrac{\kappa}{\alpha^2}>0$ and that the behavior for finite, positive $\beta$ strongly differs from both the fully non-dissipative case and the large dissipation limit (which reduces to the Euler equations). In particular, we show that in this regime the magnetic dissipation is not strong enough to stabilize the evolution in Sobolev regularity and establish Gevrey regularity as optimal both in terms of upper and lower bounds.
It remains an interesting problem for future research to determine the optimal stability classes for other partial dissipation regimes and to study the inviscid limit $\kappa \downarrow 0$.

\appendix 
\section{Estimating the Growth Factor} \label{LApp}
In Section \ref{SecI2} we observe the evolution of \eqref{Echo} on the interval $I_2=[-d,d]$. Here we observe the interaction between $ j $ and $u_1$ 
\begin{align}
\begin{split}
    \partial_s u_1 &= - j\\
    \partial_s j &= \tfrac K \beta u_1 +(\tfrac {2s} {1+s^2} - K (1+s^2) )j,\label{PDEI2b}
\end{split}
\end{align}
with $\kappa_k$ replaced by $K$ for simplicity
In particular we  bound the growth of  $u_1$ by a factor.  Let $U(\tau,s)$ be the solution of \eqref{PDEI2b} with initial data $u_1(\tau)=1$ and $ j(\tau)=0 $. We show that
\begin{itemize}
    \item $\vert U(\tau,s)\vert \le 1 $ for $\beta \ge \tfrac \pi 2 $ 
    \item $\vert U(\tau,s )\vert \le L=L (\beta, K)  $ for $\beta <\tfrac \pi 2 $.
\end{itemize}

With the restriction
\begin{align}
    c&\le (8\pi)^{\frac 4 3 } \beta^{\frac {16} 3 }. \label{cb}
\end{align}
we obtain 
\begin{align}
    L(\beta,K ) &= \left\{
    \begin{array}{cc}
         1 & 1\le K  \\
         \sqrt d &\tfrac 1  2 c^{\frac 34}\le K \le 1  \\
         2(1+\tfrac \pi \beta ) & \tfrac {2 \pi} \beta c^3\le K \le \tfrac 1  2 c^{\frac 34} \\
         1 & K\le \tfrac {2 \pi} \beta c^3
    \end{array}
    \right. \label{Lest} 
\end{align}
We note that \eqref{cb} is not optimal, in the sense that Section \ref{SecI2} we need $Lc<<1$ and we could optimize the $\tfrac 12 c^{\frac 34}$ term to obtain a larger $L$ but better \eqref{cb}. However, this would yield a lot dependencies which would make the final theorem more technical to state. The most important part of this estimates is to verify that $\beta$ can be very small if $c $ is chosen small enough. First we do an energy estimate, let
\begin{align*}
    E&= u_1^2 + \tfrac \beta K  j 
\end{align*}
which leads to 
\begin{align*}
    \tfrac 1 2 \partial_s E &\le (\tfrac {2s}{1+s^2} - K(1+s^2) )_+  E.  
\end{align*}
Therefore, we obtain for $K\ge 1 $ that $\partial_s E\le 0$, which proves our first estimate. Furthermore, we infer for $K\le 1$
\begin{align*}
    E(s) &\le E( \tau )\left\{ \begin{array}{cc}
         1 & s\le 0  \\
         (1+s^2)^2 & 0\le s \le (\tfrac K2 )^{\frac -1 3  } \\
         4(K)^{-\frac 4 3 } &(\tfrac K2)^{-\frac 1 3  } \le s,
    \end{array} \right. 
\end{align*}
We conclude 
\begin{align*}
       u_1(s) &\le \left\{ \begin{array}{cc}
         1 & s\le 0  \\
         1+s^2 & 0\le s \le (\tfrac K2)^{\frac 1 3  } \\
         2(K) ^{-\frac 2 3 } &(\tfrac K 2 )^{\frac 1 3  } \le s 
    \end{array} \right.  
\end{align*}
which proves \eqref{Lest}  for $\tfrac 1  2 c^{\frac 34}\le K \le 1 $. For small $K$ we need to make a different ansatz. We write $ j$ as, 
\begin{align*}
         j(s) &= \tfrac K \beta \int_{-d}^s \tfrac {1+s^2}{1+\tau^2} \exp(-K (s-\tau+\tfrac 1 3 ( s^3 -\tau^3)))u(\tau) \ d\tau
\end{align*}
    and so
\begin{align*}
        u(s)-1&= - \tfrac K \beta \iint_{-d\le \tau_1\le \tau_2 \le s }d(\tau_1,\tau_2 ) \tfrac {1+\tau_2^2}{1+\tau_1 ^2} \exp(-K (\tau_2 -\tau_1+\tfrac 1 3 ( \tau_2^3 -\tau_1 ^3)))u(\tau_1 ) \\
        &=  - \tfrac 1 \beta \int_{-d\le \tau_1\le s }d\tau_1 \ u(\tau_1 )  \tfrac {1}{1+\tau_1 ^2} [\exp(-K (\tau_2 -\tau_1+\tfrac 1 3 ( \tau_2^3 -\tau_1 ^3)))]_{\tau_2=\tau_1}^{\tau_2=s}\\
        &=-  \tfrac 1 \beta \int_{-d\le \tau_1\le s }d\tau_1 \ u(\tau_1 )  \tfrac {1}{1+\tau_1 ^2} (1-\exp(-K (s -\tau_1+\tfrac 1 3 ( s^3 -\tau_1 ^3)))).
\end{align*}
    Now we exploit that $u$ is decreasing till the smallest time such that $u(s)=0$. This holds, since if $u $ is positive, then $ j$ is positive and so $\partial_s u =- j\le 0 $.  Therefore, we bound 
\begin{align*}
            \tfrac 1 \beta \int_{-d\le \tau_1\le s }d\tau_1 \ u(\tau_1 )  \tfrac {1}{1+\tau_1 ^2} (1-\exp(-K (s -\tau_1+\tfrac 1 3 ( s^3 -\tau_1 ^3))))
\end{align*}
        by $1$ to deduce $0\le u(s)\le 1 $. Let $s$ be positive, then we estimate 
\begin{align*}
    \tfrac 1 \beta  \int_{-d\le \tau_1\le s }d\tau_1   &\tfrac {1}{1+\tau_1 ^2} (1-\exp(-K (s -\tau_1+\tfrac 1 3 ( s^3 -\tau_1 ^3))))\\
    &=\tfrac 1 \beta  \int_{-d\le \tau_1\le -s }d\tau_1   \tfrac {1}{1+\tau_1 ^2} (1-\exp(-K (s -\tau_1+\tfrac 1 3 ( s^3 -\tau_1 ^3))))\\
    &+\tfrac 1 \beta  \int_{-s\le \tau_1\le s }d\tau_1   \tfrac {1}{1+\tau_1 ^2} (1-\exp(-K (s -\tau_1+\tfrac 1 3 ( s^3 -\tau_1 ^3))))\\
    &\le\tfrac 1 {\beta s }+ \tfrac \pi \beta ( 1-\exp(-K (2s+\tfrac 23s^3 ))) \\
        &\le \tfrac 1 {\beta s } + \tfrac \pi \beta K (2s+\tfrac 23s^3 )\\
        &\le \tfrac 1 {\beta s }+ \tfrac \pi \beta K s^3 . 
\end{align*}
This term is less than zero if $\tfrac 2 \beta \le s \le (\tfrac \beta {2\pi K })^{\frac 13 } $. We choose $s=\min( (\tfrac \beta {2\pi K })^{\frac 13 },d ) $ maximal. When $s=d$, then  $(\tfrac \beta {2\pi K })^{\frac 13 }\ge d $ which is satisfied if $K \le \tfrac {2\pi}\beta c^3$ and so we obtain the last estimate of \eqref{Lest}. Now we need to prove the case if  $\tfrac {2 \pi} \beta c^3 \le K \le \tfrac 1  2 c^{\frac 34}$, with the previous calculation we obtain for  $s_1 =(\tfrac \beta {2\pi K })^{\frac 13 }$, that $0\le u(s_1)\le 1 $. Then for $s\ge s_1 $ we have 
    
\begin{align*}
        u(s)-1  &= \tfrac 1 \beta \int_{-d\le \tau_1\le s }d\tau \ u(\tau )  \tfrac {1}{1+\tau^2} (1-\exp(-K (s -\tau+\tfrac 1 3 ( s^3 -\tau^3))))\\
        \vert u(s)-1 \vert &\le \tfrac 1 \beta \int_{-d\le \tau\le s_1 }d\tau \ \tfrac 1 {1+\tau^2}  + \tfrac 1 \beta \int_{t_1\le \tau_1\le s }d\tau \ u(\tau )  \tfrac {1}{1+\tau^2}\\
        &\le \tfrac {\pi  }{\beta }+\tfrac 1 {\beta s_1 }  \vert u\vert_{L^\infty_s}.
\end{align*}
Due to $K \le \tfrac 1 2 c^{\frac 4 3 } $ and \eqref{cb} we obtain $s_1 \beta = (\tfrac {\beta^4 }{2\pi K })^{\frac 1 3 }\ge 2  $  and so 
\begin{align*}
    \vert u (s)\vert &\le \tfrac 1 {1-\frac 1 {\beta s_1}} (1+\tfrac \pi \beta )     \\
    &\le 2(1+\tfrac \pi \beta ) .
\end{align*}

\section{Nonlinear Instability of Waves}
\label{sec:norminflation} 
In this appendix we consider the nonlinear instability of the traveling waves.
\begin{align}
\begin{split}\label{NL_eq}
    \partial_t w+ (v\nabla w)_{\neq}  &=  \alpha \partial_x j +(b\nabla j)_{\neq}  -(2c \sin(x)\partial_y\Delta_t^{-1} w)_{\neq} \\
    \partial_t j+ (v\nabla j)_{\neq} &=\kappa \Delta_t j   +\alpha \partial_x w+(b\nabla w)_{\neq} - 2 \partial_x \partial_y^t \Delta^{-1}_t j-(2(\partial_i v\nabla )\partial_i \Delta^{-1}j)_{\neq} ,
\end{split}
\end{align}

For brevity of notation let us denote the Gevrey $2$ norm with constant $C$ by 
\begin{align*}
    \|(w, j)\|_{\mathcal{G}_C}^2 = \int \sum_k \exp(C\sqrt{|\xi|}) |\mathcal{F}(w, j)|^2 d\xi.
\end{align*}
Then the norm inflation result of Theorem \ref{thm:main} further implies the nonlinear instability of any non-trivial traveling wave for $C$ sufficiently small.
\begin{cor}
  \label{cor:nl}
  Let $0<c<\min( 10^{-4},10^{-3}\tfrac \kappa{\alpha^2}) $ be given and consider a traveling wave as in Lemma \ref{lemma:waves} and let $0<C_2<C_*$ where $C_*=C_*(c)$ is as in Theorem \ref{thm:main}.
  Then the nonlinear evolution equations around the traveling wave are unstable for small initial data in
  $\mathcal{G}_{C_2}$ in the sense that for any $0<C_1<C_2$, $\epsilon>0$ and $N>1$ there
  exists initial data with
  \begin{align*}
    \|(w_0, j_0)\|_{\mathcal{G}_{C_2}}< \epsilon
  \end{align*}
  but such that for some time $T>0$ it holds that
  \begin{align*}
    \|(w, j)|_{t=T}\|_{\mathcal{G}_{C_2}}\geq N  \|(w_0, j_0)\|_{\mathcal{G}_{C_1}}.
  \end{align*}
\end{cor}

We stress that this results considers the instability of the traveling waves and that the space with respect to which instability is established depends on the size $c$ of the wave.
A nonlinear instability result for the underlying stationary state \eqref{eq:groundstate} in the spirit of \cite{dengmasmoudi2018,bedrossian2016nonlinear,dengZ2019}
further requires that the size $c$ of the traveling is comparable to $\epsilon$.

\begin{proof}[Proof of Corollary \ref{cor:nl}]
  We argue by contradiction. Thus suppose that the nonlinear solution is
  uniformly controlled in $\mathcal{G}_{C_1}$ for all times:
  \begin{align*}
    \sup_{t>0} \|(w, j)\|_{\mathcal{G}_{C_1}} \leq D \epsilon.
  \end{align*}
  for some constant $D>0$.
  Given this a priori control of regularity we may consider the nonlinear
  equations as a forced linear problem
  \begin{align*}
    \dt (w, j) + L (w, j) = F
  \end{align*}
  where $L$ is the linear operator considered throughout this article and $F$ is
  the quadratic nonlinearity.
  If we denote by $S(t,\tau)$ the solution operator associated to $L$ it then
  follows that for any $T>0$
  \begin{align*}
    (w, j)_{t=T} = S(T,0) (w_0, j_0) + \int_0^T S(T, \tau) F(\tau) d\tau.
  \end{align*}
  By the norm inflation results of Theorem \ref{thm:main} for any $C_2<C_*$ there exists
  initial data and a time $T>0$ such that
  \begin{align}
    \label{eq:norminf}
    \|S(T,0) (w_0, j_0)\|_{L^2} \geq N \|(w_0, j_0)\|_{\mathcal{G}_{C_2}}. 
  \end{align}
  Since this estimate is linear after multiplication with a factor we may assume
  that this initial data also has size smaller than $\epsilon$.
  On the other hand, by the results of Section \ref{sec:small} and of Theorem \ref{thm:main}
  for any fixed time $T$, $S(T, \tau)$ is uniformly bounded as a map from $L^2$
  to $L^2$.
  More precisely, we recall that $S(T, \tau)$ decouples with respect to the 
  frequency $\xi$ in $y$.
  \begin{itemize}
  \item   For $\xi$ with $|\xi|\gg T^2$ by the results of
  Section \ref{sec:small} the time interval $(0,T)$ is considered ``small time''
  and hence $S(T,\tau)$ is bounded uniformly.
\item If instead $|\xi|\leq T^2$ then Theorem \ref{thm:main} provides an upper bound of the operator norm by $\exp(C \sqrt{\xi})\leq \exp(C T)$.
\end{itemize}
  Thus there exists an extremely large constant $E$ (depending on $T$) such that
  \begin{align*}
    \| \int_0^T S(T, \tau) F(\tau) d\tau\|_{L^2} \leq E \int_0^T \|F(\tau)\|_{L^2}d\tau.
  \end{align*}
  Finally, we note that by assumption
  \begin{align*}
    \|F(\tau)\|_{L^2} \leq D^2 \epsilon^2. 
  \end{align*}
  Hence, choosing $\epsilon\ll \frac{1}{E D^2N T }$ the Duhamel integral can be
  treated as a perturbation of \eqref{eq:norminf}, which concludes the proof.  
\end{proof}

\section*{Acknowledgements}
Funded by the Deutsche Forschungsgemeinschaft (DFG, German Research Foundation) – Project-ID 258734477 – SFB 1173.
This article is part of the PhD thesis of Niklas Knobel written under the supervision of Christian Zillinger.

\bibliography{library}
\bibliographystyle{alpha}

\end{document}